\documentclass[11pt]{amsart}
\usepackage[foot]{amsaddr}
\usepackage[a4paper,margin=2.65cm]{geometry}
\usepackage[english]{babel}
\usepackage{csquotes}
\usepackage[T1]{fontenc}
\usepackage{textcomp}
\usepackage{amsmath,amssymb,amsthm,mathtools,mathrsfs}
\usepackage{newtxtext}
\usepackage{newtxmath}
\usepackage[bb=boondox,cal=boondoxo,scr=boondoxo]{mathalfa}
\usepackage{enumitem}
\usepackage[hidelinks]{hyperref}
\hypersetup{
	pdftitle={Hypergeometric Mixed-Type Multiple Orthogonal Polynomials},
	pdfauthor={Manuel Ma\~nas},
	pdfsubject={Mixed-type multiple orthogonality and hypergeometric systems},
	pdfkeywords={mixed-type multiple orthogonal polynomials; hypergeometric functions; Mellin transforms; step-line recurrences; Christoffel transformations; bidiagonal factorization}
}
\usepackage{microtype}
\usepackage{bigints}
\usepackage[renew-dots,renew-matrix]{nicematrix}
\mathtoolsset{showonlyrefs}
\allowdisplaybreaks

\newcommand{\dx}{\,\mathrm{d}x}
\newcommand{\dt}{\,\mathrm{d}t}

\newcommand{\e}{\mathrm{e}}
\newcommand{\M}{\mathcal{M}}
\newcommand{\one}{\mathbf{1}}
\newcommand{\pFq}[5]{\;{}_{#1}F_{#2}\left(\begin{matrix}#3\\#4\end{matrix};#5\right)}
\newcommand{\Jmat}{\boldsymbol{\mu}^{\,\mathrm J}}
\newcommand{\Lagmat}{\boldsymbol{\mu}^{\,\mathrm L}}
\newcommand{\thetaop}{\mathscr D}
\newcommand{\N}{\mathbb N}
\theoremstyle{plain}
\newtheorem{theorem}{Theorem}[section]
\newtheorem{proposition}[theorem]{Proposition}
\newtheorem{lemma}[theorem]{Lemma}
\newtheorem{corollary}[theorem]{Corollary}

\newtheoremstyle{definitionstyle}
{6pt}
{12pt}
{\normalfont}
{}
{\bfseries}
{.}
{0.5em}
{}

\theoremstyle{definitionstyle}
\newtheorem{definition}[theorem]{Definition}
\newtheorem{example}[theorem]{Example}

\newtheoremstyle{remarkstyle}
{6pt}
{12pt}
{\normalfont}
{}
{\itshape}
{.}
{0.5em}
{}

\theoremstyle{remarkstyle}
\newtheorem{remark}[theorem]{Remark}

\title{Hypergeometric Mixed-Type Multiple Orthogonal Polynomials}
\author{Manuel Ma\~nas}
\address{Department of Theoretical Physics, Faculty of Physical Sciences,
	Complutense University of Madrid, 28040 Madrid, Spain}
\email{manuel.manas@ucm.es}

\begin{document}
	
\begin{abstract}
	We construct two hypergeometric families of mixed-type multiple orthogonal
	forms for rank-one \(q\times p\) matrices of weights, with arbitrary \(p\)
	and \(q\): Jacobi and Laguerre I systems. In the admissible
	near-diagonal range, both normalized mixed forms are obtained explicitly.
	The components on the power vector are terminating generalized
	hypergeometric polynomials. The components on the hypergeometric vector are
	finite sums of terminating generalized hypergeometric polynomials; in the
	Jacobi case they can be reorganized as finite combinations of
	terminating Kamp\'e de F\'eriet polynomials evaluated at \((x,1)\), whereas
	in the mixed beta--Euler Laguerre I case a finite triangular system links
	the residues at the finite poles with the terms generated by the Euler
	operator. Gamma-quotient Mellin formulas, Meijer \(G\)-representations, and
	Rodrigues formulas are obtained for the complete mixed forms.

	On the step-line, the two biorthogonal systems satisfy dual recurrences
	governed by matrices with \(p\) subdiagonals and \(q\) superdiagonals. Every
	recurrence coefficient is reduced to a finite Gamma--Pochhammer expression.
	Under the stated normality and nonzero-pivot hypotheses, Christoffel
	transformations and Gauss--Borel factorization give bidiagonal
	factorizations of the recurrence matrices. The lower factors have closed
	Pochhammer formulas, while the upper factors are represented by finite
	Christoffel tau-determinants and, equivalently, by cross-ratios of shifted
	moment minors. Both Christoffel chains close explicitly in the mixed
	Pi\~neiro specialization.
\end{abstract}
	\keywords{mixed-type multiple orthogonal polynomials; Jacobi weights;
		Laguerre I weights; generalized hypergeometric functions; Kamp\'e de
		F\'eriet functions; Mellin transforms; Meijer \(G\)-functions; step-line
		recurrences; Christoffel transformations; bidiagonal factorization}
	\subjclass[2020]{Primary 33C45; Secondary 42C05, 33C20, 15A23, 47B36}
	\maketitle
	
	\tableofcontents
	\section{Introduction}
	\label{sec:introduction}
	
	Multiple orthogonal polynomials and multiple orthogonal forms extend the
	classical theory to systems involving several weights or moment functionals.
	Besides their intrinsic approximation-theoretic interest, they appear in random
	matrices, nonintersecting paths, Hermite--Pad\'e approximation, integrable
	systems, spectral theory of banded operators, and stochastic processes.
	Step-line recurrences and their associated banded matrices provide the
	analogue of the Jacobi matrix for multiple orthogonality.
	
	In the scalar theory, orthogonal polynomials satisfy a three-term recurrence
	relation governed by a tridiagonal Jacobi matrix. For multiple orthogonality,
	the recurrence matrix is banded. Its structure depends on
	the geometry of the underlying multi-indices and on the interaction between
	the different weights.  On the step-line, these recurrences give rise to
	banded matrices whose spectral theory is closely connected with matrix-valued
	moments, generalized Favard theorems, and block or rectangular Weyl
	functions.
	
	In the mixed-type multiple orthogonality considered here, one side of the
	orthogonality structure is given by a vector of powers,
	\[
	\mathbf u(x)
	=
	\left[
	x^{\theta_1},\ldots,x^{\theta_p}
	\right],
	\]
	while the other side is given by a vector of weights
	\[
	\mathbf v(x)
	=
	\left[
	v_1(x),\ldots,v_q(x)
	\right].
	\]
	The resulting matrix of measures has density
	\[
	\mathrm d\boldsymbol\mu(x)
	=
	\mathbf v(x)^{\mathsf T}\mathbf u(x)\,\mathrm dx,
	\]
	so it is a rectangular \(q\times p\) matrix of pointwise rank one, with
	\(pq\) scalar weights. It leads to
	rectangular moment matrices and rectangular Weyl functions.  The
	mixed setting contains, as particular cases, several well-known families of
	multiple orthogonal systems, including Jacobi--Pi\~neiro and multiple
	Laguerre systems.

	Mixed-type multiple orthogonal polynomials for arbitrary numbers \(p\) and
	\(q\) of weights were introduced by Daems and
	Kuijlaars~\cite{DaemsKuijlaars2007}. Their Gaussian system, arising from
	nonintersecting Brownian motions with several initial and final points, is a
	fundamental rank-one example. A non-Gaussian explicit
	mixed construction associated with modified Bessel functions was later
	obtained by Zhang~\cite{Zhang2016BesselMixed}, in the restricted
	\(2\times2\) setting.
	A different Bessel family of mixed type is constructed
	in~\cite{ManasBesselMixed2026} from a matrix weight on the unit circle,
	which is generically not of rank one. Its relation to the Jacobi family is
	established through a scaled limit of Markov--Stieltjes transforms, rather
	than a finite limit of the interval measures themselves.

	The mixed Jacobi and Laguerre I families considered here extend the
	Jacobi-like and Laguerre-like systems of Wolfs~\cite{Wolfs2024}, recovered
	when \(p=1\) and the power exponent is zero. We reserve the suffix
	`-like' for these ordinary \(p=1\) systems. Their
	mixed extension and the study of the corresponding bidiagonal factorizations
	were also motivated by the preceding work on mixed Pi\~neiro orthogonality by
	Branquinho, D{\'\i}az, Foulqui\'e-Moreno, and
	Ma\~nas~\cite{PineiroMixed2026}. In that power-vector setting, the mixed
	forms, their hypergeometric representations, and the complete bidiagonal
	factorization of the step-line recurrence matrix can all be made explicit.
	In the mixed Jacobi system, the right-hand power vector is replaced by
	the Jacobi hypergeometric weights arising from Wolfs's construction,
	while the boundary degeneration
	\(\boldsymbol a,\boldsymbol b\to\boldsymbol\beta\) recovers exactly the mixed Pi\~neiro
	matrix. The Pi\~neiro system~\cite{PineiroMixed2026} therefore provides the
	power-vector boundary specialization of the Jacobi mixed construction.
	
	The present construction gives two explicit non-Gaussian hypergeometric
	families of mixed-type systems on the positive real line for arbitrary
	\(p\) and \(q\):
	Jacobi systems and Laguerre I systems.  Their components are written
	in terms of terminating generalized hypergeometric functions together with
	Gamma-quotient moment formulas. In particular, the number of terminating
	hypergeometric blocks in the components on the hypergeometric vector grows
	with the total degree. In the Jacobi case, this finite reconstruction
	can also be grouped into a finite sum
	of terminating Kamp\'e de F\'eriet polynomials evaluated at
	\((x,1)\)~\cite{SrivastavaKarlsson1985}. The analogous two-variable
	regrouping is not uniform in the mixed beta--Euler Laguerre I family.
	Instead, a finite triangular linear system links the residues at finite
	poles with the terms generated by the Euler operator. At \(s=0\), the formulas
	reduce to the Jacobi Kamp\'e de F\'eriet formula; at \(r=0\), Newton
	interpolation gives terminating hypergeometric formulas for the components.
	Rodrigues-type representations for
	the complete mixed forms by means of compositions of differential operators
	acting on suitable seed weights yield explicit formulas for the mixed forms
	and their moments.
	
	On the step-line, the mixed forms constitute two biorthogonal
	sequences. Multiplication by the independent variable is represented by a
	\((p,q)\)-banded matrix. Every band entry is evaluated directly as a finite
	pairing of the explicit components. In the Jacobi case this gives finite
	Gamma--Pochhammer sums. In the Laguerre I case it gives finite
	Gamma--Pochhammer sums after the finite triangular system determining the
	components has been solved. Only the step-line sequence is used; no
	off-step-line recurrences between neighbouring multi-indices are required.
	
	Bidiagonal factorizations of these step-line recurrence matrices are also
	constructed. Such factorizations are fundamental in
	the theory of totally positive and oscillatory matrices and play a central
	role in spectral theory~\cite{BranquinhoFoulquieManas2023Spectral}.
	Under the normality and nonzero-pivot conditions stated below, Christoffel
	transformations and the Gauss--Borel factorization of the matrix of measures
	give lower bidiagonal factorizations for both mixed families. In the mixed
	Pi\~neiro specialization~\cite{PineiroMixed2026}, where both sides of the
	matrix of measures are
	power vectors, the factorization becomes completely explicit, including all
	upper bidiagonal factors.  Outside that specialization, the upper factors
	are characterized by finite Christoffel determinants and, equivalently, by
	cross-ratios of shifted moment minors. The general construction is verified
	in an exact rational \(p=q=2\) example.

	The one-weight reductions recover the known
	factorizations for the Jacobi--Pi\~neiro and multiple Laguerre systems of the
	first kind.

	The explicit forms are considered in the admissible near-diagonal range,
	under the parameter and nonresonance hypotheses stated below. There the
	required finite Mellin cancellations hold, the evaluated moment determinants
	supply normality, and the hypergeometric expressions terminate. The
	recurrence part uses only the single sequence of indices on the step-line.
	
	The paper is organized as follows.
	Section~\ref{sec:general-mixed-framework} introduces mixed orthogonality,
	rectangular moment matrices, normalization conventions, and step-line
	recurrences. The next two sections construct the Jacobi and
	Laguerre I systems and determine their normalized mixed forms in the
	admissible near-diagonal range. Their Mellin, Meijer \(G\), hypergeometric,
	and Rodrigues representations are then developed. The step-line
	specializations give the two biorthogonal sequences and explicit finite
	formulas for every band coefficient. The final section constructs the
	bidiagonal factorizations of the recurrence matrices, evaluates the lower
	factors, represents the upper factors by shifted moment determinants, and
	treats the mixed Pi\~neiro specialization and the exact \(p=q=2\) example.
	Endpoint and central confluences are not used in these constructions and are
	treated in~\cite{ManasAskeyJacobi2026}, where the Jacobi systems lead to
	Laguerre systems of the first and second kinds and to Hermite systems of
	mixed type. For the ordinary \(p=1\) families, a complementary discrete
	construction uses the Bernstein transform of the Jacobi-like weights to
	obtain Hahn-like systems and their Askey-type
	confluences~\cite{ManasAskeyHahn2026}.
		
\subsection{Hypergeometric series, Mellin transform and Meijer G-function notation}
\label{subsec:hypergeometric-notation}

The Pochhammer symbol is denoted by
\[
(a)_k
\coloneq
\frac{\Gamma(a+k)}{\Gamma(a)},
\qquad
k\in\mathbb N_0.
\]
For a parameter vector
\[
\boldsymbol a=(a_1,\ldots,a_r),
\]
the shorthand notation
\[
\Gamma(z\one_r+\boldsymbol a)
\coloneq
\prod_{\rho=1}^{r}\Gamma(z+a_\rho),
\qquad
(z\one_r+\boldsymbol a)_k
\coloneq
\prod_{\rho=1}^{r}(z+a_\rho)_k
\]
is used.
More generally, if
\[
\boldsymbol c=(c_1,\ldots,c_r),
\qquad
\boldsymbol \nu=(\nu_1,\ldots,\nu_r)\in\mathbb N_0^r,
\]
the notation
\[
(\boldsymbol c)_{\boldsymbol \nu}
\coloneq
\prod_{\rho=1}^{r}(c_\rho)_{\nu_\rho}
\]
is used.
Products over an empty parameter vector are understood to be equal to one.

For parameter vectors
\[
\boldsymbol a=(a_1,\ldots,a_p),
\qquad
\boldsymbol b=(b_1,\ldots,b_q),
\]
the generalized hypergeometric function is
\begin{equation}
	\label{eq:generalized-hypergeometric-definition}
	\pFq{p}{q}
	{\boldsymbol a}
	{\boldsymbol b}
	{z}
	\coloneq
	\sum_{k=0}^{\infty}
	\frac{(\boldsymbol a)_k}{(\boldsymbol b)_k}
	\frac{z^k}{k!},
\end{equation}
where the series is understood in its domain of convergence, or by analytic
continuation when appropriate.  The hypergeometric series occurring as
polynomial components below are terminating series.

The classical Kamp\'e de F\'eriet double hypergeometric series is also used;
see \cite[Chapter~1]{SrivastavaKarlsson1985}. If the parameter strings
\(\boldsymbol a,\boldsymbol b,\boldsymbol c,\boldsymbol d,\boldsymbol e,\boldsymbol f\) have respective lengths
\(p,q,k,\ell,m,n\), then
\begin{equation}
	\label{eq:standard-KdF-definition}
	F_{\ell:m;n}^{p:q;k}
	\left(
	\begin{matrix}
		\boldsymbol a:\boldsymbol b;\boldsymbol c\\
		\boldsymbol d:\boldsymbol e;\boldsymbol f
	\end{matrix}
	\middle|x,y
	\right)
	\coloneq
	\sum_{r,\lambda\ge0}
	\frac{
		(\boldsymbol a)_{r+\lambda}
		(\boldsymbol b)_r
		(\boldsymbol c)_\lambda
	}{
		(\boldsymbol d)_{r+\lambda}
		(\boldsymbol e)_r
		(\boldsymbol f)_\lambda
	}
	\frac{x^r}{r!}\frac{y^\lambda}{\lambda!}.
\end{equation}
Thus the first parameter pair is coupled to \(r+\lambda\), while the second
and third pairs depend only on \(r\) and \(\lambda\), respectively.  Empty
parameter strings are allowed.

The Mellin transform of a function \(f\) on \(\mathbb R_+\) is denoted by
\begin{equation}
	\label{eq:Mellin-transform-notation}
	\mathcal M[f](s)
	\coloneq
	\int_0^\infty x^{s-1}f(x)\dx,
	\qquad
	s\in\mathcal S_f,
\end{equation}
where \(\mathcal S_f\) denotes the fundamental strip of \(f\), that is, the
vertical strip of values of \(s\) for which the integral converges.  If \(f\)
is supported on \((0,1)\), the same notation is used with the integral
restricted to \((0,1)\). The elementary shift rule
\[
\mathcal M[x^\ell f](s)=\mathcal M[f](s+\ell),
\qquad
\ell\in\mathbb N_0,
\]
is used whenever both sides are defined.

For functions on \(\mathbb R_+\), the Mellin convolution is
\[
(f*g)(x)
\coloneq
\int_0^\infty
f(t)g\left(\frac{x}{t}\right)\frac{\mathrm{d}t}{t}.
\]
Whenever the integrals are convergent, the Mellin transform turns Mellin
convolution into multiplication:
\begin{equation}
	\label{eq:Mellin-convolution-product}
	\mathcal M[f*g](s)
	=
	\mathcal M[f](s)\mathcal M[g](s).
\end{equation}
These identities will also be used in the standard meromorphic-continuation
sense when both sides have meromorphic continuations.

The Meijer \(G\)-function is defined by the Mellin--Barnes integral
\begin{equation}
	\label{eq:Meijer-G-definition}
	G_{p,q}^{m,n}
	\left(
	x
	\left|
	\begin{matrix}
		a_1,\ldots,a_p\\
		b_1,\ldots,b_q
	\end{matrix}
	\right.
	\right)
	\coloneq
	\frac{1}{2\pi\mathrm{i}}
	\bigintsss_{L}
	\frac{
		\prod_{j=1}^{m}\Gamma(b_j+s)
		\prod_{j=1}^{n}\Gamma(1-a_j-s)
	}{
		\prod_{j=m+1}^{q}\Gamma(1-b_j-s)
		\prod_{j=n+1}^{p}\Gamma(a_j+s)
	}
	x^{-s}\,\mathrm{d}s,
\end{equation}
where the contour \(L\) separates the poles of the gamma factors
\(\Gamma(b_j+s)\), \(j\in\{1,\ldots,m\}\), from those of
\(\Gamma(1-a_j-s)\), \(j\in\{1,\ldots,n\}\).
With this convention,
\begin{equation}
	\label{eq:Meijer-G-Mellin-transform}
	\int_0^\infty
	x^{z-1}
	G_{p,q}^{m,n}
	\left(
	x
	\left|
	\begin{matrix}
		a_1,\ldots,a_p\\
		b_1,\ldots,b_q
	\end{matrix}
	\right.
	\right)
	\dx
	=
	\frac{
		\prod_{j=1}^{m}\Gamma(b_j+z)
		\prod_{j=1}^{n}\Gamma(1-a_j-z)
	}{
		\prod_{j=m+1}^{q}\Gamma(1-b_j-z)
		\prod_{j=n+1}^{p}\Gamma(a_j+z)
	},
\end{equation}
in the corresponding fundamental strip. Standard contour and convergence
conditions are given in the NIST Digital Library of Mathematical
Functions~\cite{DLMFMeijerG}.

\section{General mixed orthogonality and the step-line}
\label{sec:general-mixed-framework}

Before introducing the two mixed systems, we fix the notation used throughout
the paper. In this section, \(x\) denotes the variable of
the mixed system and \(\mathrm d\nu(x)\) denotes the corresponding scalar
measure on an interval \(\Delta\subset\mathbb R\).

We consider a \(p\)-component power vector
\[
\mathbf u(x)
=
\left[
x^{\theta_1},\ldots,x^{\theta_p}
\right]
\]
and a \(q\)-component weight vector
\[
\mathbf v(x)
=
\left[
v_1(x),\ldots,v_q(x)
\right].
\]
In the Jacobi and Laguerre I settings, the parameters
\(\boldsymbol\theta=(\theta_1,\dots,\theta_p)\in\mathbb R^p\) are denoted by
\(\boldsymbol\alpha\) and \(\boldsymbol\beta\), respectively.

We adopt throughout the \(q\times p\) matrix-of-measures convention
\begin{equation}
	\label{eq:general-matrix-of-measures}
		\mathrm d\boldsymbol\mu(x)
		\coloneq
		\left[
		v_j(x)x^{\theta_i}
		\right]_{
			\substack{j\in\{1,\ldots,q\}\\ i\in\{1,\ldots,p\}}}
		\mathrm d\nu(x)
		=
		\begin{bNiceMatrix}[margin=2pt]
			v_1(x)x^{\theta_1}
			&
			\Cdots[shorten-end=2pt]
			&
			v_1(x)x^{\theta_p}
			\\
			\Vdots[shorten-end=2pt]
			&
			\Ddots[shorten-end=2pt]
			&
			\Vdots[shorten-end=2pt]
			\\
			v_q(x)x^{\theta_1}
			&
			\Cdots[shorten-end=2pt]
			&
			v_q(x)x^{\theta_p}
		\end{bNiceMatrix}
		\mathrm d\nu(x).
\end{equation}
Thus the rows correspond to the weight vector and the columns to the power
vector.  Whenever convenient, we use the bilinear pairing
\[
\left\langle f,g\right\rangle_{\nu}
\coloneq
\int_{\Delta} f(x)g(x)\,\mathrm d\nu(x).
\]
The convention in \eqref{eq:general-matrix-of-measures} is fixed from the
outset because it is the one used below for the moment matrix, the
Gauss--Borel factorization, and the left and right Christoffel
transformations.

\subsection{Mixed forms and biorthogonality}

For \(N\in\mathbb Z\), we denote by \(\mathbb P_N\) the vector space of
complex polynomials of degree at most \(N\), with the convention
\(\mathbb P_N=\{0\}\) for \(N<0\). In this paper $\mathbb N\coloneq \{1,2,\dots\}$ and $\mathbb N_0\coloneq \{0,1,2,\dots\}$.

Let
\[
\boldsymbol n=(n_1,\ldots,n_p)\in\mathbb N_0^p,
\qquad
\boldsymbol m=(m_1,\ldots,m_q)\in\mathbb N_0^q.
\]
When \(|\boldsymbol n|=|\boldsymbol m|+1\), a mixed form expanded in the power vector is an
expression
\begin{equation}
	\label{eq:gen-A-form}
	\mathcal A_{\boldsymbol n,\boldsymbol m}(z)
	\coloneq
	\sum_{i=1}^{p} A^{(i)}_{\boldsymbol n,\boldsymbol m}(z)z^{\theta_i},
	\qquad
	A^{(i)}_{\boldsymbol n,\boldsymbol m}\in\mathbb P_{n_i-1},
\end{equation}
satisfying
\begin{equation}
	\label{eq:gen-A-orthogonality}
	\left\langle z^k\mathcal A_{\boldsymbol n,\boldsymbol m},v_j\right\rangle_{\nu}=0,
	\qquad
	j\in\{1,\ldots,q\},
	\quad
	k\in\{0,\ldots,m_j-1\}.
\end{equation}
When \(|\boldsymbol m|=|\boldsymbol n|+1\), a mixed form expanded in the weight vector is an
expression
\begin{equation}
	\label{eq:gen-B-form}
	\mathcal B_{\boldsymbol n,\boldsymbol m}(z)
	\coloneq
	\sum_{j=1}^{q} B^{(j)}_{\boldsymbol n,\boldsymbol m}(z)v_j(z),
	\qquad
	B^{(j)}_{\boldsymbol n,\boldsymbol m}\in\mathbb P_{m_j-1},
\end{equation}
satisfying
\begin{equation}
	\label{eq:gen-B-orthogonality}
	\left\langle z^k\mathcal B_{\boldsymbol n,\boldsymbol m},z^{\theta_i}\right\rangle_{\nu}=0,
	\qquad
	i\in\{1,\ldots,p\},
	\quad
	k\in\{0,\ldots,n_i-1\}.
\end{equation}

\begin{definition}[Normality and perfectness]
	\label{def:general-normality}
	Following the terminology for multiple orthogonality used by Van Assche in
	Chapter~23 of Ismail's monograph~\cite{Ismail2005}, an index pair
	\((\boldsymbol n,\boldsymbol m)\) satisfying \(|\boldsymbol n|=|\boldsymbol m|+1\) is called
	\emph{normal for the \(\mathcal A\)-problem} if the conditions
	\eqref{eq:gen-A-form}--\eqref{eq:gen-A-orthogonality} determine
	\(\mathcal A_{\boldsymbol n,\boldsymbol m}\) uniquely up to multiplication by a nonzero
	constant.  Likewise, an index pair satisfying
	\(|\boldsymbol m|=|\boldsymbol n|+1\) is called \emph{normal for the
		\(\mathcal B\)-problem} if the conditions
	\eqref{eq:gen-B-form}--\eqref{eq:gen-B-orthogonality} determine
	\(\mathcal B_{\boldsymbol n,\boldsymbol m}\) uniquely up to multiplication by a nonzero
	constant.  The mixed system is called \emph{perfect} when all index pairs
	in both problems are normal in this sense.
\end{definition}

\begin{remark}
	The preceding terminology is not Mahler's stronger convention, in which
	normality also incorporates simultaneous attainment of all componentwise
	degree bounds.  Whenever the nonvanishing of a particular leading
	coefficient is needed below, it is stated separately.
\end{remark}

\begin{definition}[Normalization convention for mixed forms]
	\label{def:normalization-convention}
	Assume that the index pair under consideration is normal for the
	corresponding mixed problem.  Then the space of solutions of the
	orthogonality conditions is one-dimensional.  By a \emph{normalized}
	\(\mathcal A\)-form or a \emph{normalized} \(\mathcal B\)-form we mean the
	unique representative of this one-dimensional solution space selected by
	the scalar normalization condition explicitly stated in the corresponding
	setting.

	This normalization is not an \(L^2\)-normalization.  It fixes the remaining
	scalar freedom in a normal mixed form.  Depending on the context, this scalar
	may be fixed by an adjacent mixed pairing, by the leading coefficient
	specified on the step-line, by a Mellin moment identity, or by a
	contour/residue identity.  Thus the polynomial spaces in
	\eqref{eq:gen-A-form} and
	\eqref{eq:gen-B-form} determine the allowed components, while the displayed
	normalization formula fixes the representative.  Once this condition is
	imposed, the notation \(\mathcal A_{\boldsymbol n,\boldsymbol m}\) or
	\(\mathcal B_{\boldsymbol n,\boldsymbol m}\) always refers to the normalized form.
\end{definition}

Whenever the index pairs involved are normal, the two forms satisfy the
fundamental vanishing relations
\begin{equation}
	\label{eq:gen-biorthogonality-vanishing}
	\left\langle
	\mathcal A_{\boldsymbol n,\boldsymbol m},
	\mathcal B_{\boldsymbol\nu,\boldsymbol\mu}
	\right\rangle_{\nu}=0
	\quad\text{if}\quad
	\begin{cases}
		n_i\le \nu_i,&i\in\{1,\ldots,p\},\\
		\text{or}\\
		m_j\ge \mu_j,&j\in\{1,\ldots,q\}.
	\end{cases}
\end{equation}

The adjacent nonzero pairings may be fixed by a choice of normalization.  On
the step-line below this choice produces a biorthonormal pair of sequences.

We restrict both families of mixed forms to the step-line. This gives two
biorthogonal sequences indexed by a single
integer \(N\), and multiplication by the independent variable is represented
by a banded matrix.  No recurrence between arbitrary neighbouring
multi-indices is introduced.

\subsection{Step-line recurrences and the banded matrix}
\label{subsec:general-step-line-reduction}

The step-line chooses one multi-index of each total degree.  Along
this ordered sequence, multiplication by the variable is represented on the
\(B\)- and \(A\)-sides by a band matrix and its transpose.

For \(d\in\mathbb N\), the step-line in \(\mathbb N_0^d\) consists of the
multi-indices whose components differ by at most one and are written in
nonincreasing order.  More explicitly, if
\[
N=dm+j,
\qquad
m\in\mathbb N_0,
\quad
j\in\{0,\ldots,d-1\},
\]
then the step-line multi-index of norm \(N\) is
\begin{equation}
	\label{eq:gen-step-line-index}
	\boldsymbol\sigma_d(N)
	=
	\boldsymbol\sigma_d(dm+j)
	\coloneq
	(\underbrace{m+1,\ldots,m+1}_{j\text{ entries}},
	\underbrace{m,\ldots,m}_{d-j\text{ entries}}).
\end{equation}
Thus, for example,
\[
\boldsymbol\sigma_d(0)=(0,\ldots,0),
\qquad
\boldsymbol\sigma_d(1)=(1,0,\ldots,0),
\qquad
\boldsymbol\sigma_d(d)=(1,\ldots,1).
\]
For later use, we denote the quotient and the remainder in the Euclidean
division of \(N\) by \(d\) by
\[
N=d\rho_d(N)+\ell_d(N),
\qquad
\rho_d(N)=m,
\quad
\ell_d(N)=j.
\]
It follows directly that
\begin{equation}
	\label{eq:gen-step-coordinate}
	\boldsymbol\sigma_d(N+1)-\boldsymbol\sigma_d(N)
	=
	\boldsymbol e_{\ell_d(N)+1}^{(d)}.
\end{equation}
Exactly one coordinate increases at each step: from
\(N\) to \(N+1\) it is coordinate \(\ell_d(N)+1\).  These coordinates occur
in the cyclic order
\[
1,2,\ldots,d,1,2,\ldots.
\]
This is the convention used below when the general rectangular chains are
restricted to the step-line.

The step-line mixed forms are
\begin{equation}
	\label{eq:gen-step-forms}
	\mathcal B_N
	\coloneq
	\mathcal B_{\boldsymbol\sigma_p(N),\,\boldsymbol\sigma_q(N+1)},
	\qquad
	\mathcal A_N
	\coloneq
	\mathcal A_{\boldsymbol\sigma_p(N+1),\,\boldsymbol\sigma_q(N)},
	\qquad N\in\mathbb N_0.
\end{equation}
We require all scalar leading minors of the moment matrix to be nonzero.
This ensures uniqueness of the two forms, their nonzero adjacent pairings,
and the active leading coefficient used in the normalization below, so that
the Gauss--Borel factors defining \(T\) exist at every step. Uniqueness
alone in the sense of Definition~\ref{def:general-normality} must not be
substituted for these nonvanishing requirements. For the Jacobi
and Laguerre I families, this hypothesis is supplied in the parameter
ranges established later by
Corollary~\ref{cor:Jacobi-mixed-normality-near-diagonal} and
Proposition~\ref{prop:AT-normality-Laguerre-admissible}, respectively,
using the all-order nonresonance conditions in the Laguerre I case.
Under these nonvanishing assumptions, normalize the step-line forms by
\begin{equation}
	\label{eq:gen-step-normalization}
	B_N^{(\ell_q(N)+1)}
	\left[\bigl(\boldsymbol\sigma_q(N+1)\bigr)_{\ell_q(N)+1}-1\right]=1,
	\qquad
	\left\langle
	z^{(\boldsymbol\sigma_q(N))_{\ell_q(N)+1}}\mathcal A_N,
	v_{\ell_q(N)+1}
	\right\rangle_{\nu}=1.
\end{equation}
The orthogonality conditions then give
\begin{equation}
	\label{eq:gen-step-biorthogonality}
	\left\langle \mathcal B_N,\mathcal A_M\right\rangle_{\nu}
	=
	\delta_{N,M},
	\qquad N,M\in\mathbb N_0.
\end{equation}

\begin{proposition}[Step-line recurrence matrix]
\label{prop:general-step-recurrence}
Under the preceding nonvanishing assumptions, the step-line forms satisfy
\begin{align}
	\label{eq:gen-step-B-recurrence}
	z\mathcal B_N(z)
	={}&
	\sum_{j=1}^{q} b_N^{j}\mathcal B_{N+j}(z)
	+b_N^{0}\mathcal B_N(z)
	+\sum_{i=1}^{p} b_N^{-i}\mathcal B_{N-i}(z),
	\\[4pt]
	\label{eq:gen-step-A-recurrence}
	z\mathcal A_N(z)
	={}&
	\sum_{j=1}^{q} b_{N-j}^{j}\mathcal A_{N-j}(z)
	+b_N^{0}\mathcal A_N(z)
	+\sum_{i=1}^{p} b_{N+i}^{-i}\mathcal A_{N+i}(z),
\end{align}
where forms with negative indices are understood to be zero.  The coefficients
are
\begin{equation}
	\label{eq:gen-step-coefficients}
	b_N^k
	=
	\left\langle z\mathcal B_N,\mathcal A_{N+k}\right\rangle_{\nu},
	\qquad
	k\in\{-p,\ldots,-1,0,1,\ldots,q\}.
\end{equation}
The normalization \eqref{eq:gen-step-normalization} gives
\begin{equation}
	\label{eq:gen-top-coefficient}
	b_N^q=1,
	\qquad N\in\mathbb N_0.
\end{equation}
\end{proposition}
\begin{proof}
	The moment matrix satisfies the intertwining identity
	\(\Lambda_{[q]}\mathscr M=\mathscr M\Lambda_{[p]}^{\mathsf T}\), recorded
	explicitly in \eqref{eq:gen-hankel-symmetry} below.  Dressing this
	identity by the two Gauss--Borel factors gives the multiplication matrix
	\[
	T=\mathscr L\Lambda_{[q]}\mathscr L^{-1}
	 =\mathscr U^{-1}\Lambda_{[p]}^{\mathsf T}\mathscr U.
	\]
	The representation
	\(T=\mathscr L\Lambda_{[q]}\mathscr L^{-1}\) shows that no entry lies
	above the \(q\)-th superdiagonal, while
	\(T=\mathscr U^{-1}\Lambda_{[p]}^{\mathsf T}\mathscr U\) shows that no entry
	lies below the \(p\)-th subdiagonal.  Hence the \(N\)-th component of
	\(T\boldsymbol{\mathcal B}=z\boldsymbol{\mathcal B}\) contains only
	\(\mathcal B_{N-p},\ldots,\mathcal B_{N+q}\), which gives
	\eqref{eq:gen-step-B-recurrence}.

	Pairing that recurrence with \(\mathcal A_{N+k}\) and using
	\eqref{eq:gen-step-biorthogonality} isolates its \(k\)-th coefficient and
	gives \eqref{eq:gen-step-coefficients}.  On the \(A\)-side the matrix is
	\(T^{\mathsf T}\).  Since \(T_{M,N}=b_M^{N-M}\), the \(N\)-th component of
	\(T^{\mathsf T}\boldsymbol{\mathcal A}=z\boldsymbol{\mathcal A}\) is exactly
	\eqref{eq:gen-step-A-recurrence}.

	It remains to identify the coefficient of \(\mathcal B_{N+q}\).
	Put \(j=\ell_q(N)+1\).  The first condition in
	\eqref{eq:gen-step-normalization} makes the coefficient of
	\(z^{(\boldsymbol\sigma_q(N+1))_j-1}\) in \(B_N^{(j)}\) equal to one.
	Therefore the \(j\)-th component of \(z\mathcal B_N\) has coefficient
	one at degree \((\boldsymbol\sigma_q(N+1))_j\).  Every
	\(\mathcal B_M\) with \(M<N+q\) that occurs in the recurrence has
	smaller degree in its \(j\)-th component, whereas
	\[
	\boldsymbol\sigma_q(N+q+1)=\boldsymbol\sigma_q(N+1)+\one_q
	\]
	and the same normalization makes the coefficient of that degree in
	the \(j\)-th component of \(\mathcal B_{N+q}\) equal to one.
	Comparison of these coefficients gives \(b_N^q=1\).
\end{proof}

If
\[
\boldsymbol{\mathcal B}(z)
\coloneq
\begin{bNiceMatrix}
	\mathcal B_0(z)\\
	\mathcal B_1(z)\\
	\mathcal B_2(z)\\
	\Vdots
\end{bNiceMatrix},
\qquad
\boldsymbol{\mathcal A}(z)
\coloneq
\begin{bNiceMatrix}
	\mathcal A_0(z)\\
	\mathcal A_1(z)\\
	\mathcal A_2(z)\\
	\Vdots
\end{bNiceMatrix},
\]
then \eqref{eq:gen-step-B-recurrence}--\eqref{eq:gen-step-A-recurrence}
can be written as
\begin{equation}
	\label{eq:gen-step-matrix-recurrence}
	T\boldsymbol{\mathcal B}(z)=z\boldsymbol{\mathcal B}(z),
	\qquad
	T^{\mathsf T}\boldsymbol{\mathcal A}(z)=z\boldsymbol{\mathcal A}(z),
\end{equation}
where \(T\) is the \((p,q)\)-banded recurrence matrix defined by
\begin{equation}
	\label{eq:gen-recurrence-matrix-entries}
	T_{N,N+k}=b_N^k,
	\qquad
	N\in\mathbb N_0,
	\quad
	k\in\{-p,\ldots,q\},
\end{equation}
whenever \(N+k\ge0\), and with \(b_N^q=1\).  Moreover,
\begin{equation}
	\label{eq:gen-powers-recurrence-matrix}
	\left(T^\ell\right)_{N,M}
	=
	\left\langle
	z^\ell\mathcal B_N,
	\mathcal A_M
	\right\rangle_{\nu},
	\qquad
	\ell,N,M\in\mathbb N_0.
\end{equation}
For the Jacobi and Laguerre I systems below, the pairings in
\eqref{eq:gen-step-coefficients} are evaluated directly by finite
Gamma--Pochhammer sums.  No recurrence outside the step-line is needed.

\subsection{Moment matrix realization and Gauss--Borel factorization}
\label{subsec:general-moment-matrix-GB}

With the matrix of measures fixed in
\eqref{eq:general-matrix-of-measures}, introduce the block monomial
vector and block shift matrix
\begin{equation}
	\label{eq:gen-block-monomial-shift}
	X_{[d]}(z)
	\coloneq
	\begin{bNiceMatrix}
		I_d\\
		z I_d\\
		z^2 I_d\\
		\Vdots
	\end{bNiceMatrix},
	\qquad
	\Lambda_{[d]}
	\coloneq
	\begin{bNiceMatrix}[margin=4pt]
		0_d&I_d&0_d&\Cdots[shorten-end=3pt]\\
		0_d&0_d&I_d&\Ddots[shorten-end=3pt]\\
		0_d&0_d&0_d&\Ddots[shorten-end=3pt]\\
		\Vdots[shorten-end=3pt]&\Ddots[shorten-end=3pt]&\Ddots[shorten-end=3pt]&\Ddots[shorten-end=12pt]
	\end{bNiceMatrix},
\end{equation}
so that
\[
\Lambda_{[d]}X_{[d]}(z)=z X_{[d]}(z).
\]
The step-line moment matrix is
\begin{equation}
	\label{eq:gen-moment-matrix}
	\mathscr M
	\coloneq
	\int X_{[q]}(z)\,\mathrm d\boldsymbol\mu(z)\,
	X_{[p]}^{\mathsf T}(z),
\end{equation}
and satisfies the Hankel-type identity
\begin{equation}
	\label{eq:gen-hankel-symmetry}
	\Lambda_{[q]}\mathscr M
	=
	\mathscr M\Lambda_{[p]}^{\mathsf T}.
\end{equation}

Assume that the leading principal truncations of \(\mathscr M\) are
nonsingular.  Its Gauss--Borel factorization is written as
\begin{equation}
	\label{eq:gen-gauss-borel}
	\mathscr M
	=
	\mathscr L^{-1}\mathscr U^{-1},
\end{equation}
where \(\mathscr L\) is lower unitriangular and \(\mathscr U\) is upper
triangular.  The matrix polynomial systems
\begin{equation}
	\label{eq:gen-GB-polynomial-systems}
	B(z)
	\coloneq
	\mathscr L X_{[q]}(z),
	\qquad
	A(z)
	\coloneq
	X_{[p]}^{\mathsf T}(z)\mathscr U
\end{equation}
provide the component vectors of the step-line forms introduced in
\eqref{eq:gen-step-forms}, after the normalizations in
\eqref{eq:gen-step-normalization}.  In this realization,
\eqref{eq:gen-hankel-symmetry} gives
\begin{equation}
	\label{eq:gen-T-GB-realization}
	T
	=
	\mathscr L\Lambda_{[q]}\mathscr L^{-1}
	=
	\mathscr U^{-1}\Lambda_{[p]}^{\mathsf T}\mathscr U,
\end{equation}
which is the moment-matrix realization of the banded recurrence matrix in
\eqref{eq:gen-step-matrix-recurrence}.

\subsection{Christoffel perturbations and the bidiagonal factorization scheme}
\label{subsec:general-bidiagonal-scheme}

For the bidiagonal factorization, we first state the
Christoffel--Gauss--Borel construction and then evaluate the elementary factors
inside each explicit hypergeometric family. The lower factors come from the column-side
Christoffel chain and the upper factors from the row-side Christoffel chain.
Whenever the corresponding chain remains inside the same parametric family, the
factor is evaluated from the transformed mixed forms; otherwise it is kept in
the finite Christoffel tau-determinant form.  This distinction is essential for
mixed systems.

We recall the factorization developed for mixed step-line systems in
\cite{BranquinhoFoulquieManas2026Factorization}, in the form needed below.  Let
\begin{equation}
	\label{eq:gen-cyclic-polynomial-matrix}
	\mathfrak X_{[d]}(z)
	\coloneq
	\begin{bNiceMatrix}[margin=4pt,cell-space-limits=2pt]
		0&1&0&\Cdots[shorten-end=3pt]&0\\
		\Vdots[shorten-end=3pt]&\Ddots[shorten-end=3pt]&&\Ddots[shorten-end=3pt]&\Vdots[shorten-end=3pt]\\
		&&&\Ddots[shorten-end=3pt]&0\\
		0&\Cdots[shorten-end=3pt]&&0&1\\
		z&0&\Cdots[shorten-end=3pt]&0&0
	\end{bNiceMatrix},
	\qquad
	\mathfrak X_{[d]}(z)^d=z I_d.
\end{equation}
Thus the elementary cycle is
\[
(\theta_1,\ldots,\theta_d)
\longmapsto
(\theta_2,\ldots,\theta_d,\theta_1+1),
\]
which is the convention compatible with the step-line order
\eqref{eq:gen-step-coordinate}.  The elementary left and right Christoffel
perturbations are
\begin{equation}
	\label{eq:gen-Christoffel-perturbations}
	\mathrm d\boldsymbol\mu_{\mathrm L}^{(k)}
	\coloneq
	\mathrm d\boldsymbol\mu
	\left(\mathfrak X_{[p]}^{k}\right)^{\mathsf T},
	\qquad
	\mathrm d\boldsymbol\mu_{\mathrm R}^{(k)}
	\coloneq
	\mathfrak X_{[q]}^{k}\mathrm d\boldsymbol\mu.
\end{equation}
Their moment matrices satisfy
\begin{equation}
	\label{eq:gen-perturbed-moment-matrices}
	\mathscr M_{\mathrm L}^{(k)}
	=
	\mathscr M\left(\Lambda^{k}\right)^{\mathsf T},
	\qquad
	\mathscr M_{\mathrm R}^{(k)}
	=
	\Lambda^{k}\mathscr M,
	\qquad
	\Lambda\coloneq\Lambda_{[1]}.
\end{equation}
Assume that the Gauss--Borel factorizations of these perturbed moment matrices
exist for all shifts which occur below.

\begin{proposition}[Step-line bidiagonal factorization scheme]
	\label{prop:gen-bidiagonal-scheme}
	Assume that the step-line mixed system is normal along the Christoffel
	chains used below and that all pivots in the construction are nonzero.
	Then the step-line multiplication matrix \(T\) admits the bidiagonal
	factorization
	\begin{equation}
		\label{eq:gen-bidiagonal-factorization}
		T
		=
		L_1\cdots L_p\,U_q\cdots U_1,
	\end{equation}
	where \(L_1,\ldots,L_p\) are lower bidiagonal matrices with unit main
	diagonal, and \(U_1,\ldots,U_q\) are upper bidiagonal matrices with unit
	first superdiagonal.
	
	For \(k\in\{0,\ldots,p\}\), let
	\[
	\mathcal A_N^{\langle k\rangle}(z)
	=
	\sum_{i=1}^{k}
	A_N^{\langle k\rangle,(i)}(z)z^{\theta_i+1}
	+
	\sum_{i=k+1}^{p}
	A_N^{\langle k\rangle,(i)}(z)z^{\theta_i}
	\]
	be the step-line \(A\)-form for the system obtained by multiplying the
	first \(k\) left powers by \(z\), in the step-line biorthonormal
	normalization.  For a polynomial \(P\), write \(\operatorname{LC}(P)\)
	for its leading coefficient.  Then
	\begin{equation}
		\label{eq:gen-L-leading-coefficients}
		(L_k)_{N+1,N}
		=
		\frac{
			\operatorname{LC}
			\left(
			A_N^{\langle k\rangle,(\ell_p(N)+1)}
			\right)
		}{
			\operatorname{LC}
			\left(
			A_{N+1}^{\langle k-1\rangle,(\ell_p(N+1)+1)}
			\right)
		},
		\qquad
		k\in\{1,\ldots,p\}.
	\end{equation}
	
	For \(b\in\{0,\ldots,q\}\), set
	\[
	v_j^{[b]}(z)
	=
	\begin{cases}
		zv_j(z), & j\in\{1,\ldots,b\},\\
		v_j(z), & j\in\{b+1,\ldots,q\}.
	\end{cases}
	\]
	In particular, \(v_j^{[0]}=v_j\).  Let
	\[
	\mathcal B_N(z)
	=
	\sum_{j=1}^{q}B_N^{(j)}(z)v_j(z)
	\]
	be the untransformed step-line \(B\)-form in the same biorthonormal
	normalization.  Define
	\begin{equation}
		\label{eq:gen-upper-Christoffel-tau}
		\tau^B_{0,N}\coloneq1,
		\qquad
		\tau^B_{b,N}
		\coloneq
		\det\left[
		B_{N+\alpha-1}^{(\nu)}(0)
		\right]_{\alpha,\nu=1}^{b},
		\qquad
		b\in\{1,\ldots,q\}.
	\end{equation}
	Whenever these Christoffel tau-functions do not vanish, the upper factors
	are
	\begin{equation}
		\label{eq:gen-U-tau-quotients}
		(U_b)_{N,N}
		=
		-
		\frac{
			\tau^B_{b-1,N}\tau^B_{b,N+1}
		}{
			\tau^B_{b-1,N+1}\tau^B_{b,N}
		},
		\qquad
		b\in\{1,\ldots,q\}.
	\end{equation}
	
	If the Christoffel chain associated with the upper bidiagonal factors is
	closed inside the same parametric family, let
	\[
	\mathcal A_N^{[b]}(z)
	=
	\sum_{i=1}^{p}
	A_N^{[b],(i)}(z)z^{\theta_i}
	\]
	be the step-line \(A\)-form for the system with right weights
	\[
	v_1^{[b]},\ldots,v_q^{[b]},
	\]
	in the corresponding biorthonormal normalization.  Then the same
	quantities may equivalently be written as quotients of leading
	coefficients:
	\begin{equation}
		\label{eq:gen-U-leading-coefficients}
		(U_b)_{N,N}
		=
		\frac{
			\operatorname{LC}
			\left(
			A_N^{[b-1],(\ell_p(N)+1)}
			\right)
		}{
			\operatorname{LC}
			\left(
			A_N^{[b],(\ell_p(N)+1)}
			\right)
		},
		\qquad
		b\in\{1,\ldots,q\}.
	\end{equation}
\end{proposition}

The identity \eqref{eq:gen-bidiagonal-factorization} is an entrywise identity
of semi-infinite matrices.  No convergence of infinite matrix products is
involved: there are only \(p+q\) bidiagonal factors, and each entry of their
product is a finite sum.  Equivalently, for any fixed set of rows and all
columns meeting their band, the identity is obtained from a sufficiently
large finite Christoffel truncation.  This is the form of the general
factorization theorem from
\cite{BranquinhoFoulquieManas2026Factorization} used here;
below we evaluate its factors for the two hypergeometric families.

The distinction in Proposition~\ref{prop:gen-bidiagonal-scheme} is essential.
The lower factors are evaluated from the left Christoffel chain.  In the two
mixed settings below this chain is closed because the \(p\)-component side is
a power vector.  Thus the left-transformed \(A\)-forms are obtained by the
same formulas after a cyclic affine shift of the power exponents.  The upper
factors are different: outside a closed right chain they must be computed from
the Christoffel tau-functions \eqref{eq:gen-upper-Christoffel-tau}. 

Define the cyclic affine shift
\begin{equation}
	\label{eq:gen-cyclic-affine-shift}
	\mathscr C_d(\theta_1,\ldots,\theta_d)
	\coloneq
	(\theta_2,\ldots,\theta_d,\theta_1+1).
\end{equation}
In the Jacobi and Laguerre I settings, respectively, the left
Christoffel chain preserves the family by
\[
\boldsymbol\alpha\longmapsto \mathscr C_p^k(\boldsymbol\alpha),
\qquad
\boldsymbol\beta\longmapsto \mathscr C_p^k(\boldsymbol\beta).
\]
Bilateral preservation occurs in the mixed Pi\~neiro specialization
of~\cite{PineiroMixed2026} in the Jacobi setting, and formally at the
\(s=0\) boundary of the
Laguerre I setting; in those cases one may use either the closed
coefficient formula \eqref{eq:gen-U-leading-coefficients} or the tau-formula
\eqref{eq:gen-U-tau-quotients}.

\section{Jacobi mixed-type multiple orthogonal systems}

We first consider a rectangular mixed extension of the Jacobi
hypergeometric moment systems introduced by Wolfs~\cite{Wolfs2024}. The
power vector is
\[
\mathbf u(x)=\left[x^{\alpha_1},\ldots,x^{\alpha_p}\right],
\]
whereas the weight vector is formed by Jacobi Mellin weights. When
\(q=1\), the construction reduces to the ordinary Jacobi--Pi\~neiro system
with \(p\) weights. When \(p=1\) and \(\alpha_1=0\), it reduces to the
	ordinary Jacobi-like system of Wolfs~\cite{Wolfs2024}. The boundary degeneration
	\(\boldsymbol a,\boldsymbol b\to\boldsymbol\beta\) gives the mixed Pi\~neiro matrix
of~\cite{PineiroMixed2026}. Finally,
the case \(q=2\) gives a Gauss hypergeometric specialization. After matching
normalizations, the relevant two-dimensional space of weights is related to
the Lima--Loureiro Gauss pair~\cite{LimaLoureiro2022} by an invertible
constant change of basis.

Let
\[
\boldsymbol a=(a_1,\ldots,a_q),
\qquad
\boldsymbol b=(b_1,\ldots,b_q),
\qquad
-1<a_j<b_j,
\quad j\in\{1,\ldots,q\}.
\]
For \(-1<a<b\), we use the Mellin-normalized beta weight
\begin{equation}
	\label{eq:beta-weight-definition}
	\mathcal B_{a,b}(x)
	\coloneq
	\frac{x^a(1-x)^{b-a-1}}{\Gamma(b-a)}
	\boldsymbol 1_{(0,1)}(x),
	\qquad
	\mathcal M[\mathcal B_{a,b}](s)
	=
	\frac{\Gamma(s+a)}{\Gamma(s+b)}.
\end{equation}
For vector-valued parameter strings, scalar shifts, addition, subtraction,
and products are understood componentwise; in particular,
\[
\Gamma(s\one_q+\boldsymbol a)\coloneq\prod_{j=1}^{q}\Gamma(s+a_j),
\qquad
(\boldsymbol a)_k\coloneq\prod_{j=1}^{q}(a_j)_k.
\]
The vector obtained by removing the \(j\)-th component of \(\boldsymbol a\) is
denoted by \(\boldsymbol a^{\,*j}\).

The finite nonresonance assumptions needed for the residue formulas are stated
explicitly in each result.

The basic Jacobi weight is
\begin{equation}
	\label{eq:w0-convolution}
	w_0(x;\boldsymbol a,\boldsymbol b)
	\coloneq
	\mathcal B_{a_1,b_1}*\cdots*\mathcal B_{a_q,b_q}(x).
\end{equation}
Since each factor in the Mellin convolution is supported on \((0,1)\), the
iterated convolution is also supported on \((0,1)\). Moreover, by the Mellin
transform of the beta density and by
\eqref{eq:Mellin-convolution-product}, it satisfies
\begin{equation}
	\label{eq:w0-Mellin}
	\mathcal M[w_0(\cdot;\boldsymbol a,\boldsymbol b)](s)
	=
	\frac{\Gamma(s\one_q+\boldsymbol a)}{\Gamma(s\one_q+\boldsymbol b)}.
\end{equation}

With the Meijer \(G\)-function convention fixed in
\eqref{eq:Meijer-G-definition}--\eqref{eq:Meijer-G-Mellin-transform}, the
upper parameters of the Meijer \(G\)-function appear in the denominator of
the Mellin transform, whereas the lower parameters appear in the numerator.
Therefore
\begin{equation}
	\label{eq:w0-Meijer}
	w_0(x;\boldsymbol a,\boldsymbol b)
	=
	G_{q,q}^{q,0}\left(
	x\,\middle|\,\begin{matrix}\boldsymbol b\\ \boldsymbol a\end{matrix}
	\right),
	\qquad 0<x<1.
\end{equation}

Following Wolfs~\cite{Wolfs2024}, define
\begin{equation}
	\label{eq:wj-definition}
	w_j(x;\boldsymbol a,\boldsymbol b)
	\coloneq
	w_0(x;\boldsymbol a,\boldsymbol b+\boldsymbol e_j),
	\qquad j\in\{1,\ldots,q\}.
\end{equation}
Then
\begin{equation}
	\label{eq:wj-Mellin}
	\mathcal M[w_j(\cdot;\boldsymbol a,\boldsymbol b)](s)
	=
	\frac{\Gamma(s\one_q+\boldsymbol a)}{\Gamma(s\one_q+\boldsymbol b+\boldsymbol e_j)}
	=
	\frac{\Gamma(s\one_q+\boldsymbol a)}{\Gamma(s\one_q+\boldsymbol b)}
	\frac{1}{s+b_j}.
\end{equation}

Assume that
\[
a_\rho-a_\sigma\notin\mathbb Z,
\qquad \rho\neq\sigma.
\]
With the Meijer \(G\)-normalization fixed in
\eqref{eq:Meijer-G-definition}--\eqref{eq:Meijer-G-Mellin-transform},
Slater's expansion of the Meijer \(G\)-function near \(x=0\), obtained by
summing the residues at the poles \(s=-a_\rho-k\), gives, for \(0<x<1\) and
every \(j\in\{1,\ldots,q\}\),
\begin{align}
	\label{eq:wj-general-hypergeometric}
	w_j(x;\boldsymbol a,\boldsymbol b)
	={}&
	\sum_{\rho=1}^{q}
	C_{j,\rho}\,x^{a_\rho}
	\pFq{q}{q-1}
	{
		(1+a_\rho)\one_q-\boldsymbol b-\boldsymbol e_j
	}
	{
		(1+a_\rho)\one_{q-1}-\boldsymbol a^{\,*\rho}
	}
	{x},
\end{align}
where
\begin{equation}
	\label{eq:wj-general-coefficients}
	C_{j,\rho}
	\coloneq
	\frac{
		\Gamma(\boldsymbol a^{\,*\rho}-a_\rho\one_{q-1})
	}{
		\Gamma(\boldsymbol b+\boldsymbol e_j-a_\rho\one_q)
	}.
\end{equation}
Thus each Jacobi weight is, generically, a linear combination of
Frobenius-type terms \(x^{a_\rho}{}_qF_{q-1}(x)\) at the origin.

\begin{remark}[The Jacobi--Pi\~neiro specialization: \(q=1\)]
	\label{rem:JP-edge}
	If \(q=1\), write \(\boldsymbol a=(a)\) and \(\boldsymbol b=(b)\). Then
	\begin{equation}
		\label{eq:q1-weight}
		w_1(x;a,b)
		=
		w_0(x;a,b+1)
		=
		\frac{x^a(1-x)^{b-a}}{\Gamma(b-a+1)}.
	\end{equation}
	Consequently, multiplication by arbitrary powers gives
	\[
	x^{\alpha_i}w_1(x;a,b)
	=
	\frac{x^{\alpha_i+a}(1-x)^{b-a}}{\Gamma(b-a+1)},
	\qquad i\in\{1,\ldots,p\},
	\]
	which, up to a common nonzero constant, are precisely the Jacobi--Pi\~neiro
	weights after setting
	\[
	\widetilde\alpha_i\coloneq\alpha_i+a,
	\qquad
	\gamma\coloneq b-a.
	\]
	This agrees with the Mellin-transform normalization for Jacobi--Pi\~neiro
	weights used, for instance, by Smet and Van Assche~\cite{SmetVanAssche2010}.
	Although the sufficient hypotheses used above to construct the full positive
	Mellin-convolution family impose \(-1<a<b\), the specialized
	Jacobi--Pi\~neiro weights are integrable under the usual conditions
	\(\widetilde\alpha_i>-1\) and \(\gamma>-1\).
\end{remark}

\begin{remark}[Relation with the Gauss system of Lima and Loureiro]
	\label{rem:Gauss-normalization}
	For \(q=2\), set
	\[
	\boldsymbol a=(a-1,b-1),
	\qquad
	\boldsymbol b=(c-1,d-1),
	\qquad
	\delta\coloneq c+d-a-b.
	\]
	Then
	\begin{equation}
		\label{eq:w0-Gauss-LL-corrected}
		w_0(x;(a-1,b-1),(c-1,d-1))
		=
		\frac{\Gamma(a)\Gamma(b)}{\Gamma(c)\Gamma(d)}
		W_{\rm LL}(x;a,b;c,d),
	\end{equation}
	where \(W_{\rm LL}\) denotes the Gauss hypergeometric weight used by
	Lima and Loureiro~\cite{LimaLoureiro2022}. We write
	\[
	\mathcal W(x;a,b;c,d)\coloneq
	\frac{\Gamma(a)\Gamma(b)}{\Gamma(c)\Gamma(d)}
	W_{\rm LL}(x;a,b;c,d)
	\]
	for its Mellin-normalized version. Their second weight is
	\(W_{\rm LL}(x;a,b+1;c+1,d)\). The shifted Jacobi weight
	\[
	w_0(x;(a-1,b-1),(c,d-1))
	\]
	corresponds, in Mellin normalization, to \(\mathcal W(x;a,b;c+1,d)\),
	and the constant contiguous relation
	\begin{equation}
		\label{eq:Gauss-contiguous-Mellin-normalized}
		\mathcal W(x;a,b;c+1,d)
		=
		\frac{1}{c-b}\mathcal W(x;a,b;c,d)
		-
		\frac{1}{c-b}\mathcal W(x;a,b+1;c+1,d)
	\end{equation}
	shows that the two pairs of weights span the same two-dimensional space,
	provided \(c-b\neq0\). Thus the \(q=2\) Jacobi construction recovers
	the Lima--Loureiro Gauss system up to normalization and an invertible
	constant change of basis, in the corresponding triangular range of
	multi-indices.
\end{remark}

\begin{remark}[Pi\~neiro degeneration]
	\label{rem:Pineiro-degeneration}
	This is the mixed Pi\~neiro system studied in~\cite{PineiroMixed2026}.
	If one takes the boundary limit
	\[
		\boldsymbol a,\boldsymbol b\longrightarrow\boldsymbol\beta,
	\]
	then \eqref{eq:wj-Mellin} gives
	\[
	\mathcal M[w_j(\cdot;\boldsymbol\beta,\boldsymbol\beta)](s)
	=
	\frac{1}{s+\beta_j},
	\]
	and therefore
	\begin{equation}
		\label{eq:wj-Pineiro-degeneration}
		w_j(x;\boldsymbol\beta,\boldsymbol\beta)=x^{\beta_j},
		\qquad j\in\{1,\ldots,q\}.
	\end{equation}
\end{remark}

Let
\[
\boldsymbol\alpha=(\alpha_1,\ldots,\alpha_p),
\]
and assume that
\[
\alpha_i+a_\rho>-1,
\qquad
i\in\{1,\ldots,p\},
\quad
\rho\in\{1,\ldots,q\}.
\]
This condition guarantees the absolute convergence of all pairings appearing
below. Equivalently, if
\[
a_{\min}\coloneq\min_{\rho\in\{1,\ldots,q\}}a_\rho,
\]
it is enough to require \(\alpha_i>-1-a_{\min}\) for every
\(i\in\{1,\ldots,p\}\). Consider the two vectors
\[
\mathbf u(x)=\left[x^{\alpha_1},\ldots,x^{\alpha_p}\right],
\qquad
\mathbf w(x)=\left[w_1(x;\boldsymbol a,\boldsymbol b),\ldots,w_q(x;\boldsymbol a,\boldsymbol b)\right].
\]
Following the convention \eqref{eq:general-matrix-of-measures}, we define the
\(q\times p\) matrix of measures
\begin{equation}
	\label{eq:mixed-Jacobi-like-matrix}
	\mathrm{d}\Jmat(x)
	\coloneq
	\left[
	w_j(x;\boldsymbol a,\boldsymbol b)x^{\alpha_i}
	\right]_{
		\substack{j\in\{1,\ldots,q\}\\i\in\{1,\ldots,p\}}}
	\dx.
\end{equation}

The preceding specializations pass directly to this matrix of measures. In the
case \(q=1\), Remark~\ref{rem:JP-edge} gives, up to a common nonzero constant,
the ordinary Jacobi--Pi\~neiro vector with parameters
\(\widetilde\alpha_i=\alpha_i+a\) and \(\gamma=b-a\). If \(p=1\) and
\(\alpha_1=0\), the matrix reduces to the ordinary Jacobi-like vector studied
by Wolfs. More generally, multiplication by \(x^{\alpha_1}\) corresponds, at
the level of Mellin transforms, to the simultaneous shift
\[
(\boldsymbol a,\boldsymbol b)
\longmapsto
(\boldsymbol a+\alpha_1\one_q,\boldsymbol b+\alpha_1\one_q).
\]
	Finally, in the boundary degeneration
	\(\boldsymbol a,\boldsymbol b\to\boldsymbol\beta\),
	\eqref{eq:wj-Pineiro-degeneration} gives the mixed Pi\~neiro matrix
of~\cite{PineiroMixed2026},
\[
\mathrm{d}\Jmat(x)
=
\left[
x^{\beta_j+\alpha_i}
\right]_{
	\substack{j\in\{1,\ldots,q\}\\i\in\{1,\ldots,p\}}}
\dx.
\]

\subsection{Leading moment determinants}
\label{subsec:Jacobi-leading-moment-determinants}

The leading moment determinants can be evaluated at every step, including a
final incomplete cycle. For \(N\ge1\), let
\begin{equation}
	\label{eq:Jacobi-step-component-counts}
	m_j(N)\coloneq\left\lfloor\frac{N+q-j}{q}\right\rfloor,
	\qquad
	n_i(N)\coloneq\left\lfloor\frac{N+p-i}{p}\right\rfloor.
\end{equation}
Thus \(m_j(N)\) and \(n_i(N)\) count the occurrences of the \(j\)-th row
component and the \(i\)-th column component among the first \(N\) step-line
indices, and
\[
	\sum_{j=1}^{q}m_j(N)=\sum_{i=1}^{p}n_i(N)=N.
\]
For \(0\le v<n_i(N)\), set
\[
	z_{v,i}\coloneq v+\alpha_i+1,
\]
and order the resulting \(N\) nodes by the scalar step-line index
\(pv+i-1\). Denote this ordered list by \(\mathcal Z_N\), and write
\[
	\operatorname{VdM}(\mathcal Z_N)
	\coloneq
	\prod_{\lambda<\mu}(z_\mu-z_\lambda).
\]
The leading \(N\times N\) moment block is
\begin{equation}
	\label{eq:Jacobi-leading-moment-block}
	\mathscr M_N^{\mathrm J}
	\coloneq
	\left[
		\int_0^1 x^{u+v+\alpha_i}w_j(x;\boldsymbol a,\boldsymbol b)\,\dx
	\right]_{
		\substack{0\le u<m_j(N),\ 1\le j\le q\\
		0\le v<n_i(N),\ 1\le i\le p}
	},
\end{equation}
\begin{samepage}
where the rows and columns follow their scalar step-line order. Define
\begin{align}
	D_N(z)
	&\coloneq
	\prod_{j=1}^{q}(z+b_j)_{m_j(N)},
	\label{eq:Jacobi-leading-common-denominator}\\
	\Delta_N^{\mathrm J}
	&\coloneq
	\prod_{j=1}^{q}
	\left[
		\prod_{h=1}^{j-1}(b_j-b_h)^{m_j(N)}
		\prod_{\rho=1}^{q}\prod_{u=0}^{m_j(N)-1}
		(b_j+1-a_\rho)_u
	\right].
	\label{eq:Jacobi-leading-coefficient-determinant}
\end{align}
\end{samepage}

\begin{proposition}[All leading Jacobi moment determinants]
	\label{prop:Jacobi-all-leading-moment-determinants}
	Assume that the Gamma factors and denominators below are finite and
	nonzero. For every \(N\ge1\),
	\begin{equation}
		\label{eq:Jacobi-all-leading-moment-determinants}
		\det\mathscr M_N^{\mathrm J}
		=
		\operatorname{VdM}(\mathcal Z_N)\,
		\Delta_N^{\mathrm J}
		\prod_{z\in\mathcal Z_N}
		\frac{\Gamma(z\one_q+\boldsymbol a)}{\Gamma(z\one_q+\boldsymbol b)D_N(z)}.
	\end{equation}
\end{proposition}

\begin{proof}
	Put
	\[
		G(z)\coloneq\frac{\Gamma(z\one_q+\boldsymbol a)}{\Gamma(z\one_q+\boldsymbol b)}
	\]
	and, for every row index \((u,j)\), define
	\begin{equation}
		\label{eq:Jacobi-rational-collocation-functions}
		F_{u,j}(z)
		\coloneq
		\frac{\prod_{\rho=1}^{q}(z+a_\rho)_u}
		{(z+b_j)_{u+1}
		\prod_{\substack{1\le h\le q\\h\ne j}}(z+b_h)_u}.
	\end{equation}
	Formula \eqref{eq:wj-Mellin}, with
	\(s=u+v+\alpha_i+1=u+z_{v,i}\), gives
	\begin{equation}
		\label{eq:Jacobi-moment-collocation-identity}
		\int_0^1x^{u+v+\alpha_i}w_j(x;\boldsymbol a,\boldsymbol b)\,\dx
		=
		G(z_{v,i})F_{u,j}(z_{v,i}).
	\end{equation}

	The component counts in \eqref{eq:Jacobi-step-component-counts} differ by
	at most one, and the components with the larger count occur first. If
	\(u<m_j(N)\), then \(u+1\le m_j(N)\) and
	\(u\le m_h(N)\) for \(h\ne j\). Hence the denominator of
	\(F_{u,j}\) divides \(D_N\), and
	\[
		P_{u,j}^{(N)}(z)\coloneq D_N(z)F_{u,j}(z)
	\]
	is a polynomial of degree at most \(N-1\). Let \(\mathcal C_N\) be the
	coefficient matrix of these \(N\) polynomials in the ascending monomial
	basis. Evaluation at \(\mathcal Z_N\) factors the moment block into a
	coefficient matrix, a Vandermonde matrix, and a diagonal matrix. Therefore,
	\begin{equation}
		\label{eq:Jacobi-leading-moment-collocation-factorization}
		\det\mathscr M_N^{\mathrm J}
		=
		\det\mathcal C_N\,
		\operatorname{VdM}(\mathcal Z_N)
		\prod_{z\in\mathcal Z_N}\frac{G(z)}{D_N(z)}.
	\end{equation}

	It remains to calculate \(\det\mathcal C_N\). Write
	\[
		N=qm+j,
		\qquad
		m\ge0,
		\quad
		j\in\{1,\ldots,q\}.
	\]
	The row added when passing from order \(N-1\) to order \(N\) is
	\((m,j)\), and
	\[
		D_N(z)=D_{N-1}(z)(z+b_j+m).
	\]
	Here and below, \(D_0(z)=1\).
	Thus every old polynomial acquires the factor \(z+b_j+m\), whereas the
	new polynomial is
	\[
		P_{m,j}^{(N)}(z)
		=
		\prod_{\rho=1}^{q}(z+a_\rho)_m
		\prod_{h=1}^{j-1}(z+b_h+m).
	\]
	At the root \(-b_j-m\) of the new factor,
	\begin{equation}
		\label{eq:Jacobi-new-polynomial-root-value}
		P_{m,j}^{(N)}(-b_j-m)
		=
		(-1)^{N-1}
		\prod_{h=1}^{j-1}(b_j-b_h)
		\prod_{\rho=1}^{q}(b_j+1-a_\rho)_m.
	\end{equation}
	Reducing the last polynomial modulo \(z+b_j+m\) and expanding the
	coefficient determinant gives
	\[
		\det\mathcal C_N
		=
		(-1)^{N-1}P_{m,j}^{(N)}(-b_j-m)\det\mathcal C_{N-1}.
	\]
	The two signs cancel by
	\eqref{eq:Jacobi-new-polynomial-root-value}, and hence
	\begin{equation}
		\label{eq:Jacobi-leading-coefficient-recursion}
		\frac{\det\mathcal C_N}{\det\mathcal C_{N-1}}
		=
		\prod_{h=1}^{j-1}(b_j-b_h)
		\prod_{\rho=1}^{q}(b_j+1-a_\rho)_m.
	\end{equation}
	Starting with \(\det\mathcal C_0=1\) and iterating in the scalar
	step-line order yields
	\(\det\mathcal C_N=\Delta_N^{\mathrm J}\). Substitution in
	\eqref{eq:Jacobi-leading-moment-collocation-factorization} proves
	\eqref{eq:Jacobi-all-leading-moment-determinants}.
\end{proof}

\begin{corollary}[Step-line normality]
	\label{cor:Jacobi-step-line-normality}
	Under the hypotheses of
	Proposition~\ref{prop:Jacobi-all-leading-moment-determinants}, the block
	\(\mathscr M_N^{\mathrm J}\) is nonsingular if and only if the nodes in
	\(\mathcal Z_N\) are distinct and
	\[
		b_j\ne b_h
		\quad\textnormal{for }1\le h<j\le q\textnormal{ with }m_j(N)>0,
	\]
	and
	\[
		b_j+r\ne a_\rho
		\quad\textnormal{for }
		1\le \rho,j\le q,
		\quad
		1\le r<m_j(N).
	\]
	Consequently, if these conditions hold for every \(N\), all leading
	principal truncations of the step-line moment matrix are nonsingular and
	the Gauss--Borel factorization \eqref{eq:gen-gauss-borel} exists. This
	criterion concerns the chosen step-line. The AT result below additionally
	provides normality for the stated class of multi-indices, as well as the
	zero properties of the mixed forms; interlacing is stated under the full
	AT hypothesis below.
\end{corollary}

\begin{proposition}[Normality and zeros under the full AT hypothesis]
\label{prop:AT-perfect-mixed}
Assume the standing integrability hypotheses and that the row vector
\(\mathbf w=(w_1,\ldots,w_q)\) is an AT-system on \((0,1)\) for every
multi-index. Assume also that \(\alpha_i-\alpha_h\notin\mathbb Z\) for
\(i\ne h\). Then both mixed problems are normal and every component with positive
prescribed length has its maximal degree. The forms \(\mathcal A\)
and \(\mathcal B\) have respectively \(|\boldsymbol m|\) and
\(|\boldsymbol n|\) simple zeros in \((0,1)\), and neighboring forms
interlace in the sense of \cite[Theorem~5]{FidalgoMedinaMinguez2013}.
\end{proposition}

\begin{proof}
Distinctness modulo integers of the \(\alpha_i\)'s makes the power
vector an AT-system for every multi-index. Thus
\(\mathrm d\Jmat=\mathbf w^{\mathsf T}[x^{\alpha_i}]_{i\in\{1,\ldots,p\}}\dx\)
is an AT matrix measure as defined in
\cite[Definition~5]{FidalgoMedinaMinguez2013}.
Theorems~1 and~2 of that reference give the zero count and uniqueness.
For the degree assertion, reducing the degree of component \(j\) places
the form in the space indexed by the corresponding multi-index minus
\(\boldsymbol e_j\). That space is still AT by the full hypothesis and
admits one fewer zero, contradicting the zero count of the form.
This argument applies to either side of the matrix. The neighboring-form
assertion follows from Theorem~5 of the same reference.
\end{proof}

For the Jacobi weights, the available AT assertion below concerns only
near-diagonal indices. Normality in that range will instead be proved
directly from the Mellin coefficient matrix. This also specifies exactly
which extra minors control the maximal degrees of the \(B\)-components.

\begin{definition}[Near-diagonal multi-index]
	\label{def:near-diagonal}
	A multi-index \(\boldsymbol m=(m_1,\ldots,m_q)\in\mathbb N_0^q\) is said to be
	near the diagonal if
	\[
	|m_i-m_j|\le 1,
	\qquad i,j\in\{1,\ldots,q\}.
	\]
\end{definition}

\begin{lemma}[Wolfs {\cite[Corollary~2.4]{Wolfs2024}}]
	\label{lem:Wolfs-near-diagonal-AT}
	Assume that
	\[
	b_j-a_j\in\mathbb N_{0},
	\qquad j\in\{1,\ldots,q\},
	\]
	and
	\[
	b_i-b_j\notin\mathbb Z,
	\qquad i,j\in\{1,\ldots,q\},\quad i\ne j.
	\]
	Then the Jacobi vector
	\[
	\mathbf w(x)
	\coloneq
	\left[
	w_1(x;\boldsymbol a,\boldsymbol b),\ldots,
	w_q(x;\boldsymbol a,\boldsymbol b)
	\right]
	\]
	is an AT-system on \((0,1)\) for every near-diagonal multi-index
	\(\boldsymbol m\in\mathbb N_0^q\).
\end{lemma}

\begin{corollary}[Near-diagonal mixed normality]
\label{cor:Jacobi-mixed-normality-near-diagonal}
Assume the hypotheses of Lemma~\ref{lem:Wolfs-near-diagonal-AT}, the
standing integrability conditions, and
\(\alpha_i-\alpha_h\notin\mathbb Z\) for \(i\ne h\).
Let \(\boldsymbol m\) be near the diagonal and put \(N=|\boldsymbol m|\).
Both balances are normal. In the \(A\)-balance \(|\boldsymbol n|=N+1\),
every component with \(n_i>0\) has degree \(n_i-1\). In the
\(B\)-balance \(|\boldsymbol n|=N-1\), the complete form is nonzero
and its polynomial components are unique up to a common factor.

For the latter balance, let \(G\) be the \(N\)-by-\((N-1)\) matrix
with rows \((j,k)\), \(0\le k<m_j\), columns \((i,r)\),
\(0\le r<n_i\), and entries
\(G_{(j,k),(i,r)}=\int_0^1x^{k+r+\alpha_i}w_j(x)\dx\).
Then all \(B\)-components attain their maximal degrees if and only if
\begin{equation}
 \det G\bigl[\textnormal{row }(j,m_j-1)\textnormal{ deleted}\bigr]\ne0,
 \qquad m_j>0.
 \label{eq:Jacobi-terminal-minor-condition}
\end{equation}
An empty determinant equals one. In particular, this condition holds
for every component whose reduced index
\(\boldsymbol m-\boldsymbol e_j\) remains near the diagonal.
The forms \(\mathcal A\) and \(\mathcal B\) have respectively
\(|\boldsymbol m|\) and \(|\boldsymbol n|\) simple zeros in \((0,1)\).
The full AT hypothesis of Proposition~\ref{prop:AT-perfect-mixed}
provides, separately, its neighboring-form interlacing conclusion.
\end{corollary}

\begin{proof}
Put \(G_0(z)\coloneq\Gamma(z\one_q+\boldsymbol a)/\Gamma(z\one_q+\boldsymbol b)\)
and \(D_{\boldsymbol m}(z)\coloneq\prod_h(z+b_h)_{m_h}\). Division of the
weighted Mellin transforms by \(G_0/D_{\boldsymbol m}\) gives the
polynomials
\[
 P_{j,k}(z)\coloneq\prod_{\rho=1}^q(z+a_\rho)_k
             \prod_{h=1}^q(z+b_h+k+\delta_{h,j})_{m_h-k-\delta_{h,j}},
 \qquad 0\le k<m_j.
\]
Near-diagonality makes their lengths nonnegative and their degrees at
most \(N-1\). In degree-first row order, their coefficient determinant is
\begin{equation}
 c_{\boldsymbol m}^{\mathrm J}
 =\prod_{h<j}(b_j-b_h)^{\min(m_h,m_j)}
   \prod_{j=1}^q\prod_{\rho=1}^q\prod_{k=0}^{m_j-1}(b_j+1-a_\rho)_k.
 \label{eq:Jacobi-near-diagonal-coefficient-determinant}
\end{equation}
To obtain this formula, perform the complete degree cycles in the proof of
Proposition~\ref{prop:Jacobi-all-leading-moment-determinants} and then
append precisely the components with one extra entry, in their original
order. Appending component \(j\) at degree \(k\) multiplies the old
polynomials by \(z+b_j+k\); evaluation of the new row at
\(-b_j-k\) gives the factor
\(\prod_\rho(b_j+1-a_\rho)_k\) and one factor \(b_j-b_h\) for each
component \(h\) already appended in that cycle. Each pair occurs in
\(\min(m_h,m_j)\) cycles, proving the displayed determinant.

All factors are nonzero. For \(\rho=j\), \(b_j+1-a_j>0\).
For \(\rho\ne j\), a zero would imply an integer difference between
\(b_j\) and \(b_\rho\), since \(b_\rho-a_\rho\) is an integer.
Thus the weighted Mellin map is an isomorphism onto
\(\mathbb P_{N-1}\). At \(N\) distinct column nodes
\(z=\alpha_i+r+1\), the square moment determinant is consequently
\[
 c_{\boldsymbol m}^{\mathrm J}
 \prod_{u<v}(z_v-z_u)\prod_v\frac{G_0(z_v)}{D_{\boldsymbol m}(z_v)}\ne0.
\]
Evaluation at \(N+1\) distinct nodes has rank \(N\), and every maximal
minor is nonzero. Its kernel is one-dimensional with every coordinate
nonzero, proving the \(A\)-assertions. Evaluation at \(N-1\) distinct
nodes has rank \(N-1\); its transpose kernel gives the unique
\(B\)-vector, whose nonzero Mellin numerator is proportional to
\(\prod_i(\alpha_i+1-z)_{n_i}\). Its coordinates are the signed
maximal minors. This proves \eqref{eq:Jacobi-terminal-minor-condition}.
If the reduced row index is near the diagonal, its square determinant
has the same nonzero product evaluation, proving the stated sufficient
condition.

For the zero counts, the row space at \(\boldsymbol m\) is AT by
Lemma~\ref{lem:Wolfs-near-diagonal-AT}, and the power space is AT at
every index. The sign-change argument for mixed orthogonality
\cite[Theorem~1]{FidalgoMedinaMinguez2013} therefore gives the asserted
counts. The smaller test spaces used in that argument can be chosen
along a nested sequence of near-diagonal indices ending at
\(\boldsymbol m\); the power-side test spaces are available at every
size. This use of AT does not require the reduced terminal indices used
in the separate maximal-degree question.
\end{proof}

We obtain the two mixed forms associated with the rectangular matrix
\eqref{eq:mixed-Jacobi-like-matrix}. The first form is expanded in the
power vector \(\mathbf u=(x^{\alpha_i})_{i=1}^p\), whereas the second form
is expanded in the Jacobi vector
\(\mathbf w=(w_j)_{j=1}^q\).

In the statements below, we use the notation \(\boldsymbol b^{\,*j}\) for the
vector obtained from \(\boldsymbol b\) by deleting its \(j\)-th component.
Proposition~\ref{prop:AT-perfect-mixed} explains the structural source of
normality whenever the two underlying systems are AT-systems. In the explicit
formulas below, the near-diagonal hypothesis provides the finite cancellations
needed in the contour argument for the form on the power vector and in the
rational Mellin reconstruction of the form on the Jacobi vector. All
displayed identities are identities of meromorphic functions of the parameters
wherever both sides are defined.

\begin{theorem}[Explicit mixed Jacobi forms]
	\label{thm:mixed-Jacobi-like-forms}
	Let
	\[
	\boldsymbol n=(n_1,\ldots,n_p)\in\mathbb N_0^p,
	\qquad
	\boldsymbol m=(m_1,\ldots,m_q)\in\mathbb N_0^q.
	\]
	
	\smallskip
	
	\noindent
	\textnormal{(i) The form on the power vector.}
	Assume that \(|\boldsymbol n|=|\boldsymbol m|+1\) and that \(\boldsymbol m\) is near the
	diagonal. Assume also the finite nonresonance condition that the shifted
	nodes
	\[
	\alpha_i+\ell,
	\qquad
	i\in\{1,\ldots,p\}\textnormal{ with }n_i\ge1,
	\quad
	\ell\in\{0,\ldots,n_i-1\},
	\]
	are pairwise distinct. Then there exists a mixed type~I form
	\[
	\mathcal A_{\boldsymbol n,\boldsymbol m}(x)
	=
	\sum_{i=1}^{p}A_{\boldsymbol n,\boldsymbol m}^{(i)}(x)x^{\alpha_i},
	\]
	where
	\[
	A_{\boldsymbol n,\boldsymbol m}^{(i)}\in\mathbb P_{n_i-1}
	\quad\textnormal{if }n_i\ge1,
	\qquad
	A_{\boldsymbol n,\boldsymbol m}^{(i)}\equiv0
	\quad\textnormal{if }n_i=0.
	\]
	This form satisfies
	\begin{equation}
		\label{eq:A-orthogonality}
		\int_0^1x^r w_j(x;\boldsymbol a,\boldsymbol b)\mathcal A_{\boldsymbol n,\boldsymbol m}(x)\dx=0,
		\qquad
		r\in\{0,\ldots,m_j-1\},
		\quad
		j\in\{1,\ldots,q\}\textnormal{ with }m_j\ge1.
	\end{equation}
	Define
	\[
	D_{\boldsymbol\alpha,\boldsymbol n}(t)
	\coloneq
	((t+1)\one_p-\boldsymbol\alpha-\boldsymbol n)_{\boldsymbol n},
	\qquad
	P_{\boldsymbol b,\boldsymbol m}(t)
	\coloneq
	((t+1)\one_q+\boldsymbol b)_{\boldsymbol m}.
	\]
	The following contour formula fixes a particular representative. Whenever
	the corresponding problem is normal, it is the normalized \(\mathcal A\)-form
	in the sense of Definition~\ref{def:normalization-convention}:
	\begin{equation}
		\label{eq:A-contour}
		\mathcal A_{\boldsymbol n,\boldsymbol m}(x)
		=
		\frac{(-1)^{|\boldsymbol n|}}{2\pi\mathrm{i}}
		\bigintsss_{\Sigma}
		\frac{P_{\boldsymbol b,\boldsymbol m}(t)}{D_{\boldsymbol\alpha,\boldsymbol n}(t)}
		\frac{\Gamma(t\one_q+\boldsymbol b+\one_q)}{\Gamma(t\one_q+\boldsymbol a+\one_q)}
		x^t\dt,
	\end{equation}
	where \(\Sigma\) encloses all nodes
	\(\alpha_i,\ldots,\alpha_i+n_i-1\), with \(n_i\ge1\), and no other
	pole of the integrand. Its polynomial components, for
	\(i\in\{1,\ldots,p\}\) with \(n_i\ge1\), are
	\begin{equation}
		\label{eq:A-components}
		A_{\boldsymbol n,\boldsymbol m}^{(i)}(x)
		=
		\kappa_i^{\mathrm A}
		\pFq{p+q}{p+q-1}
		{
			-n_i+1,\
			(\alpha_i+1)\one_q+\boldsymbol b+\boldsymbol m,\
			(\alpha_i+1)\one_{p-1}
			-\boldsymbol\alpha^{\,*i}-\boldsymbol n^{\,*i}
		}
		{
			(\alpha_i+1)\one_q+\boldsymbol a,\
			(\alpha_i+1)\one_{p-1}-\boldsymbol\alpha^{\,*i}
		}
		{x},
	\end{equation}
	where
	\begin{equation}
		\label{eq:A-kappa}
		\kappa_i^{\mathrm A}
		\coloneq
		(-1)^{|\boldsymbol n^{\,*i}|+1}
		\frac{\Gamma((\alpha_i+1)\one_q+\boldsymbol b)}
		{\Gamma((\alpha_i+1)\one_q+\boldsymbol a)}
		\frac{
			((\alpha_i+1)\one_q+\boldsymbol b)_{\boldsymbol m}
		}{
			(n_i-1)!
			((\alpha_i+1)\one_{p-1}-\boldsymbol\alpha^{\,*i}
			-\boldsymbol n^{\,*i})_{\boldsymbol n^{\,*i}}
		}.
	\end{equation}
	
	\smallskip
	
	\noindent
	\textnormal{(ii) The form on the Jacobi vector.}
	Assume that \(|\boldsymbol m|=|\boldsymbol n|+1\) and that \(\boldsymbol m\) is near the
	diagonal. Assume also the finite nonresonance conditions that the nodes
	\[
	-b_j-\ell,
	\qquad
	j\in\{1,\ldots,q\}\textnormal{ with }m_j\ge1,
	\quad
	\ell\in\{0,\ldots,m_j-1\},
	\]
	are pairwise distinct and that no denominator appearing in the finite
	reconstruction formula below vanishes. Then there exists a mixed type~II
	form
	\[
	\mathcal B_{\boldsymbol n,\boldsymbol m}(x)
	=
	\sum_{j=1}^{q}B_{\boldsymbol n,\boldsymbol m}^{(j)}(x)w_j(x;\boldsymbol a,\boldsymbol b),
	\]
	where
	\[
	B_{\boldsymbol n,\boldsymbol m}^{(j)}\in\mathbb P_{m_j-1}
	\quad\textnormal{if }m_j\ge1,
	\qquad
	B_{\boldsymbol n,\boldsymbol m}^{(j)}\equiv0
	\quad\textnormal{if }m_j=0.
	\]
	This form satisfies
	\begin{equation}
		\label{eq:B-orthogonality}
		\int_0^1\mathcal B_{\boldsymbol n,\boldsymbol m}(x)x^{\alpha_i+r}\dx=0,
		\qquad
		r\in\{0,\ldots,n_i-1\},
		\quad
		i\in\{1,\ldots,p\}\textnormal{ with }n_i\ge1.
	\end{equation}
	The following Mellin identity fixes a particular representative. Whenever
	the corresponding problem is normal, it is the normalized \(\mathcal B\)-form
	in the sense of Definition~\ref{def:normalization-convention}:
	\begin{equation}
		\label{eq:B-Mellin}
		\mathcal M[\mathcal B_{\boldsymbol n,\boldsymbol m}](s)
		=
		-
		\frac{\Gamma(s\one_q+\boldsymbol a)}{\Gamma(s\one_q+\boldsymbol b+\boldsymbol m)}
		(\boldsymbol\alpha+(1-s)\one_p)_{\boldsymbol n}.
	\end{equation}
	If, in addition,
	\[
	a_\rho-a_\sigma\notin\mathbb Z,
	\qquad \rho\ne\sigma,
	\]
	then the simple-pole expansion of the complete form is
	\begin{multline}
		\label{eq:B-complete-hypergeometric}
			\mathcal B_{\boldsymbol n,\boldsymbol m}(x)
			=
			-
			\sum_{j=1}^{q}
			\frac{
				\Gamma(\boldsymbol a^{\,*j}-a_j\one_{q-1})
				(\boldsymbol\alpha+(a_j+1)\one_p)_{\boldsymbol n}
			}{
				\Gamma(\boldsymbol b+\boldsymbol m-a_j\one_q)
			}
			x^{a_j}
			\\* \times
			\pFq{p+q}{p+q-1}
			{
				\boldsymbol\alpha+\boldsymbol n+(a_j+1)\one_p,\
				(a_j+1)\one_q-\boldsymbol b-\boldsymbol m
			}
			{
				\boldsymbol\alpha+(a_j+1)\one_p,\
				(a_j+1)\one_{q-1}-\boldsymbol a^{\,*j}
			}
			{x}.
	\end{multline}
	For every parameter choice covered above, the same complete form admits the
	Meijer \(G\)-representation
	\begin{equation}
		\label{eq:B-complete-Meijer-G}
		\mathcal B_{\boldsymbol n,\boldsymbol m}(x)
		=
		-
		G_{p+q,p+q}^{\,q,p}
		\left(
		x\,\middle|\,
		\begin{matrix}
			-\boldsymbol\alpha-\boldsymbol n,\ \boldsymbol b+\boldsymbol m\\
			\boldsymbol a,\ -\boldsymbol\alpha
		\end{matrix}
		\right),
		\qquad 0<x<1.
	\end{equation}
	At coalescent values of the \(a_\rho\),
	\eqref{eq:B-complete-hypergeometric} is understood by continuity, or
	equivalently by evaluating the higher-order residues of
	\eqref{eq:B-complete-Meijer-G}.
	
	Moreover, for \(j\in\{1,\ldots,q\}\) with \(m_j\ge1\), the polynomial
	components are given by the finite hypergeometric expansion
	\begin{align}
		\label{eq:B-components-hypergeometric}
		B_{\boldsymbol n,\boldsymbol m}^{(j)}(x)
		={}&
		\sum_{\substack{
				J\in\{1,\ldots,q\}\\
				K\in\{0,\ldots,m_J-1\}\\
				K-1+\delta_{j,J}\ge0
		}}
		\mathcal C_{j;J,K}^{\mathcal B}
		\pFq{q+1}{q}
		{
			1,\ \boldsymbol b+(1-b_J-K)\one_q-\boldsymbol e_j
		}
		{
			\boldsymbol a+(1-b_J-K)\one_q
		}
		{x},
	\end{align}
	where, for the indices appearing in
	\eqref{eq:B-components-hypergeometric},
		\begin{multline}
			\label{eq:B-components-coefficients}
			\mathcal C_{j;J,K}^{\mathcal B}
			\coloneq
			(-1)^{K+1}
			\frac{
				(\boldsymbol\alpha+(b_J+K+1)\one_p)_{\boldsymbol n}
			}{
				K!(m_J-K-1)!
				(\boldsymbol b^{\,*J}-(b_J+K)\one_{q-1})_{\boldsymbol m^{\,*J}}
			}
			\\*\times
			\frac{(\boldsymbol a-b_j\one_q)_1}
			{(\boldsymbol b^{\,*j}-b_j\one_{q-1})_1}
			\frac{
				(\boldsymbol b^{\,*j}-(b_J+K)\one_{q-1})_1
			}{
				(\boldsymbol a-(b_J+K)\one_q)_1
			}.
		\end{multline}
	Each hypergeometric series appearing in
	\eqref{eq:B-components-hypergeometric} is terminating, of degree
	\(K-1+\delta_{j,J}\).

\end{theorem}

\begin{proof}
	\textit{Step 1: Construction of the form on the power vector.}
	The only poles of the integrand in \eqref{eq:A-contour} enclosed by
	\(\Sigma\) are the simple poles of \(D_{\boldsymbol\alpha,\boldsymbol n}(t)^{-1}\). They
	are located at
	\[
	t=\alpha_i+\ell,
	\qquad
	i\in\{1,\ldots,p\}\textnormal{ with }n_i\ge1,
	\quad
	\ell\in\{0,\ldots,n_i-1\}.
	\]
	Evaluating the contour integral by residues gives
	\[
	\mathcal A_{\boldsymbol n,\boldsymbol m}(x)
	=
	\sum_{\substack{i=1\\ n_i\ge1}}^{p}
	\sum_{\ell=0}^{n_i-1}
	c_{i,\ell}x^{\alpha_i+\ell},
	\]
	where
	\[
	c_{i,\ell}
	\coloneq 
	(-1)^{|\boldsymbol n|}
	\operatorname*{Res}_{t=\alpha_i+\ell}
	\left[
	\frac{P_{\boldsymbol b,\boldsymbol m}(t)}{D_{\boldsymbol\alpha,\boldsymbol n}(t)}
	\frac{\Gamma(t\one_q+\boldsymbol b+\one_q)}{\Gamma(t\one_q+\boldsymbol a+\one_q)}
	\right].
	\]
	Equivalently,
	\[
	\mathcal A_{\boldsymbol n,\boldsymbol m}(x)
	=
	\sum_{\substack{i=1\\ n_i\ge1}}^{p}
	x^{\alpha_i}
	\sum_{\ell=0}^{n_i-1}c_{i,\ell}x^\ell.
	\]
	Thus the coefficient multiplying \(x^{\alpha_i}\) lies in
	\(\mathbb P_{n_i-1}\) for every \(i\in\{1,\ldots,p\}\), with the
	convention \(\mathbb P_{-1}=\{0\}\). Hence
	\(A_{\boldsymbol n,\boldsymbol m}^{(i)}\in\mathbb P_{n_i-1}\) when \(n_i\ge1\), and
	\(A_{\boldsymbol n,\boldsymbol m}^{(i)}\equiv0\) when \(n_i=0\).
	
	Let \(j\in\{1,\ldots,q\}\) with \(m_j\ge1\), and let
	\(r\in\{0,\ldots,m_j-1\}\). By \eqref{eq:wj-Mellin},
	\begin{equation*}
		\int_0^1x^r w_j(x;\boldsymbol a,\boldsymbol b)
		\mathcal A_{\boldsymbol n,\boldsymbol m}(x)\dx
		=
		\frac{(-1)^{|\boldsymbol n|}}{2\pi\mathrm{i}}
		\bigintsss_{\Sigma}
		\frac{P_{\boldsymbol b,\boldsymbol m}(t)}{D_{\boldsymbol\alpha,\boldsymbol n}(t)}
		\frac{(t\one_q+\boldsymbol a+\one_q)_r}
		{(t\one_q+\boldsymbol b+\one_q)_r(t+b_j+r+1)}
		\dt.
	\end{equation*}
	Since \(\boldsymbol m\) is near the diagonal, \(r\le m_h\) for every
	\(h\in\{1,\ldots,q\}\). Therefore \(P_{\boldsymbol b,\boldsymbol m}(t)\) cancels the
	factor \((t\one_q+\boldsymbol b+\one_q)_r\). Moreover, since \(r+1\le m_j\), it also cancels
	the remaining factor \(t+b_j+r+1\). The resulting integrand is rational and
	is \(\mathrm{O}(t^{-2})\) at infinity, because
	\(|\boldsymbol n|=|\boldsymbol m|+1\). Therefore its residue at infinity is zero.
	Equivalently, the sum of its finite residues vanishes. Since the contour
	\(\Sigma\) encloses precisely the poles \(t=\alpha_i+\ell\), this sum of
	finite residues is exactly the contour integral above. Hence the integral is
	zero, and \eqref{eq:A-orthogonality} follows.
	
	\textit{Step 2: Hypergeometric components on the power vector.}
	We evaluate the residues in the contour representation and show that, for
	each \(i\in\{1,\ldots,p\}\) with \(n_i\ge1\), the component
	\(A_{\boldsymbol n,\boldsymbol m}^{(i)}\) is a terminating generalized hypergeometric
	polynomial. Fix such an index \(i\). The coefficient of
	\(x^{\alpha_i+\ell}\) in \eqref{eq:A-contour} is
	\[
	c_{i,\ell}
	=
	(-1)^{|\boldsymbol n|}
	\frac{P_{\boldsymbol b,\boldsymbol m}(\alpha_i+\ell)}
	{D_{\boldsymbol\alpha,\boldsymbol n}'(\alpha_i+\ell)}
	\frac{\Gamma((\alpha_i+\ell)\one_q+\boldsymbol b+\one_q)}
	{\Gamma((\alpha_i+\ell)\one_q+\boldsymbol a+\one_q)}.
	\]
	For \(\ell=0\),
	\[
	D_{\boldsymbol\alpha,\boldsymbol n}'(\alpha_i)
	=
	(-1)^{n_i-1}(n_i-1)!
	(\alpha_i\one_{p-1}-\boldsymbol\alpha^{\,*i}-\boldsymbol n^{\,*i}+\one_{p-1})_{\boldsymbol n^{\,*i}},
	\]
	and therefore \(c_{i,0}=\kappa_i^{\mathrm A}\). For \(\ell\ge0\), direct use
	of Pochhammer identities gives
	\[
	\frac{c_{i,\ell}}{c_{i,0}}
	=
	\frac{
		(-n_i+1)_\ell
		(\alpha_i\one_q+\boldsymbol b+\boldsymbol m+\one_q)_\ell
		(\alpha_i\one_{p-1}-\boldsymbol\alpha^{\,*i}-\boldsymbol n^{\,*i}+\one_{p-1})_\ell
	}{
		\ell!
		((\alpha_i)\one_q+\boldsymbol a+\one_q)_\ell
		(\alpha_i\one_{p-1}-\boldsymbol\alpha^{\,*i}+\one_{p-1})_\ell
	}.
	\]
	Thus the residue expansion for the \(i\)-th component can be written as
	\[
	A_{\boldsymbol n,\boldsymbol m}^{(i)}(x)
	=
	c_{i,0}
	\sum_{\ell=0}^{n_i-1}
	\frac{
		(-n_i+1)_\ell
		(\alpha_i\one_q+\boldsymbol b+\boldsymbol m+\one_q)_\ell
		(\alpha_i\one_{p-1}-\boldsymbol\alpha^{\,*i}-\boldsymbol n^{\,*i}+\one_{p-1})_\ell
	}{
		((\alpha_i)\one_q+\boldsymbol a+\one_q)_\ell
		(\alpha_i\one_{p-1}-\boldsymbol\alpha^{\,*i}+\one_{p-1})_\ell
	}
	\frac{x^\ell}{\ell!}.
	\]
	This is exactly the terminating generalized hypergeometric series in
	\eqref{eq:A-components}, with \(c_{i,0}=\kappa_i^{\mathrm A}\). Hence
	\eqref{eq:A-components} and \eqref{eq:A-kappa} follow. In particular, for
	every \(i\in\{1,\ldots,p\}\) with \(n_i\ge1\), the component
	\(A_{\boldsymbol n,\boldsymbol m}^{(i)}\) is a single terminating polynomial of type
	\({}_{p+q}F_{p+q-1}\).
	
	\textit{Step 3: A finite reconstruction formula on the Jacobi vector.}
	We prove the elementary Mellin reconstruction formula used in
	Step~4. The goal is to realize each simple rational factor
	\[
	\frac{1}{s+b_J+K},
	\qquad
	J\in\{1,\ldots,q\},
	\quad
	K\in\{0,\ldots,m_J-1\},
	\]
	as the Mellin transform of a finite linear combination of the Jacobi
	weights with polynomial coefficients. For fixed
	\(J\in\{1,\ldots,q\}\) with \(m_J\ge1\), and
	\(K\in\{0,\ldots,m_J-1\}\), we construct polynomials
	\[
	\mathcal R_{j;J,K}(x)\in\mathbb P_{m_j-1},
	\qquad
	j\in\{1,\ldots,q\},
	\]
	such that the following elementary Mellin reconstruction identity holds:
	\begin{equation}
		\label{eq:Jacobi-elementary-Mellin-reconstruction}
		\mathcal M\left[
		\sum_{j=1}^{q}
		\mathcal R_{j;J,K}(\cdot)w_j(\cdot;\boldsymbol a,\boldsymbol b)
		\right](s)
		=
		\frac{\Gamma(s\one_q+\boldsymbol a)}{\Gamma(s\one_q+\boldsymbol b)}
		\frac{1}{s+b_J+K}.
	\end{equation}
	By linearity, any finite partial-fraction expansion whose poles are simple and
	belong to the set of nodes \(-b_J-K\), \(J\in\{1,\ldots,q\}\),
	\(K\in\{0,\ldots,m_J-1\}\), can then be reconstructed by taking the
	corresponding linear combination of these elementary polynomials.
	
	By \eqref{eq:wj-Mellin}, the Mellin transform of
	\(x^\ell w_j(x;\boldsymbol a,\boldsymbol b)\) is
	\[
	\mathcal M[x^\ell w_j(\cdot;\boldsymbol a,\boldsymbol b)](s)
	=
	\frac{\Gamma(s\one_q+\boldsymbol a)}{\Gamma(s\one_q+\boldsymbol b)}
	\frac{(s\one_q+\boldsymbol a)_\ell}
	{(s\one_q+\boldsymbol b)_\ell(s+b_j+\ell)}.
	\]
	Thus, after factoring out the common quotient
	\(\Gamma(s\one_q+\boldsymbol a)/\Gamma(s\one_q+\boldsymbol b)\), it remains to prove a rational
	identity in the variable \(s\).
	
	For fixed \(J\in\{1,\ldots,q\}\) with \(m_J\ge1\), and
	\(K\in\{0,\ldots,m_J-1\}\), define
	\begin{equation}
		\label{eq:Jacobi-elementary-reconstructed-polynomial}
		\mathcal R_{j;J,K}(x)
		\coloneq
		\frac{(\boldsymbol a-b_j\one_q)_1}{(\boldsymbol b^{\,*j}-b_j\one_{q-1})_1}
		\sum_{\ell=0}^{K-1+\delta_{j,J}}
		\frac{
			(\boldsymbol b^{\,*j}-(b_J+K)\one_{q-1})_{\ell+1}
			(b_j-b_J-K)_\ell
		}{
			(\boldsymbol a-(b_J+K)\one_q)_{\ell+1}
		}
		x^\ell,
	\end{equation}
	where the convention is that a sum with upper limit \(-1\) is empty. Since
	\(\boldsymbol m\) is near the diagonal, \(0\le K\le m_J-1\) implies
	\[
	K-1+\delta_{j,J}\le m_j-1.
	\]
	Hence \(\mathcal R_{j;J,K}\in\mathbb P_{m_j-1}\).
	
	We claim that these polynomials satisfy
	\eqref{eq:Jacobi-elementary-Mellin-reconstruction}. Using the Mellin transform
	of \(x^\ell w_j\), this claim is equivalent to the rational identity
	\begin{equation}
		\label{eq:Jacobi-rational-reconstruction-identity}
		\sum_{j=1}^{q}
		\sum_{\ell=0}^{K-1+\delta_{j,J}}
		\tau_{j;J,K,\ell}
		\frac{(s\one_q+\boldsymbol a)_\ell}
		{(s\one_q+\boldsymbol b)_\ell(s+b_j+\ell)}
		=
		\frac{1}{s+b_J+K},
	\end{equation}
	where
	\[
	\tau_{j;J,K,\ell}
	=
	\frac{(\boldsymbol a-b_j\one_q)_1}{(\boldsymbol b^{\,*j}-b_j\one_{q-1})_1}
	\frac{
		(\boldsymbol b^{\,*j}-(b_J+K)\one_{q-1})_{\ell+1}
		(b_j-b_J-K)_\ell
	}{
		(\boldsymbol a-(b_J+K)\one_q)_{\ell+1}
	}.
	\]
	We prove \eqref{eq:Jacobi-rational-reconstruction-identity} by decomposing
	the single fraction \(1/(s+B)\) in \(K\) successive steps. Put
	\(B\coloneq b_J+K\). The nonresonance assumptions imply that
	\(B\neq b_j+\ell\), for \(j\in\{1,\ldots,q\}\) and
	\(\ell\in\{0,\ldots,K-1\}\), and that all poles appearing below are simple.
	If \(K=0\), then the left-hand side of
	\eqref{eq:Jacobi-rational-reconstruction-identity} contains only the term
	\(j=J\), \(\ell=0\), and \(\tau_{J;J,0,0}=1\). Hence the identity is
	immediate. We assume from now on that \(K\ge1\).
	
	The successive steps are arranged as follows. At step \(\ell\), with
	\(\ell\in\{0,\ldots,K-1\}\), the part of \(1/(s+B)\) not yet decomposed is
	written as the sum of the fractions with poles
	\[
	s=-b_j-\ell,
	\qquad j\in\{1,\ldots,q\},
	\]
	and a new part to be decomposed at the next step. For
	\(\ell\ge0\), set
	\[
	E_\ell(s)
	\coloneq
	\frac{(\boldsymbol b-B\one_q)_\ell}{(\boldsymbol a-B\one_q)_\ell}
	\frac{(s\one_q+\boldsymbol a)_\ell}{(s\one_q+\boldsymbol b)_\ell}.
	\]
	Then \(E_0(s)=1\), so the procedure starts with
	\[
	\frac{E_0(s)}{s+B}
	=
	\frac{1}{s+B}.
	\]
	At step \(\ell\), the part not yet decomposed is
	\[
	\frac{E_\ell(s)}{s+B},
	\]
	and the next part to be decomposed is
	\[
	\frac{E_{\ell+1}(s)}{s+B}.
	\]
	Thus it is enough to prove the transition from step \(\ell\) to step
	\(\ell+1\). After the \(K\)-th step, the part not yet decomposed is
	\(E_K(s)/(s+B)\).
	
	For \(\ell\in\{0,\ldots,K-1\}\), set
	\[
	\tau_{j,\ell}
	\coloneq
	\frac{(\boldsymbol a-b_j\one_q)_1}{(\boldsymbol b^{\,*j}-b_j\one_{q-1})_1}
	\frac{
		(\boldsymbol b^{\,*j}-B\one_{q-1})_{\ell+1}
		(b_j-B)_\ell
	}{
		(\boldsymbol a-B\one_q)_{\ell+1}
	},
	\qquad
	j\in\{1,\ldots,q\}.
	\]
	Since \(B=b_J+K\), we have
	\(\tau_{j,\ell}=\tau_{j;J,K,\ell}\). We claim that
	\begin{equation}
		\label{eq:Jacobi-reconstruction-transition-identity}
		\frac{E_\ell(s)}{s+B}
		=
		\sum_{j=1}^{q}
		\tau_{j,\ell}
		\frac{(s\one_q+\boldsymbol a)_\ell}
		{(s\one_q+\boldsymbol b)_\ell(s+b_j+\ell)}
		+
		\frac{E_{\ell+1}(s)}{s+B}.
	\end{equation}
	This is the transition identity at step \(\ell\).
	
	To prove \eqref{eq:Jacobi-reconstruction-transition-identity}, divide both
	sides by \((s\one_q+\boldsymbol a)_\ell/(s\one_q+\boldsymbol b)_\ell\). Since
	\[
	\frac{E_\ell(s)}
	{(s\one_q+\boldsymbol a)_\ell/(s\one_q+\boldsymbol b)_\ell}
	=
	\frac{(\boldsymbol b-B\one_q)_\ell}
	{(\boldsymbol a-B\one_q)_\ell},
	\]
	and
	\[
	\frac{E_{\ell+1}(s)}
	{(s\one_q+\boldsymbol a)_\ell/(s\one_q+\boldsymbol b)_\ell}
	=
	\frac{(\boldsymbol b-B\one_q)_{\ell+1}}
	{(\boldsymbol a-B\one_q)_{\ell+1}}
	\frac{(s\one_q+\boldsymbol a+\ell\one_q)_1}
	{(s\one_q+\boldsymbol b+\ell\one_q)_1},
	\]
	the required identity is equivalent to
	\begin{equation}
		\label{eq:Jacobi-reduced-reconstruction-transition-identity}
		\frac{(\boldsymbol b-B\one_q)_\ell}
		{(\boldsymbol a-B\one_q)_\ell}
		\frac{1}{s+B}
		=
		\sum_{j=1}^{q}
		\frac{\tau_{j,\ell}}{s+b_j+\ell}
		+
		\frac{(\boldsymbol b-B\one_q)_{\ell+1}}
		{(\boldsymbol a-B\one_q)_{\ell+1}}
		\frac{(s\one_q+\boldsymbol a+\ell\one_q)_1}
		{(s\one_q+\boldsymbol b+\ell\one_q)_1}
		\frac{1}{s+B}.
	\end{equation}
	We verify this identity by comparing principal parts.
	
	At \(s=-B\), the terms in the sum over \(j\) have no pole, because
	\(B\neq b_j+\ell\). The residue of the last term on the right-hand side of
	\eqref{eq:Jacobi-reduced-reconstruction-transition-identity} is
		\begin{align*}
		\frac{(\boldsymbol b-B\one_q)_{\ell+1}}
		{(\boldsymbol a-B\one_q)_{\ell+1}}
		\frac{((-B)\one_q+\boldsymbol a+\ell\one_q)_1}
		{((-B)\one_q+\boldsymbol b+\ell\one_q)_1}
		&=
		\frac{(\boldsymbol b-B\one_q)_\ell}
		{(\boldsymbol a-B\one_q)_\ell}
		\frac{
			(\boldsymbol b-B\one_q+\ell\one_q)_1
		}{
			(\boldsymbol a-B\one_q+\ell\one_q)_1
		}
		\frac{
			(\boldsymbol a-B\one_q+\ell\one_q)_1
		}{
			(\boldsymbol b-B\one_q+\ell\one_q)_1
		}
		\\
		&=
		\frac{(\boldsymbol b-B\one_q)_\ell}
		{(\boldsymbol a-B\one_q)_\ell}.
		\end{align*}
	which is the residue of the left-hand side.
	
	Now fix \(j\in\{1,\ldots,q\}\). At \(s=-b_j-\ell\), the left-hand side of
	\eqref{eq:Jacobi-reduced-reconstruction-transition-identity} has no pole,
	because \(B\neq b_j+\ell\). On the right-hand side, only two terms may have
	a pole: the \(j\)-th simple fraction in the sum and the last term. The
	residue of the latter is
	\[
	\frac{(\boldsymbol b-B\one_q)_{\ell+1}}
	{(\boldsymbol a-B\one_q)_{\ell+1}}
	\frac{1}{B-b_j-\ell}
	\operatorname*{Res}_{s=-b_j-\ell}
	\frac{(s\one_q+\boldsymbol a+\ell\one_q)_1}
	{(s\one_q+\boldsymbol b+\ell\one_q)_1}
	=
	\frac{(\boldsymbol b-B\one_q)_{\ell+1}}
	{(\boldsymbol a-B\one_q)_{\ell+1}}
	\frac{1}{B-b_j-\ell}
	\frac{(\boldsymbol a-b_j\one_q)_1}
	{(\boldsymbol b^{\,*j}-b_j\one_{q-1})_1}.
	\]
	Using
	\[
	(\boldsymbol b-B\one_q)_{\ell+1}
	=
	(\boldsymbol b^{\,*j}-B\one_{q-1})_{\ell+1}
	(b_j-B)_{\ell+1},
	\qquad
	\frac{(b_j-B)_{\ell+1}}{B-b_j-\ell}
	=
	-(b_j-B)_\ell,
	\]
	this residue becomes
	\[
	-
	\frac{(\boldsymbol a-b_j\one_q)_1}
	{(\boldsymbol b^{\,*j}-b_j\one_{q-1})_1}
	\frac{
		(\boldsymbol b^{\,*j}-B\one_{q-1})_{\ell+1}
		(b_j-B)_\ell
	}{
		(\boldsymbol a-B\one_q)_{\ell+1}
	}
	=
	-\tau_{j,\ell}.
	\]
	This cancels the residue of \(\tau_{j,\ell}/(s+b_j+\ell)\). Thus both sides
	of \eqref{eq:Jacobi-reduced-reconstruction-transition-identity} have the same
	principal part at every finite pole. Their difference is therefore a
	polynomial. Since both sides are \(\mathrm{O}(s^{-1})\) as \(s\to\infty\),
	that polynomial is zero. This proves
	\eqref{eq:Jacobi-reduced-reconstruction-transition-identity}, and hence
	\eqref{eq:Jacobi-reconstruction-transition-identity}.
	
	We apply \eqref{eq:Jacobi-reconstruction-transition-identity}
	successively for \(\ell\in\{0,\ldots,K-1\}\). When these \(K\) identities
	are added, the terms
	\[
	\frac{E_{\ell+1}(s)}{s+B},
	\qquad \ell\in\{0,\ldots,K-2\},
	\]
	cancel with the same terms appearing on the left-hand side of the next
	identity. Since \(E_0(s)=1\), the sum of the identities gives
	\begin{equation}
		\label{eq:Jacobi-rational-identity-before-terminal-term}
		\frac{1}{s+B}
		=
		\sum_{j=1}^{q}
		\sum_{\ell=0}^{K-1}
		\tau_{j;J,K,\ell}
		\frac{(s\one_q+\boldsymbol a)_\ell}
		{(s\one_q+\boldsymbol b)_\ell(s+b_j+\ell)}
		+
		\frac{E_K(s)}{s+B}.
	\end{equation}
	
	It remains to identify the last term. By definition,
	\[
	\frac{E_K(s)}{s+B}
	=
	\frac{(\boldsymbol b-B\one_q)_K}
	{(\boldsymbol a-B\one_q)_K}
	\frac{(s\one_q+\boldsymbol a)_K}
	{(s\one_q+\boldsymbol b)_K(s+B)}.
	\]
	Since \(B=b_J+K\), we claim that
	\[
	\tau_{J;J,K,K}
	=
	\frac{(\boldsymbol b-B\one_q)_K}
	{(\boldsymbol a-B\one_q)_K}.
	\]
	Indeed,
	\begin{align*}
		\tau_{J;J,K,K}
		&=
		\frac{(\boldsymbol a-b_J\one_q)_1}
		{(\boldsymbol b^{\,*J}-b_J\one_{q-1})_1}
		\frac{
			(\boldsymbol b^{\,*J}-B\one_{q-1})_{K+1}
			(b_J-B)_K
		}{
			(\boldsymbol a-B\one_q)_{K+1}
		}
		\\
		&=
		\frac{(\boldsymbol a-b_J\one_q)_1}
		{(\boldsymbol b^{\,*J}-b_J\one_{q-1})_1}
		\frac{
			(\boldsymbol b^{\,*J}-B\one_{q-1})_K
			(\boldsymbol b^{\,*J}-b_J\one_{q-1})_1
			(b_J-B)_K
		}{
			(\boldsymbol a-B\one_q)_K
			(\boldsymbol a-b_J\one_q)_1
		}
		\\
		&=
		\frac{
			(\boldsymbol b^{\,*J}-B\one_{q-1})_K
			(b_J-B)_K
		}{
			(\boldsymbol a-B\one_q)_K}
		=
		\frac{(\boldsymbol b-B\one_q)_K}
		{(\boldsymbol a-B\one_q)_K}.
	\end{align*}
	Therefore,
	\[
	\frac{E_K(s)}{s+B}
	=
	\tau_{J;J,K,K}
	\frac{(s\one_q+\boldsymbol a)_K}
	{(s\one_q+\boldsymbol b)_K(s+b_J+K)}.
	\]
	Replacing the last term in
	\eqref{eq:Jacobi-rational-identity-before-terminal-term} by this expression
	gives \eqref{eq:Jacobi-rational-reconstruction-identity}. This proves
	\eqref{eq:Jacobi-rational-reconstruction-identity}, and therefore
	\eqref{eq:Jacobi-elementary-Mellin-reconstruction}.
	
	\textit{Step 4: Construction and components of the form on the Jacobi vector.}
	We apply the elementary reconstruction formula of Step~3 to the rational
	function required by the form on the Jacobi vector. The orthogonality
	conditions will be imposed by forcing zeros of the Mellin transform at the
	nodes \(s=\alpha_i+r+1\). Consider
	\[
	R_{\boldsymbol n,\boldsymbol m}^{\mathcal B}(s)
	\coloneq
	-
	\frac{
		(\boldsymbol\alpha+(1-s)\one_p)_{\boldsymbol n}
	}{
		(s\one_q+\boldsymbol b)_{\boldsymbol m}
	}.
	\]
	The denominator has simple poles at the nodes \(s=-b_J-K\), with
	\(J\in\{1,\ldots,q\}\), \(m_J\ge1\), and
	\(K\in\{0,\ldots,m_J-1\}\). Hence
	\[
	R_{\boldsymbol n,\boldsymbol m}^{\mathcal B}(s)
	=
	\sum_{\substack{J=1\\ m_J\ge1}}^{q}
	\sum_{K=0}^{m_J-1}
	\frac{\pi_{J,K}^{\mathcal B}}{s+b_J+K},
	\]
	where
	\[
	\pi_{J,K}^{\mathcal B}
	=
	\lim_{s\to -b_J-K}
	(s+b_J+K)R_{\boldsymbol n,\boldsymbol m}^{\mathcal B}(s).
	\]
	A direct evaluation gives
	\[
	\pi_{J,K}^{\mathcal B}
	=
	(-1)^{K+1}
	\frac{
		(\boldsymbol\alpha+(b_J+K+1)\one_p)_{\boldsymbol n}
	}{
		K!(m_J-K-1)!
		(\boldsymbol b^{\,*J}-(b_J+K)\one_{q-1})_{\boldsymbol m^{\,*J}}
	}.
	\]
	
	For \(j\in\{1,\ldots,q\}\), define
	\[
	B_{\boldsymbol n,\boldsymbol m}^{(j)}(x)
	\coloneq
	\sum_{\substack{J=1\\ m_J\ge1}}^{q}
	\sum_{K=0}^{m_J-1}
	\pi_{J,K}^{\mathcal B}\,
	\mathcal R_{j;J,K}(x),
	\]
	where the elementary reconstructed polynomials
	\(\mathcal R_{j;J,K}\) are those defined in
	\eqref{eq:Jacobi-elementary-reconstructed-polynomial}. Since
	\(\mathcal R_{j;J,K}\in\mathbb P_{m_j-1}\), we have
	\(B_{\boldsymbol n,\boldsymbol m}^{(j)}\in\mathbb P_{m_j-1}\). Now set
	\[
	\mathcal B_{\boldsymbol n,\boldsymbol m}(x)
	\coloneq
	\sum_{j=1}^{q}
	B_{\boldsymbol n,\boldsymbol m}^{(j)}(x)w_j(x;\boldsymbol a,\boldsymbol b).
	\]
	By \eqref{eq:Jacobi-elementary-Mellin-reconstruction} and linearity,
	\[
	\mathcal M[\mathcal B_{\boldsymbol n,\boldsymbol m}](s)
	=
	\frac{\Gamma(s\one_q+\boldsymbol a)}{\Gamma(s\one_q+\boldsymbol b)}
	R_{\boldsymbol n,\boldsymbol m}^{\mathcal B}(s),
	\]
	that is,
	\[
	\mathcal M[\mathcal B_{\boldsymbol n,\boldsymbol m}](s)
	=
	-
	\frac{\Gamma(s\one_q+\boldsymbol a)}{\Gamma(s\one_q+\boldsymbol b+\boldsymbol m)}
	(\boldsymbol\alpha+(1-s)\one_p)_{\boldsymbol n},
	\]
	which is \eqref{eq:B-Mellin}.
	
	We also obtain the Meijer \(G\)-representation of the complete form from
	this Mellin transform. Using
	\[
	(\alpha_i+1-s)_{n_i}
	=
	\frac{\Gamma(\alpha_i+n_i+1-s)}
	{\Gamma(\alpha_i+1-s)},
	\]
	the Mellin transform \eqref{eq:B-Mellin} can be written as
	\begin{equation}
		\label{eq:B-Mellin-gamma-quotient}
		\mathcal M[\mathcal B_{\boldsymbol n,\boldsymbol m}](s)
		=
		-
		\frac{
			\Gamma(s\one_q+\boldsymbol a)
			\Gamma(\boldsymbol\alpha+\boldsymbol n+(1-s)\one_p)
		}{
			\Gamma(s\one_q+\boldsymbol b+\boldsymbol m)
			\Gamma(\boldsymbol\alpha+(1-s)\one_p)
		}.
	\end{equation}
	Since the form \(\mathcal B_{\boldsymbol n,\boldsymbol m}\) is represented on \((0,1)\), this
	Mellin transform may be read as a Mellin transform on \(\mathbb R_+\).
	With the Meijer \(G\)-function convention fixed in
	\eqref{eq:Meijer-G-definition}--\eqref{eq:Meijer-G-Mellin-transform},
	\eqref{eq:B-Mellin-gamma-quotient} is the Mellin transform of
	\[
	-
	G_{p+q,p+q}^{\,q,p}
	\left(
	x\,\middle|\,
	\begin{matrix}
		-\boldsymbol\alpha-\boldsymbol n,\ \boldsymbol b+\boldsymbol m\\
		\boldsymbol a,\ -\boldsymbol\alpha
	\end{matrix}
	\right).
	\]
	This proves \eqref{eq:B-complete-Meijer-G}.
	
	We derive the displayed formula for the polynomial components. Substituting
	\eqref{eq:Jacobi-elementary-reconstructed-polynomial} into the definition of
	\(B_{\boldsymbol n,\boldsymbol m}^{(j)}\), we get
	\[
	B_{\boldsymbol n,\boldsymbol m}^{(j)}(x)
	=
	\frac{(\boldsymbol a-b_j\one_q)_1}{(\boldsymbol b^{\,*j}-b_j\one_{q-1})_1}
	\sum_{\substack{J=1\\ m_J\ge1}}^{q}
	\sum_{K=0}^{m_J-1}
	\pi_{J,K}^{\mathcal B}
	\sum_{\ell=0}^{K-1+\delta_{j,J}}
	\frac{
		(\boldsymbol b^{\,*j}-(b_J+K)\one_{q-1})_{\ell+1}
		(b_j-b_J-K)_\ell
	}{
		(\boldsymbol a-(b_J+K)\one_q)_{\ell+1}
	}
	x^\ell.
	\]
	If \(K=0\) and \(j\neq J\), the inner sum is empty. Otherwise, extracting
	the \(\ell=0\) factor from the inner sum gives
	\[
	\frac{
		(\boldsymbol b^{\,*j}-(b_J+K)\one_{q-1})_1
	}{
		(\boldsymbol a-(b_J+K)\one_q)_1
	}
	\sum_{\ell=0}^{K-1+\delta_{j,J}}
	\frac{
		(\boldsymbol b^{\,*j}+(1-b_J-K)\one_{q-1})_\ell
		(b_j-b_J-K)_\ell
	}{
		(\boldsymbol a+(1-b_J-K)\one_q)_\ell
	}
	x^\ell.
	\]
	Since
	\[
	(\boldsymbol b^{\,*j}-(b_J+K)\one_{q-1}+\one_{q-1})_\ell
	(b_j-b_J-K)_\ell
	=
		(\boldsymbol b+(1-b_J-K)\one_q-\boldsymbol e_j)_\ell
	\]
	and \((1)_\ell/\ell!=1\), this finite sum is
	\[
	\pFq{q+1}{q}
	{
		1,\ \boldsymbol b+(1-b_J-K)\one_q-\boldsymbol e_j
	}
	{
		\boldsymbol a+(1-b_J-K)\one_q
	}
	{x}.
	\]
	Combining this expression with the formula for
	\(\pi_{J,K}^{\mathcal B}\) gives exactly
	\eqref{eq:B-components-hypergeometric} and
	\eqref{eq:B-components-coefficients}. Notice that the case
	\(K=0\), \(j\neq J\), does not contribute, because the corresponding
	reconstructed polynomial \(\mathcal R_{j;J,0}\) is identically zero.
	
	If \(s=\alpha_i+r+1\), with \(n_i\ge1\) and \(0\le r\le n_i-1\), then
	\[
	(\alpha_i+1-s)_{n_i}=(-r)_{n_i}=0.
	\]
	Thus \eqref{eq:B-orthogonality} follows from \eqref{eq:B-Mellin}. Finally,
	Mellin inversion applied to \eqref{eq:B-Mellin}, or equivalently to
	\eqref{eq:B-Mellin-gamma-quotient}, followed by residue evaluation at the
	poles
	\[
	s=-a_j-k,
	\qquad
	j\in\{1,\ldots,q\},
	\quad
	k\in\mathbb N_0,
	\]
	gives \eqref{eq:B-complete-hypergeometric}. Under the additional
	nonresonance condition on the \(a_\rho\), these poles are simple, and collecting the
	residues with fixed \(j\) gives exactly the \(j\)-th hypergeometric term
	in \eqref{eq:B-complete-hypergeometric}. Thus
	\eqref{eq:B-complete-hypergeometric} is the residue expansion of the
	Meijer \(G\)-function in \eqref{eq:B-complete-Meijer-G}. When poles
	coalesce, the Meijer \(G\)-identity remains valid and the same expansion is
	obtained by continuity, equivalently from the corresponding higher-order
	residues.
\end{proof}

The finite expression in
\eqref{eq:B-components-hypergeometric} is directly computable.  Its terms can
also be grouped according to the pole family \(J\), and each group is then a
terminating Kamp\'e de F\'eriet polynomial.  This reformulation is not needed
for the orthogonality proof, but it gives a compact hypergeometric form of each
component.  The diagonal family \(J=j\) and the families \(J\ne j\) must be
kept separate, since their first arguments are different.

\begin{corollary}[Kamp\'e de F\'eriet form of the \(B\)-components]
	\label{cor:Jacobi-B-KdF}
	Under the assumptions of
	Theorem~\ref{thm:mixed-Jacobi-like-forms}\textnormal{(ii)}, let
	\(j\in\{1,\ldots,q\}\).  If \(m_j=0\), then
	\(B_{\boldsymbol n,\boldsymbol m}^{(j)}\equiv0\).  If \(m_j\ge1\), then
	\begin{equation}
		\label{eq:Jacobi-B-KdF-representation}
		B_{\boldsymbol n,\boldsymbol m}^{(j)}(x)
		=
		\mathcal K_{j,j}^{\mathcal B}\mathcal F_j(x)
		+
		\sum_{\substack{J=1\\J\ne j,\ m_J\ge2}}^{q}
		\mathcal K_{j,J}^{\mathcal B}\mathcal F_{j,J}(x),
	\end{equation}
	where
	\begin{equation}
		\label{eq:Jacobi-B-KdF-diagonal-prefactor}
		\mathcal K_{j,j}^{\mathcal B}
		\coloneq
		-
		\frac{
			(\boldsymbol\alpha+(b_j+1)\one_p)_{\boldsymbol n}
		}{
			(m_j-1)!
			(\boldsymbol b^{\,*j}-b_j\one_{q-1})_{\boldsymbol m^{\,*j}}
		},
	\end{equation}
	and
	\begingroup
	\small
	\begin{equation}
		\label{eq:Jacobi-B-KdF-diagonal}
		\mathcal F_j(x)
		\coloneq
		F_{p+q:0;q-1}^{p+q:1;q}
		\left[
		\begin{array}{c}
			1-m_j,\,
			\boldsymbol\alpha+(b_j+1)\one_p+\boldsymbol n,\,
			b_j\one_{q-1}-\boldsymbol b^{\,*j}
			+\one_{q-1}-\boldsymbol m^{\,*j}
			:1;\,
			b_j\one_q-\boldsymbol a
			\\
			\boldsymbol\alpha+(b_j+1)\one_p,\,
			b_j\one_q-\boldsymbol a+\one_q
			:\text{---};\,
			b_j\one_{q-1}-\boldsymbol b^{\,*j}
		\end{array}
		\middle|x,1
		\right].
	\end{equation}
	\endgroup
	For \(J\ne j\) with \(m_J\ge2\),
	\begin{equation}
		\label{eq:Jacobi-B-KdF-offdiagonal-prefactor}
		\mathcal K_{j,J}^{\mathcal B}
		\coloneq
		\frac{1}{(m_J-2)!}
		\frac{
			(\boldsymbol\alpha+(b_J+2)\one_p)_{\boldsymbol n}
		}{
			(\boldsymbol b^{\,*J}-(b_J+1)\one_{q-1})_{\boldsymbol m^{\,*J}}
		}
		\frac{
			(\boldsymbol a-b_j\one_q)_1
		}{
			(\boldsymbol b^{\,*j}-b_j\one_{q-1})_1
		}
		\frac{
			(\boldsymbol b^{\,*j}-(b_J+1)\one_{q-1})_1
		}{
			(\boldsymbol a-(b_J+1)\one_q)_1
		},
	\end{equation}
	while
	\begin{equation}
		\label{eq:Jacobi-B-KdF-offdiagonal}
		\resizebox{\textwidth}{!}{$
		\begin{multlined}[t]
			\mathcal F_{j,J}(x)
			\coloneq
			F_{p+q:0;q-1}^{p+q:1;q}
			\left[
			\begin{array}{c}
				2-m_J,\,
				\boldsymbol\alpha+(b_J+2)\one_p+\boldsymbol n,\,
				b_J\one_{q-1}-\boldsymbol b^{\,*J}
				+2\one_{q-1}-\boldsymbol m^{\,*J}
				:1;\,
				b_J\one_q-\boldsymbol a+\one_q
				\\
				\boldsymbol\alpha+(b_J+2)\one_p,\,
				b_J\one_q-\boldsymbol a+2\one_q
				:\text{---};\,
				b_J\one_{q-1}-\boldsymbol b^{\,*J}
				+\one_{q-1}+\boldsymbol e_j^{\,*J}
			\end{array}
			\middle|x,1
			\right].
		\end{multlined}
		$}
	\end{equation}
	Here \(\boldsymbol e_j^{\,*J}\) is obtained from the \(j\)-th coordinate
	vector \(\boldsymbol e_j\in\mathbb R^q\) by deleting its \(J\)-th component.
	The parameter \(1-m_j\) restricts \(\mathcal F_j\) to
	\(u+\lambda\le m_j-1\), whereas \(2-m_J\) restricts
	\(\mathcal F_{j,J}\) to \(u+\lambda\le m_J-2\).
\end{corollary}

\begin{proof}
	Fix \(j\) and expand every terminating \({}_{q+1}F_q\) in
	\eqref{eq:B-components-hypergeometric}.  We treat the pole family
	\(J=j\) first.  Write
	\[
		K=u+\lambda,
		\qquad
		u,\lambda\in\mathbb N_0,
		\qquad
		u+\lambda\le m_j-1.
	\]
	Here \(u\) is the summation index of the hypergeometric polynomial and
	\(\lambda=K-u\).  Thus this change of indices includes the term \(K=0\)
	at \(u=\lambda=0\).  Put \(s=u+\lambda\).  From
	\eqref{eq:B-components-coefficients},
	\begin{equation}
		\label{eq:Jacobi-B-KdF-diagonal-coefficient-ratio}
		\frac{\mathcal C_{j;j,s}^{\mathcal B}}
		{\mathcal C_{j;j,0}^{\mathcal B}}
		=
		\frac{(1-m_j)_s}{s!}
		\frac{
			(\boldsymbol\alpha+(b_j+1)\one_p+\boldsymbol n)_s
		}{
			(\boldsymbol\alpha+(b_j+1)\one_p)_s
		}
		\frac{
			(b_j\one_{q-1}-\boldsymbol b^{\,*j}
			+\one_{q-1}-\boldsymbol m^{\,*j})_s
		}{
			(b_j\one_{q-1}-\boldsymbol b^{\,*j}
			+\one_{q-1})_s
		}.
	\end{equation}
	The \(j\)-th upper parameter of the hypergeometric polynomial is
	\(-s\), and therefore
	\[
		(-s)_u=(-1)^u\frac{s!}{\lambda!}.
	\]
	Using, componentwise,
	\[
		(a+s)_n=(a)_n\frac{(a+n)_s}{(a)_s},
		\qquad
		(a-s)_n=(a)_n\frac{(1-a)_s}{(1-a-n)_s},
	\]
	in \eqref{eq:Jacobi-B-KdF-diagonal-coefficient-ratio} and in the
	expanded hypergeometric term, the summand indexed by
	\((u,\lambda)\) becomes
	\begin{multline}
		\label{eq:Jacobi-B-KdF-diagonal-summand}
			\mathcal K_{j,j}^{\mathcal B}
			\frac{
				(1-m_j)_{u+\lambda}
				(\boldsymbol\alpha+(b_j+1)\one_p+\boldsymbol n)_{u+\lambda}
			}{
				(\boldsymbol\alpha+(b_j+1)\one_p)_{u+\lambda}
			}
			\\*
			\times
			\frac{
				(b_j\one_{q-1}-\boldsymbol b^{\,*j}
				+\one_{q-1}-\boldsymbol m^{\,*j})_{u+\lambda}
			}{
				(b_j\one_q-\boldsymbol a+\one_q)_{u+\lambda}
			}
			\frac{
				(1)_u(b_j\one_q-\boldsymbol a)_\lambda
			}{
				(b_j\one_{q-1}-\boldsymbol b^{\,*j})_\lambda
			}
			\frac{x^u}{u!}\frac{1}{\lambda!}.
	\end{multline}
	To see the cancellations explicitly, the \(j\)-th entry of
	\(b_j\one_q-\boldsymbol b+\boldsymbol e_j\) is \(1\), and its remaining entries form
	\(b_j\one_{q-1}-\boldsymbol b^{\,*j}\).  These entries cancel the corresponding
	scalar and \((q-1)\)-strings produced by the two shift identities.
	The common string \(b_j\one_q-\boldsymbol a\) cancels in the coupled block
	and remains in the \(\lambda\)-dependent numerator.  The signs from
	the reflected factors cancel in this diagonal case.  Consequently,
	\eqref{eq:Jacobi-B-KdF-diagonal-summand} is exactly the
	\((u,\lambda)\)-term of
	\(\mathcal K_{j,j}^{\mathcal B}\mathcal F_j(x)\). Its first
	argument is \(x\).

	Now let \(J\ne j\).  Write
	\[
		K=1+u+\lambda,
		\qquad
		u,\lambda\in\mathbb N_0,
		\qquad
		u+\lambda\le m_J-2.
	\]
	This family is therefore absent when \(m_J<2\).  At
	\(u=\lambda=0\), the coefficient is
	\(\mathcal K_{j,J}^{\mathcal B}\).  With \(s=u+\lambda\), direct
	simplification of \eqref{eq:B-components-coefficients} gives
	\begin{equation}
		\label{eq:Jacobi-B-KdF-offdiagonal-coefficient-ratio}
		\frac{\mathcal C_{j;J,1+s}^{\mathcal B}}
		{\mathcal C_{j;J,1}^{\mathcal B}}
		=
		\frac{(2-m_J)_s}{(2)_s}
		\frac{
			(\boldsymbol\alpha+(b_J+2)\one_p+\boldsymbol n)_s
		}{
			(\boldsymbol\alpha+(b_J+2)\one_p)_s
		}
		\frac{
			(b_J\one_{q-1}-\boldsymbol b^{\,*J}
			+2\one_{q-1}-\boldsymbol m^{\,*J})_s
		}{
			(b_J\one_{q-1}-\boldsymbol b^{\,*J}
			+2\one_{q-1})_s
		}.
	\end{equation}
	The \(J\)-th upper parameter is now \(1-K=-s\), so again
	\[
		(-s)_u=(-1)^u\frac{s!}{\lambda!}.
	\]
	Applying the same two shift identities as above, now about the base
	point \(K=1\), transforms the summand into
	\begin{multline}
		\label{eq:Jacobi-B-KdF-offdiagonal-summand}
			\mathcal K_{j,J}^{\mathcal B}
			\frac{
				(2-m_J)_{u+\lambda}
				(\boldsymbol\alpha+(b_J+2)\one_p+\boldsymbol n)_{u+\lambda}
			}{
				(\boldsymbol\alpha+(b_J+2)\one_p)_{u+\lambda}
			}
			\\* \times
			\frac{
				(b_J\one_{q-1}-\boldsymbol b^{\,*J}
				+2\one_{q-1}-\boldsymbol m^{\,*J})_{u+\lambda}
			}{
				(b_J\one_q-\boldsymbol a+2\one_q)_{u+\lambda}
			}
			\frac{
				(1)_u(b_J\one_q-\boldsymbol a+\one_q)_\lambda
			}{
				(b_J\one_{q-1}-\boldsymbol b^{\,*J}
				+\one_{q-1}+\boldsymbol e_j^{\,*J})_\lambda
			}
			\frac{x^u}{u!}
			\frac{1}{\lambda!}.
	\end{multline}
	Indeed, the \(J\)-th entry of
	\(b_J\one_q-\boldsymbol b+\one_q+\boldsymbol e_j\) is \(1\).  Deleting it leaves
	the \(\lambda\)-dependent denominator string displayed above.  The
	remaining common strings cancel between numerator and denominator,
	and the reflected signs cancel as well. Hence
	\eqref{eq:Jacobi-B-KdF-offdiagonal-summand} is the
	\((u,\lambda)\)-term of
	\(\mathcal K_{j,J}^{\mathcal B}\mathcal F_{j,J}(x)\).
	Summing the diagonal family and all nonempty off-diagonal families gives
	\eqref{eq:Jacobi-B-KdF-representation}.
\end{proof}

\begin{remark}[Beyond the near-diagonal range]
	\label{rem:Jacobi-beyond-near-diagonal}
	The near-diagonal assumption in
	Theorem~\ref{thm:mixed-Jacobi-like-forms} is a sufficient condition for the
	finite cancellations needed in the Mellin construction. Under the parameter
	hypotheses of Lemma~\ref{lem:Wolfs-near-diagonal-AT}, it also supplies the
	AT property of the underlying Jacobi vector. It is not, however,
	intrinsic to the residue
	argument itself. Outside the near-diagonal range, the same construction
	can still be carried out whenever suitable finite cancellation conditions
	hold for the extra Mellin factors.
	
	That extension is not considered here. We use the near-diagonal formulation,
	for which cancellation is automatic and, under the hypotheses of
	Corollary~\ref{cor:Jacobi-mixed-normality-near-diagonal}, normality follows
	from the cited AT result; the formulas then have a terminating
	hypergeometric form. The off-step-line
	finite-cancellation theory beyond the near-diagonal range lies outside the
	scope of this paper.
\end{remark}

\begin{remark}[The ordinary Wolfs specialization: \(p=1\)]
	\label{rem:ordinary-Wolfs-hypergeometric}
	If
	\[
	p=1,
	\qquad
	\alpha_1=0,
	\]
	then the rectangular matrix reduces to the ordinary Jacobi-like vector
	\[
	\left[w_1(x;\boldsymbol a,\boldsymbol b),\ldots,w_q(x;\boldsymbol a,\boldsymbol b)\right]\dx
	\]
	studied by Wolfs~\cite{Wolfs2024}. In this specialization, the form
	\(\mathcal A_{\boldsymbol n,\boldsymbol m}\) in part~\textnormal{(i)} is the ordinary
	Jacobi-like type II polynomial. Since there is only one power component,
	formula \eqref{eq:A-components} becomes a single terminating polynomial of
	type \({}_{q+1}F_q\). On the other hand, the form
	\(\mathcal B_{\boldsymbol n,\boldsymbol m}\) in part~\textnormal{(ii)} is the ordinary
	Jacobi-like type I form. Formula \eqref{eq:B-complete-hypergeometric}
	writes the complete type I form as a sum of \(q\) generalized
	hypergeometric functions of type \({}_{q+1}F_q\), whereas each of its
	polynomial components remains, generically, a finite linear combination
	of terminating \({}_{q+1}F_q\) polynomials through
	\eqref{eq:B-components-hypergeometric}. Thus the finite hypergeometric
	combination occurring for the components of \(\mathcal B\) is already
	present in the ordinary Jacobi-like theory and is not an artefact of the
	rectangular mixed extension.
\end{remark}

\begin{remark}[Meijer \(G\)-representation of the complete form]
	\label{rem:B-complete-Meijer-G}
	The sum of generalized hypergeometric functions in
	\eqref{eq:B-complete-hypergeometric} can be compressed at the level of the
	complete form. Indeed, using
	\[
	(\alpha_i+1-s)_{n_i}
	=
	\frac{\Gamma(\alpha_i+n_i+1-s)}
	{\Gamma(\alpha_i+1-s)},
	\]
	formula \eqref{eq:B-Mellin} becomes
	\[
	\mathcal M[\mathcal B_{\boldsymbol n,\boldsymbol m}](s)
	=
	-
	\frac{
		\Gamma(s\one_q+\boldsymbol a)
		\Gamma(\boldsymbol\alpha+\boldsymbol n+(1-s)\one_p)
	}{
		\Gamma(s\one_q+\boldsymbol b+\boldsymbol m)
		\Gamma(\boldsymbol\alpha+(1-s)\one_p)
	}.
	\]
	Since the form \(\mathcal B_{\boldsymbol n,\boldsymbol m}\) is represented on \((0,1)\), this
	Mellin transform may be read as a Mellin transform on \(\mathbb R_+\).
	Consequently, with the Meijer \(G\)-function convention fixed in
	\eqref{eq:Meijer-G-definition}--\eqref{eq:Meijer-G-Mellin-transform},
	\[
	\mathcal B_{\boldsymbol n,\boldsymbol m}(x)
	=
	-
	G_{p+q,p+q}^{\,q,p}
	\left(
	x\,\middle|\,
	\begin{matrix}
		-\boldsymbol\alpha-\boldsymbol n,\ \boldsymbol b+\boldsymbol m\\
		\boldsymbol a,\ -\boldsymbol\alpha
	\end{matrix}
	\right),
	\qquad 0<x<1.
	\]
	Under the corresponding nonresonance assumptions, the hypergeometric sum
	in \eqref{eq:B-complete-hypergeometric} is the residue expansion of this
	single Meijer \(G\)-function. Formula \eqref{eq:B-complete-Meijer-G}
	concerns the complete mixed form \(\mathcal B_{\boldsymbol n,\boldsymbol m}\); it does
	not, in general, replace each individual polynomial component
	\(B_{\boldsymbol n,\boldsymbol m}^{(j)}\) by one Meijer \(G\)-function with comparably
	simple parameters. In the Pi\~neiro degeneration
	of~\cite{PineiroMixed2026}, the simplification of the
	components arises instead from re-expansion in the degenerate power basis,
	as explained below.
\end{remark}

\begin{corollary}[Degeneration to the mixed Pi\~neiro family]
	\label{cor:reduction-mixed-Pineiro}
	Take the boundary degeneration
	\(\boldsymbol a,\boldsymbol b\to\boldsymbol\beta\) in the Mellin identities above.
	Assume that \(\alpha_i+\beta_j>-1\) for all \(i,j\), and that the
	denominators in the two limiting component formulas below do not vanish.
	Then
	\eqref{eq:mixed-Jacobi-like-matrix} reduces to the mixed Pi\~neiro matrix
	of~\cite{PineiroMixed2026},
	\begin{equation}
		\label{eq:Pineiro-matrix}
		\mathrm{d}\boldsymbol{\mu}(x)
		=
		\left[
		x^{\beta_j+\alpha_i}
		\right]_{
			\substack{j\in\{1,\ldots,q\}\\i\in\{1,\ldots,p\}}}
		\dx.
	\end{equation}
	With the normalizations fixed in
	Theorem~\ref{thm:mixed-Jacobi-like-forms}, the nonzero polynomial
	components become
	\begin{multline}
		\label{eq:Pineiro-A}
			A_{\boldsymbol n,\boldsymbol m}^{(i)}(x)
			=
			-
			\frac{
				((\alpha_i+1)\one_q+\boldsymbol\beta)_{\boldsymbol m}
			}{
				(n_i-1)!
				(\boldsymbol\alpha^{\,*i}-\alpha_i\one_{p-1})_{\boldsymbol n^{\,*i}}
			}
			\\*\times
			\pFq{p+q}{p+q-1}
			{
				-n_i+1,\
				(\alpha_i+1)\one_q+\boldsymbol\beta+\boldsymbol m,\
				(\alpha_i+1)\one_{p-1}
				-\boldsymbol\alpha^{\,*i}-\boldsymbol n^{\,*i}
			}
			{
				(\alpha_i+1)\one_q+\boldsymbol\beta,\
				(\alpha_i+1)\one_{p-1}-\boldsymbol\alpha^{\,*i}
			}
			{x},
	\end{multline}
	for \(i\in\{1,\ldots,p\}\) with \(n_i\ge1\), while
	\(A_{\boldsymbol n,\boldsymbol m}^{(i)}\equiv0\) if \(n_i=0\).
	Similarly,
	\begin{multline}
		\label{eq:Pineiro-B}
			B_{\boldsymbol n,\boldsymbol m}^{(j)}(x)
			=
			-
			\frac{
				(\boldsymbol\alpha+(\beta_j+1)\one_p)_{\boldsymbol n}
			}{
				(m_j-1)!
				(\boldsymbol\beta^{\,*j}-\beta_j\one_{q-1})_{\boldsymbol m^{\,*j}}
			}
			\\*\times
			\pFq{p+q}{p+q-1}
			{
				-m_j+1,\
				\boldsymbol\alpha+\boldsymbol n+(\beta_j+1)\one_p,\
				(\beta_j+1)\one_{q-1}
				-\boldsymbol\beta^{\,*j}-\boldsymbol m^{\,*j}
			}
			{
				\boldsymbol\alpha+(\beta_j+1)\one_p,\
				(\beta_j+1)\one_{q-1}-\boldsymbol\beta^{\,*j}
			}
			{x},
	\end{multline}
	for \(j\in\{1,\ldots,q\}\) with \(m_j\ge1\), while
	\(B_{\boldsymbol n,\boldsymbol m}^{(j)}\equiv0\) if \(m_j=0\).
\end{corollary}

\begin{proof}
	Equation \eqref{eq:wj-Pineiro-degeneration} gives the degeneration of the
	weights to the mixed Pi\~neiro system studied in~\cite{PineiroMixed2026}.
	In part~\textnormal{(i)} of
	Theorem~\ref{thm:mixed-Jacobi-like-forms}, the factor
	\((-1)^{|\boldsymbol n|}\) in \eqref{eq:A-contour} compensates the residue sign
	\((-1)^{|\boldsymbol n|-1}\), and hence the formula reduces exactly to
	\eqref{eq:Pineiro-A}.
	
	For the type II components, the specialization must be taken at the level of
	the Mellin transform.
	Before degeneration, formula \eqref{eq:B-components-hypergeometric}
	expresses each component \(B_{\boldsymbol n,\boldsymbol m}^{(j)}\) as a finite linear
	combination of terminating polynomials of type \({}_{q+1}F_q\). This is a
	representation in the Jacobi basis and it is not the convenient
		formula to specialize term by term in the boundary limit
		\(\boldsymbol a,\boldsymbol b\to\boldsymbol\beta\): in this limit
		the Jacobi weights degenerate to pure powers and the relevant basis
		changes.
		
		Instead, specialize the Mellin transform \eqref{eq:B-Mellin} itself. Since
		\(\boldsymbol a,\boldsymbol b\to\boldsymbol\beta\), it becomes
	\begin{equation}
		\label{eq:B-Pineiro-degenerate-Mellin}
		\int_0^1 x^{s-1}\mathcal B_{\boldsymbol n,\boldsymbol m}(x)\dx
		=
		-
		\frac{
			(\boldsymbol\alpha+(1-s)\one_p)_{\boldsymbol n}
		}{
			(s\one_q+\boldsymbol\beta)_{\boldsymbol m}
		}.
	\end{equation}
	Moreover, \(w_j(x;\boldsymbol\beta,\boldsymbol\beta)=x^{\beta_j}\), so we may write
	\[
	\mathcal B_{\boldsymbol n,\boldsymbol m}(x)
	=
	\sum_{\substack{j=1\\m_j\ge1}}^{q}
	\left(
	\sum_{\ell=0}^{m_j-1}c_{j,\ell}x^\ell
	\right)x^{\beta_j}.
	\]
	The coefficients \(c_{j,\ell}\) are precisely the coefficients obtained by
	expanding the rational function in
	\eqref{eq:B-Pineiro-degenerate-Mellin} at its poles
	\(s=-\beta_j-\ell\). For fixed \(j\), their quotient with the constant
	coefficient is
	\begin{align*}
		\frac{c_{j,\ell}}{c_{j,0}}
		={}&
		\frac{(-m_j+1)_\ell}{\ell!}
		\prod_{i=1}^{p}
		\frac{(\alpha_i+\beta_j+n_i+1)_\ell}
		{(\alpha_i+\beta_j+1)_\ell}
		\prod_{\substack{k=1\\k\neq j}}^{q}
		\frac{(\beta_j-\beta_k-m_k+1)_\ell}
		{(\beta_j-\beta_k+1)_\ell}.
	\end{align*}
	Thus all terms belonging to the degenerate power weight
	\(x^{\beta_j}\) collect into one terminating hypergeometric series of
	type \({}_{p+q}F_{p+q-1}\). With the normalization fixed in
	part~\textnormal{(ii)} of Theorem~\ref{thm:mixed-Jacobi-like-forms}, this
	gives exactly \eqref{eq:Pineiro-B}. Hence the apparent simplification from
	a finite combination of \({}_{q+1}F_q\) polynomials to a single
	\({}_{p+q}F_{p+q-1}\) polynomial is a consequence of degeneration and
	re-expansion in the power basis, rather than a termwise collapse of
	\eqref{eq:B-components-hypergeometric}.
\end{proof}

\begin{remark}[Bilateral Christoffel preservation in the Pi\~neiro specialization]
	The degeneration in Corollary~\ref{cor:reduction-mixed-Pineiro}, which
	underlies the mixed Pi\~neiro system studied in~\cite{PineiroMixed2026},
	differs from the nonspecialized Jacobi case. In fact, both
	sides of \eqref{eq:Pineiro-matrix} are power vectors. Hence the left
	and right Christoffel chains preserve the same Pi\~neiro family through
	the cyclic affine shifts
	\[
	\boldsymbol\alpha\longmapsto
	\mathscr C_p^k(\boldsymbol\alpha),
	\qquad
	\boldsymbol\beta\longmapsto
	\mathscr C_q^k(\boldsymbol\beta),
	\]
	respectively. Thus both the lower factors \(L_1,\ldots,L_p\) and the
	upper factors \(U_1,\ldots,U_q\) in
	\eqref{eq:gen-bidiagonal-factorization} are obtained from leading
	coefficients of forms that remain in the Pi\~neiro family. Away from
	this degeneration, only the left chain is automatically preserved inside
	the Jacobi parametric class.
\end{remark}
\section{Laguerre I mixed-type multiple orthogonal systems}
\label{sec:mixed-Laguerre-like-family}

In this section the power vector has length \(p\), whereas the Laguerre I
vector has length
\[
q=r+s.
\]
The first \(r\) Laguerre I weights are beta-shifted weights and the last
\(s\) are Mellin-Euler derivative weights. Thus the mixed matrix below is a
\(q\times p\) matrix of measures. In parallel with the Jacobi notation,
the form expanded in the Laguerre I vector is denoted by \(\mathcal B\),
whereas the dual form expanded in the power vector is denoted by
\(\mathcal A\).

The notation in this section is adapted to the Laguerre I construction.
The multi-index on the Laguerre I vector will be denoted by
\(\boldsymbol\lambda=(\boldsymbol n,\boldsymbol\eta)\in\mathbb N_0^q\), whereas
\(\boldsymbol m\in\mathbb N_0^p\) denotes the multi-index on the power vector. Thus
the order of the two multi-indices is not meant to duplicate the Jacobi
notation, but to keep the Laguerre I side visible in the formulas below.

Let
\[
q=r+s,
\qquad r\ge0,
\qquad s\ge1,
\]
and let
\[
\boldsymbol a=(a_1,\ldots,a_q)\in(-1,\infty)^q,
\qquad
\boldsymbol b=(b_1,\ldots,b_r)\in(-1,\infty)^r.
\]
When \(r>0\), assume
\[
a_h<b_h,
\qquad h\in\{1,\ldots,r\}.
\]

We use the beta density \(\mathcal B_{a,b}\) introduced in the Jacobi
section. For \(a>-1\), define the gamma weight
\[
\mathcal G_a(x)
\coloneq
x^a\mathrm{e}^{-x}\boldsymbol 1_{(0,\infty)}(x),
\qquad
\mathcal M[\mathcal G_a](z)=\Gamma(z+a).
\]

The Laguerre I base weight is the Mellin convolution
\begin{equation}
	\label{eq:w0-explicit-convolution}
	w_0(x;\boldsymbol a,\boldsymbol b)
	\coloneq
	\mathcal B_{a_1,b_1}*\cdots*\mathcal B_{a_r,b_r}
	*
	\mathcal G_{a_{r+1}}*\cdots*\mathcal G_{a_q}(x),
\end{equation}
with the convention that the beta block is absent when \(r=0\). Hence
\begin{equation}
	\label{eq:w0-Laguerre-Mellin}
	\mathcal M[w_0(\cdot;\boldsymbol a,\boldsymbol b)](z)
	=
	\frac{\Gamma(z\one_q+\boldsymbol a)}{\Gamma(z\one_r+\boldsymbol b)}.
\end{equation}
Equivalently, with the Mellin--Barnes convention fixed in
\eqref{eq:Meijer-G-definition}--\eqref{eq:Meijer-G-Mellin-transform},
\begin{equation}
	\label{eq:w0-Meijer-Laguerre}
	w_0(x;\boldsymbol a,\boldsymbol b)
	=
	G_{r,q}^{q,0}
	\left(
	x\,\middle|\,
	\begin{matrix}
		\boldsymbol b\\
		\boldsymbol a
	\end{matrix}
	\right),
	\qquad x>0.
\end{equation}

Following Wolfs's Laguerre-like construction, define the shifted beta weights
\begin{equation}
	\label{eq:wh-def}
	w_h(x;\boldsymbol a,\boldsymbol b)
	\coloneq
	w_0(x;\boldsymbol a,\boldsymbol b+\boldsymbol e_h),
	\qquad h\in\{1,\ldots,r\}.
\end{equation}
For the derivative block we use the Mellin-Euler operator
\[
\thetaop\coloneq -x\frac{\mathrm d}{\mathrm dx},
\]
so that, after integration by parts,
\[
\mathcal M[\thetaop f](z)=z\mathcal M[f](z).
\]
Set
\begin{equation}
	\label{eq:vl-def}
	v_\ell(x;\boldsymbol a,\boldsymbol b)
	\coloneq
	\thetaop^{\ell-1}w_0(x;\boldsymbol a,\boldsymbol b),
	\qquad \ell\in\{1,\ldots,s\}.
\end{equation}
Then
\begin{align}
	\label{eq:wh-Mellin}
	\mathcal M[w_h(\cdot;\boldsymbol a,\boldsymbol b)](z)
	&=
	\frac{\Gamma(z\one_q+\boldsymbol a)}{\Gamma(z\one_r+\boldsymbol b+\boldsymbol e_h)},
	&&h\in\{1,\ldots,r\},
	\\
	\label{eq:vl-Mellin}
	\mathcal M[v_\ell(\cdot;\boldsymbol a,\boldsymbol b)](z)
	&=
	z^{\ell-1}\frac{\Gamma(z\one_q+\boldsymbol a)}{\Gamma(z\one_r+\boldsymbol b)},
	&&\ell\in\{1,\ldots,s\}.
\end{align}

We collect these \(q=r+s\) functions in the Laguerre I vector
\begin{equation}
	\label{eq:Laguerre-vector}
	\mathbf U^{\mathrm L}(x)
	\coloneq
	\left[
	w_1,\ldots,w_r,
	v_1,\ldots,v_s
	\right]
	=
	\left[U_1^{\mathrm L},\ldots,U_q^{\mathrm L}\right].
\end{equation}
Let the independent power vector be
\begin{equation}
	\label{eq:power-vector}
	\mathbf V(x)
	\coloneq
	\left[x^{\beta_1},\ldots,x^{\beta_p}\right].
\end{equation}
Assume that
\[
\beta_i+a_\rho>-1,
\qquad
i\in\{1,\ldots,p\},
\quad
\rho\in\{1,\ldots,q\}.
\]
Under this condition all pairings appearing in the mixed orthogonality
relations below are absolutely convergent. Equivalently, if
\[
a_{\min}\coloneq\min_{\rho\in\{1,\ldots,q\}}a_\rho,
\]
it is enough to require \(\beta_i>-1-a_{\min}\) for every
\(i\in\{1,\ldots,p\}\). The rectangular mixed matrix of measures is
\begin{equation}
	\label{eq:mixed-Laguerre-matrix}
	\mathrm d\Lagmat(x)
	\coloneq
	\left[
	U_\rho^{\mathrm L}(x)x^{\beta_i}
	\right]_{
		\substack{\rho\in\{1,\ldots,q\}\\i\in\{1,\ldots,p\}}
	}\dx.
\end{equation}

\begin{remark}[The multiple Laguerre-I specialization: \(q=1\)]
	\label{rem:multiple-Laguerre-I-edge}
	If \(q=1\), then \(q=r+s\), \(r\ge0\), and \(s\ge1\) force
	\(r=0\) and \(s=1\). Consequently,
	\[
		U_1^{\mathrm L}(x)=w_0(x;a_1)=x^{a_1}\e^{-x},
	\]
	and the entries of \eqref{eq:mixed-Laguerre-matrix} are
	\[
		x^{\beta_i}U_1^{\mathrm L}(x)
		=
		x^{\widetilde a_i-1}\e^{-x},
		\qquad
		\widetilde a_i\coloneq a_1+\beta_i+1,
		\qquad i\in\{1,\ldots,p\}.
	\]
	Thus the one-row specialization is exactly the system of multiple Laguerre
	polynomials of the first kind.
\end{remark}

\begin{definition}[Laguerre admissibility]
	\label{def:Laguerre-admissible}
	Let
	\[
	\boldsymbol\lambda=(\boldsymbol n,\boldsymbol\eta)
	=
	(n_1,\ldots,n_r,\eta_1,\ldots,\eta_s)
	\in
	\mathbb N_0^r\times\mathbb N_0^s.
	\]
	Following the index range used by Wolfs for Laguerre I weights, we say
	that \(\boldsymbol\lambda\) is \emph{Laguerre admissible} if the full multi-index
	\(\boldsymbol\lambda\) is near the diagonal and the derivative multi-index
	\(\boldsymbol\eta\) lies on the step-line in \(\mathbb N_0^s\).
\end{definition}

\begin{proposition}[Combinatorial criterion for Laguerre admissibility]
	\label{prop:Laguerre-admissibility-inequalities}
	Let
	\[
	\boldsymbol\lambda=(\boldsymbol n,\boldsymbol\eta)
	=
	(n_1,\ldots,n_r,\eta_1,\ldots,\eta_s)
	\in
	\mathbb N_0^r\times\mathbb N_0^s.
	\]
	Then \(\boldsymbol\lambda\) is Laguerre admissible if and only if the following four explicit inequalities hold:
	\begin{subequations}
		\label{eq:Laguerre-admissibility-inequalities}
		\begin{align}
			n_h-1&\le n_u,
			&
			h&\in\{1,\ldots,r\}\textnormal{ with }n_h\ge1,
			\quad
			u\in\{1,\ldots,r\},
			\label{eq:Laguerre-admissibility-nn}
			\\[2pt]
			\eta_\ell-1&\le n_u,
			&
			\ell&\in\{1,\ldots,s\}\textnormal{ with }\eta_\ell\ge1,
			\quad
			u\in\{1,\ldots,r\},
			\label{eq:Laguerre-admissibility-etan}
			\\[2pt]
			s(n_h-1)&\le |\boldsymbol\eta|,
			&
			h&\in\{1,\ldots,r\}\textnormal{ with }n_h\ge1,
			\label{eq:Laguerre-admissibility-n-eta}
			\\[2pt]
			s(\eta_\ell-1)+\ell&\le |\boldsymbol\eta|,
			&
			\ell&\in\{1,\ldots,s\}\textnormal{ with }\eta_\ell\ge1.
			\label{eq:Laguerre-admissibility-eta-step}
		\end{align}
	\end{subequations}
	When \(r=0\), the conditions involving \(\boldsymbol n\) are absent. The criterion is purely combinatorial: it rewrites the condition that \(\boldsymbol\lambda\) is near the diagonal and that \(\boldsymbol\eta\) is on the step-line without referring to those two definitions separately. 
\end{proposition}
\begin{proof}
	If \(r=0\), the assertion reduces to the standard characterization of the
	step-line condition for \(\boldsymbol\eta\), namely
	\eqref{eq:Laguerre-admissibility-eta-step}. Thus we assume \(r>0\) in the
	rest of the proof.
	Write
	\[
	|\boldsymbol\eta|=sQ+R,
	\qquad
	Q\in\mathbb N_0,
	\quad
	R\in\{0,\ldots,s-1\}.
	\]
	With our step-line convention,
	\[
	\eta_\ell=
	\begin{cases}
		Q+1, & \ell\in\{1,\ldots,R\},\\
		Q, & \ell\in\{R+1,\ldots,s\}.
	\end{cases}
	\]
	This is equivalent to
	\[
	s(\eta_\ell-1)+\ell\le |\boldsymbol\eta|,
	\qquad
	\ell\in\{1,\ldots,s\}\textnormal{ with }\eta_\ell\ge1.
	\]
	Indeed, these inequalities imply
	\[
	\eta_\ell\le Q+1
	\quad\textnormal{for }\ell\in\{1,\ldots,R\},
	\qquad
	\eta_\ell\le Q
	\quad\textnormal{for }\ell\in\{R+1,\ldots,s\},
	\]
	and, since the sum of the \(\eta_\ell\)'s is \(sQ+R\), all these upper
	bounds must be attained.
	
	Assume first that \(\boldsymbol\lambda\) is Laguerre admissible. Then
	\(\boldsymbol\lambda\) is near the diagonal, so the first two inequalities in
	\eqref{eq:Laguerre-admissibility-inequalities} follow immediately. The
	fourth inequality is the step-line condition just described. Since the
	minimum component of \(\boldsymbol\eta\) is \(Q\), near diagonality gives
	\(n_h\le Q+1\). Hence
	\[
	s(n_h-1)\le sQ\le |\boldsymbol\eta|,
	\]
	which is the third inequality.
	
	Conversely, assume
	\eqref{eq:Laguerre-admissibility-inequalities}. The first inequality says
	that the components of \(\boldsymbol n\) differ from each other by at most one.
	The second says that no component of \(\boldsymbol\eta\) exceeds any component
	of \(\boldsymbol n\) by more than one. The fourth says that \(\boldsymbol\eta\) lies on
	the step-line, and therefore \(\min_\ell\eta_\ell=Q\). The third
	inequality gives \(n_h\le Q+1\), so no component of \(\boldsymbol n\) exceeds
	any component of \(\boldsymbol\eta\) by more than one. Hence all components of
	\(\boldsymbol\lambda=(\boldsymbol n,\boldsymbol\eta)\) differ by at most one, and
	\(\boldsymbol\lambda\) is near the diagonal. Thus \(\boldsymbol\lambda\) is Laguerre
	admissible.
\end{proof}

\begin{remark}
	The preceding proposition is an admissibility criterion, not a normality statement.
	The role of the four inequalities will be made explicit in
	Theorem~\ref{thm:mixed-Laguerre-forms}. Namely,
	\eqref{eq:Laguerre-admissibility-nn} and
	\eqref{eq:Laguerre-admissibility-etan} ensure that the relevant gamma
	quotients are polynomials, whereas
	\eqref{eq:Laguerre-admissibility-n-eta} and
	\eqref{eq:Laguerre-admissibility-eta-step} give the required membership in
	the polynomial spaces \(\mathbb P_{m_i-1}\), \(\mathbb P_{n_h-1}\), and
	\(\mathbb P_{\eta_\ell-1}\) in the contour argument.  Thus the admissible
	range is used here only to keep the residue formulas inside the prescribed
	mixed polynomial spaces; normality itself is supplied separately by
	Proposition~\ref{prop:AT-normality-Laguerre-admissible}.
\end{remark}

\begin{proposition}[Exact Laguerre determinants and normality]
\label{prop:AT-normality-Laguerre-admissible}
Let \(\boldsymbol\lambda=(\boldsymbol n,\boldsymbol\eta)\) be Laguerre
admissible, and set \(N=|\boldsymbol\lambda|\).  Retain the
integrability assumptions of the Laguerre matrix measure.  Put
\begin{gather}
 D_{\boldsymbol\lambda}(z):=\prod_{h=1}^r(z+b_h)_{n_h},
 \qquad G_0(z):=\frac{\Gamma(z\one_q+\boldsymbol a)}
                  {\Gamma(z\one_r+\boldsymbol b)},
 \label{eq-Laguerre-determinant-data}\\
 c_{\boldsymbol\lambda}
 :=\prod_{1\leq h<j\leq r}(b_j-b_h)^{\min(n_h,n_j)}
 \prod_{j=1}^r\prod_{\rho=1}^q
       \prod_{u=0}^{n_j-1}(b_j+1-a_\rho)_u.
 \label{eq-Laguerre-coefficient-determinant}
\end{gather}
Empty products are one.  For a power-side multi-index
\(\boldsymbol m\), order its nodes \(z=\beta_i+k+1\), \(0\leq k<m_i\),
by increasing \(k\), and then increasing \(i\); use the analogous
degree-first order on the weighted rows.  If \(|\boldsymbol m|=N\),
the square mixed moment block satisfies
\begin{equation}
 \det G_{\boldsymbol\lambda,\boldsymbol m}
 =c_{\boldsymbol\lambda}
   \prod_{1\leq u<v\leq N}(z_v-z_u)
   \prod_{v=1}^N\frac{G_0(z_v)}{D_{\boldsymbol\lambda}(z_v)}.
 \label{eq-Laguerre-moment-determinant}
\end{equation}
In particular, assume \(c_{\boldsymbol\lambda}\ne0\) and
\begin{equation}
 \beta_i-\beta_j\notin\mathbb Z,
 \qquad i\ne j.
 \label{eq-Laguerre-beta-nonresonance}
\end{equation}
Then the two balances
\begin{equation}
 |\boldsymbol m|=N+1
 \qquad\hbox{and}\qquad
 |\boldsymbol m|=N-1,
 \label{eq-Laguerre-two-balances}
\end{equation}
have one-dimensional solution spaces; the second is used only for
\(N\geq1\).  In the first balance every active power component has
its maximal degree.  In the second, the complete Laguerre form is
nonzero and has a unique coefficient representation.  All its active
components have maximal degree if and only if every minor obtained
by deleting a terminal weighted row from the
\(N\)-by-\((N-1)\) moment block is nonzero.  A terminal row is the
highest-degree row of any active component.

All scalar step-line leading minors are therefore nonzero if
\eqref{eq-Laguerre-beta-nonresonance} and
\begin{equation}
 b_h\ne b_j\ (h\ne j),\qquad
 b_h+v\ne a_\rho
 \quad(1\leq h\leq r,\ 1\leq\rho\leq q,\ v\in\N).
 \label{eq-Laguerre-global-nonresonance}
\end{equation}
These conditions also ensure the active step-line normalizations of
both forms at every order.  They do not assert maximal degrees of
all non-active Laguerre components without the terminal-minor test.
\end{proposition}

\begin{proof}
The weighted Mellin transforms share the factor
\(G_0(z)/D_{\boldsymbol\lambda}(z)\).  After division, they become
\begin{align*}
 P_{h,k}(z)
 &=\prod_{\rho=1}^q(z+a_\rho)_k
   \prod_{u=1}^r
    (z+b_u+k+\delta_{uh})_{n_u-k-\delta_{uh}},
 &&0\leq k<n_h,\\
 E_{\ell,k}(z)
 &=\prod_{\rho=1}^q(z+a_\rho)_k
   \prod_{u=1}^r(z+b_u+k)_{n_u-k}(z+k)^{\ell-1},
 &&0\leq k<\eta_\ell.
\end{align*}
Admissibility makes every Pochhammer length nonnegative and every
degree at most \(N-1\).  Let \(C_{\boldsymbol\lambda}\) be their
coefficient matrix in \(1,z,\ldots,z^{N-1}\), in degree-first row
order.  We prove \(\det C_{\boldsymbol\lambda}=c_{\boldsymbol\lambda}\).

Write \(t=\min\lambda_j\).  Build the rows as \(t\) complete
cycles, each containing the beta rows followed by the Euler rows,
then the additional beta rows in their original order and the first
additional Euler rows.  Every prefix is admissible.  Appending beta
row \((v,j)\) multiplies every previous polynomial by
\(z+b_j+v\).  If \(H\) is the set of beta rows already appended in
this cycle, the new polynomial is
\[
 P(z)=\prod_{\rho=1}^q(z+a_\rho)_v
       \prod_{h\in H}(z+b_h+v).
\]
There are \(d=qv+|H|\) previous rows.  The coefficient determinant
of \((z+c)f_1,\ldots,(z+c)f_d,P\) is
\((-1)^dP(-c)\) times that of \(f_1,\ldots,f_d\): evaluate the
last row at \(-c\), or divide by \(z+c\) in the monomial basis.
For \(c=b_j+v\) its multiplier is therefore
\[
 (-1)^dP(-b_j-v)
 =\prod_{\rho=1}^q(b_j+1-a_\rho)_v
   \prod_{h\in H}(b_j-b_h).
\]
Appending Euler row \((v,\ell)\) leaves the common denominator
unchanged.  Its polynomial is monic of degree equal to the number
of previous rows, so it leaves the coefficient determinant unchanged.
Each beta pair \(h<j\) contributes once per shared cycle, namely
\(\min(n_h,n_j)\) times.  Multiplication of these factors proves
\eqref{eq-Laguerre-coefficient-determinant}, including zero
factors; the case \(N=0\) is the empty determinant.

At a moment node \(z\), the column of the moment block is the
evaluation column of these polynomials multiplied by
\(G_0(z)/D_{\boldsymbol\lambda}(z)\).  Factoring the coefficient
matrix and the Vandermonde evaluation matrix proves
\eqref{eq-Laguerre-moment-determinant}.  The gamma arguments
and the factors in \(D_{\boldsymbol\lambda}(z)\) are positive in
the stated integrability range, so they introduce no zeros or poles.

If \(c_{\boldsymbol\lambda}\ne0\), the weighted Mellin map is an
isomorphism onto \(\mathbb P_{N-1}\).  Evaluating at \(N+1\)
distinct nodes gives an \(N\)-by-\((N+1)\) matrix whose every
maximal minor is nonzero, by the same Vandermonde factorization.
Its kernel has dimension one and all its coordinates are nonzero.
This proves the first balance, including maximal component degrees.
For \(N-1\) nodes the evaluation map has rank \(N-1\), so the
second kernel also has dimension one.  Its polynomial image is a
nonzero multiple of \(\prod_{i=1}^p(\beta_i+1-z)_{m_i}\).
The coordinates of this kernel are the signed maximal row minors.
Thus maximal degree in each active Laguerre component is equivalent
to the asserted terminal-minor condition.

For a step-line row index, deletion of its last row gives the
preceding step-line moment block.  Under
\eqref{eq-Laguerre-global-nonresonance} both coefficient
determinants are nonzero.  Hence the last coefficient used in the
step-line normalization is nonzero.  Together with the square
determinant formula this proves all the final assertions.
\end{proof}

\begin{remark}[Paired beta gaps do not guarantee independence]
The condition \(b_h-a_h\in\mathbb N_0\cup(N-1,\infty)\), used in
the Laguerre I AT argument of \cite[Proposition~3.3]{Wolfs2024},
must not replace the determinant test for the present family.
For \(p=r=s=1\), \(\boldsymbol a=(0,2)\), \(b_1=1\), and \(\beta_1=0\),
one has \(w_1=e^{-x}\) and \(v_1=(x+1)e^{-x}\). The admissible
index \(\boldsymbol\lambda=(2,1)\) gives dependent rows
\(w_1,v_1,xw_1\), and
\[
 \det\begin{bmatrix}1&1&2\\2&3&8\\1&2&6\end{bmatrix}=0.
\]
Here \(b_1-a_1=1\), but \(c_{(2,1)}=0\) since \(b_1+1=a_2\).
A simple sufficient chamber for all scalar step-line leading minors
is given by distinct \(b_h\), \(a_\rho<b_h+1\) for all \(h,\rho\),
and noninteger differences of the \(\beta_i\), in addition to
integrability. Maximality of every B-component remains a separate
terminal-coefficient question, addressed next.
\end{remark}

\begin{proposition}[Strong B-normality by exponent shifts]
\label{prop-Laguerre-shift-normality}
Fix an admissible \(\boldsymbol\lambda\), \(N=|\boldsymbol\lambda|\geq1\),
with \(c_{\boldsymbol\lambda}\ne0\), and a power index
\(|\boldsymbol m|=N-1\). Let
\(\beta_i=t+\gamma_i\), where \(t\) is real and the fixed real
\(\gamma_i\) have pairwise noninteger differences. Form the coefficient
matrix \(C_{\boldsymbol\lambda}\) in the preceding proof and write
\[
 R_t(z):=\prod_{i=1}^p\prod_{k=0}^{m_i-1}
                (z-t-\gamma_i-k-1)
       =\sum_{d=0}^{N-1}r_d(t)z^d,\qquad
 h_j(t):=\sum_{d=0}^{N-1}r_d(t)
                         [C_{\boldsymbol\lambda}^{-1}]_{d,j}.
\]
Here both matrix indices run from \(0\) to \(N-1\). For each weighted
row \(j\), put
\[
 d_j:=\min\{d:[C_{\boldsymbol\lambda}^{-1}]_{d,j}\ne0\},
 \qquad D_j:=N-1-d_j.
\]
Then \(h_j\) is a nonzero polynomial of degree \(D_j\), with leading
coefficient
\[
 [t^{D_j}]h_j(t)=
 (-1)^{D_j}\binom{N-1}{d_j}
 [C_{\boldsymbol\lambda}^{-1}]_{d_j,j}.
\]
In the integrability range, all active B-components have maximal degree
except at the real zeros of the finitely many \(h_j\) corresponding
to terminal weighted rows. In particular, writing
\(h_j(t)=\sum_{v=0}^{D_j}h_{j,v}t^v\), the explicit sufficient condition
\[
 t>1+\max_{\substack{j\ {\rm terminal}\\0\leq v<D_j}}
               \left|\frac{h_{j,v}}{h_{j,D_j}}\right|
\]
excludes every such zero; an empty maximum is zero.

One common sufficiently large shift works simultaneously for any
prescribed finite collection of these index pairs. Under the all-order
conditions \eqref{eq-Laguerre-global-nonresonance}, at most a
countable set of real shifts is exceptional for strong B-normality
at any admissible index pair. Thus almost every integrable common
shift works at all orders. This last assertion does not provide an
order-independent threshold for all sufficiently large shifts.
\end{proposition}

\begin{proof}
The normalized Mellin numerator of a B-form is, up to its overall sign,
\(R_t\). Its weighted coefficient row is
\((r_0(t),\ldots,r_{N-1}(t))C_{\boldsymbol\lambda}^{-1}\).
The coefficient \(r_d(t)\) has degree \(N-1-d\) and leading coefficient
\((-1)^{N-1-d}\binom{N-1}{d}\). In the sum defining \(h_j\), the first
nonzero inverse-matrix entry therefore gives the unique term of
highest degree; no cancellation of that term is possible. This proves
the degree and leading-coefficient formulas, including \(N=1\).
A terminal coefficient vanishes exactly when its \(h_j\) does.
Cauchy's elementary root bound gives the stated sufficient inequality:
if \(|t|>1+M\), the sum of the absolute values of all lower terms of
a monic polynomial with coefficients bounded by \(M\) is less than
\(|t|^{D_j}\). Taking the maximum of finitely many bounds proves the
finite-collection assertion. Finally, there are only countably many
index pairs, and each contributes finitely many real roots.
Their union has Lebesgue measure zero. Integrability is retained by
restricting to \(t+\gamma_i+a_\rho>-1\) for every \(i,\rho\).
\end{proof}

Let
\[
\boldsymbol n=(n_1,\ldots,n_r)\in\mathbb N_0^r,
\qquad
\boldsymbol\eta=(\eta_1,\ldots,\eta_s)\in\mathbb N_0^s,
\]
and put
\[
\boldsymbol\lambda\coloneq(\boldsymbol n,\boldsymbol\eta)
=
(n_1,\ldots,n_r,\eta_1,\ldots,\eta_s),
\qquad
N\coloneq |\boldsymbol n|+|\boldsymbol\eta|.
\]

Empty sums and empty products are understood in the usual sense. In the
Laguerre admissible range considered below, the normality input is supplied
by Proposition~\ref{prop:AT-normality-Laguerre-admissible}.  Once this
normality is known, the word normalized is used in the sense of
Definition~\ref{def:normalization-convention}: the formulas below choose a
specific representative of the one-dimensional space of solutions.

\begin{theorem}[Explicit mixed Laguerre I forms]
	\label{thm:mixed-Laguerre-forms}
	Assume that
	\[
	\boldsymbol\lambda=(\boldsymbol n,\boldsymbol\eta)
	\in
	\mathbb N_0^r\times\mathbb N_0^s
	\]
	is Laguerre admissible in the sense of
	Definition~\ref{def:Laguerre-admissible}. Assume the standing integrability
	conditions, \(c_{\boldsymbol\lambda}\ne0\) in
	\eqref{eq-Laguerre-coefficient-determinant}, and
	\(\beta_i-\beta_j\notin\mathbb Z\) for \(i\ne j\).
	
	\smallskip
	
	\noindent
	\textnormal{(i) The \(A\)-form.}
	Let
	\[
	\boldsymbol m=(m_1,\ldots,m_p)\in\mathbb N_0^p,
	\qquad
	|\boldsymbol m|=N+1,
	\]
	and assume that the nodes
	\[
	\beta_i+k,
	\qquad
	i\in\{1,\ldots,p\}\textnormal{ with }m_i\ge1,
	\quad
	k\in\{0,\ldots,m_i-1\},
	\]
	are pairwise distinct. Set
	\[
	D_{\boldsymbol\beta,\boldsymbol m}(t)
	\coloneq
	\prod_{i=1}^{p}(\beta_i-t)_{m_i}.
	\]
	The contour-normalized representative of the mixed form on the power vector is
	\begin{equation}
		\label{eq:Laguerre-A-contour}
		\mathcal A_{\boldsymbol\lambda,\boldsymbol m}^{\mathrm L}(x)
		\coloneq
		\frac{1}{2\pi\mathrm i}
		\int_{\Sigma}
		\frac{\Gamma(t\one_r+\boldsymbol b+\boldsymbol n+\one_r)}
		{\Gamma(t\one_q+\boldsymbol a+\one_q)}
		\frac{x^t}{D_{\boldsymbol\beta,\boldsymbol m}(t)}\dt,
	\end{equation}
	where \(\Sigma\) encloses precisely the nodes
	\(\beta_i,\ldots,\beta_i+m_i-1\),
	\(i\in\{1,\ldots,p\}\) with \(m_i\ge1\),
	and no other pole of the integrand. Its residue expansion is
	\[
	\mathcal A_{\boldsymbol\lambda,\boldsymbol m}^{\mathrm L}(x)
	=
	\sum_{i=1}^{p}
	A_{\boldsymbol\lambda,\boldsymbol m}^{(i)}(x)x^{\beta_i},
	\]
	where
	\(A_{\boldsymbol\lambda,\boldsymbol m}^{(i)}\equiv0\) if \(m_i=0\), and
	\(A_{\boldsymbol\lambda,\boldsymbol m}^{(i)}\in\mathbb P_{m_i-1}\) if \(m_i\ge1\).
	More explicitly, for \(m_i\ge1\),
	\begin{equation}
		\label{eq:Laguerre-A-components-hypergeometric}
		A_{\boldsymbol\lambda,\boldsymbol m}^{(i)}(x)
		=
		\kappa_i^{\mathrm L,\mathcal A}(\boldsymbol\lambda,\boldsymbol m)
		\pFq{p+r}{p+q-1}
		{
			-m_i+1,\
			(\beta_i+1)\one_r+\boldsymbol b+\boldsymbol n,\
			(\beta_i+1)\one_{p-1}
			-\boldsymbol\beta^{\,*i}-\boldsymbol m^{\,*i}
		}
		{
			(\beta_i+1)\one_q+\boldsymbol a,\
			(\beta_i+1)\one_{p-1}-\boldsymbol\beta^{\,*i}
		}
		{x},
	\end{equation}
	where
	\begin{equation}
		\label{eq:Laguerre-A-kappa}
		\kappa_i^{\mathrm L,\mathcal A}(\boldsymbol\lambda,\boldsymbol m)
		\coloneq
		-
		\frac{
			\Gamma((\beta_i+1)\one_r+\boldsymbol b+\boldsymbol n)
		}{
			(m_i-1)!\,
			\Gamma((\beta_i+1)\one_q+\boldsymbol a)\,
			(\boldsymbol\beta^{\,*i}-\beta_i\one_{p-1})_{\boldsymbol m^{\,*i}}
		}.
	\end{equation}
	If \(m_i=0\), then \(A_{\boldsymbol\lambda,\boldsymbol m}^{(i)}\equiv0\). The \(A\)-form
	satisfies
	\begin{align}
		\label{eq:A-orthogonality-wh}
		\int_0^\infty
		\mathcal A_{\boldsymbol\lambda,\boldsymbol m}^{\mathrm L}(x)
		x^k w_h(x;\boldsymbol a,\boldsymbol b)\dx
		&=0,
		&&k\in\{0,\ldots,n_h-1\},
		\quad
		h\in\{1,\ldots,r\}\textnormal{ with }n_h\ge1,
		\\
		\label{eq:A-orthogonality-vl}
		\int_0^\infty
		\mathcal A_{\boldsymbol\lambda,\boldsymbol m}^{\mathrm L}(x)
		x^k v_\ell(x;\boldsymbol a,\boldsymbol b)\dx
		&=0,
		&&k\in\{0,\ldots,\eta_\ell-1\},
		\quad
		\ell\in\{1,\ldots,s\}\textnormal{ with }\eta_\ell\ge1.
	\end{align}
	
	\smallskip
	
	\noindent
	\textnormal{(ii) The \(B\)-form.}
	Let
	\[
	\boldsymbol m=(m_1,\ldots,m_p)\in\mathbb N_0^p,
	\qquad
	|\boldsymbol m|=N-1.
	\]
	Assume also, for the residue expansion below, that
	\[
	a_\rho-a_\sigma\notin\mathbb Z,
	\qquad
	\rho,\sigma\in\{1,\ldots,q\},
	\quad
	\rho\neq\sigma.
	\]
	There exists a unique representative of the normal \(\mathcal B\)-form,
	normalized by the Mellin identity below,
	\begin{equation}
		\label{eq:Laguerre-B-form}
		\mathcal B_{\boldsymbol\lambda,\boldsymbol m}^{\mathrm L}(x)
		=
		\sum_{h=1}^{r}
		B_{\boldsymbol\lambda,\boldsymbol m}^{(h)}(x)
		w_h(x;\boldsymbol a,\boldsymbol b)
		+
		\sum_{\ell=1}^{s}
		D_{\boldsymbol\lambda,\boldsymbol m}^{(\ell)}(x)
		v_\ell(x;\boldsymbol a,\boldsymbol b),
	\end{equation}
	where
	\[
	B_{\boldsymbol\lambda,\boldsymbol m}^{(h)}\in\mathbb P_{n_h-1}
	\quad\textnormal{if }n_h\ge1,
	\qquad
	B_{\boldsymbol\lambda,\boldsymbol m}^{(h)}\equiv0
	\quad\textnormal{if }n_h=0,
	\]
	and
	\[
	D_{\boldsymbol\lambda,\boldsymbol m}^{(\ell)}\in\mathbb P_{\eta_\ell-1}
	\quad\textnormal{if }\eta_\ell\ge1,
	\qquad
	D_{\boldsymbol\lambda,\boldsymbol m}^{(\ell)}\equiv0
	\quad\textnormal{if }\eta_\ell=0.
	\]
	It satisfies
	\begin{equation}
		\label{eq:Laguerre-B-orthogonality}
		\int_0^\infty
		\mathcal B_{\boldsymbol\lambda,\boldsymbol m}^{\mathrm L}(x)
		x^{\beta_i+k}\dx
		=0,
		\qquad
		k\in\{0,\ldots,m_i-1\},
		\quad
		i\in\{1,\ldots,p\},
	\end{equation}
	where this condition is absent when \(m_i=0\). The normalization is
	prescribed by
	\begin{equation}
		\label{eq:Laguerre-B-Mellin}
		\mathcal M[\mathcal B_{\boldsymbol\lambda,\boldsymbol m}^{\mathrm L}](z)
		=
		\frac{\Gamma(z\one_q+\boldsymbol a)}
		{\Gamma(z\one_r+\boldsymbol b+\boldsymbol n)}
		(\boldsymbol\beta+(1-z)\one_p)_{\boldsymbol m}.
	\end{equation}
	The complete form also has the hypergeometric residue expansion
	\begin{multline}
		\label{eq:Laguerre-B-complete-hypergeometric}
			\mathcal B_{\boldsymbol\lambda,\boldsymbol m}^{\mathrm L}(x)
			=
			\sum_{\rho=1}^{q}
			\frac{
				\Gamma(\boldsymbol a^{\,*\rho}-a_\rho\one_{q-1})
				(\boldsymbol\beta+(a_\rho+1)\one_p)_{\boldsymbol m}
			}{
				\Gamma(\boldsymbol b+\boldsymbol n-a_\rho\one_r)
			}
			x^{a_\rho}
			\\*\times
			\pFq{r+p}{q+p-1}
			{
				(a_\rho+1)\one_r-\boldsymbol b-\boldsymbol n,\
				\boldsymbol\beta+\boldsymbol m+(a_\rho+1)\one_p
			}
			{
				(a_\rho+1)\one_{q-1}-\boldsymbol a^{\,*\rho},\
				\boldsymbol\beta+(a_\rho+1)\one_p
			}
			{(-1)^s x}.
	\end{multline}
	
	Equivalently, the same complete form can be written compactly as the
	Meijer \(G\)-function
	\begin{equation}
		\label{eq:Laguerre-B-complete-Meijer-G}
		\mathcal B_{\boldsymbol\lambda,\boldsymbol m}^{\mathrm L}(x)
		=
		G_{p+r,p+q}^{q,p}
		\left(
		x\,\middle|\,
		\begin{matrix}
			-\boldsymbol\beta-\boldsymbol m,\ \boldsymbol b+\boldsymbol n\\
			\boldsymbol a,\ -\boldsymbol\beta
		\end{matrix}
		\right).
	\end{equation}
	
	The shifted-beta components also admit an explicit closed
	hypergeometric form. For the formulas below, assume additionally that the
	nodes
	\[
	-b_H-K,
	\qquad
	H\in\{1,\ldots,r\},
	\quad K\in\{0,\ldots,n_H-1\},
	\]
	are pairwise distinct and that every displayed denominator is nonzero.
	These extra finite nonresonance conditions are not needed for the Mellin or
	Meijer \(G\)-representations of the complete form. Let
	\[
	\Pi_{\boldsymbol\lambda,\boldsymbol m}^{\mathcal B,\mathrm L}(z)
	\coloneq
	\left[
	\frac{(\boldsymbol\beta+(1-z)\one_p)_{\boldsymbol m}}{(z\one_r+\boldsymbol b)_{\boldsymbol n}}
	\right]_+
	\]
	be the polynomial part at \(z=\infty\). For
	\(\ell\in\{1,\ldots,s\}\) and
	\(k\in\{0,\ldots,\eta_\ell-1\}\), set
	\[
	E_{\ell,k}^{\mathrm L}(z)
	\coloneq
	\left[
	\frac{(z\one_q+\boldsymbol a)_k}{(z\one_r+\boldsymbol b)_k}(z+k)_{\ell-1}
	\right]_+.
	\]
	The coefficients \(\delta_{\ell,k}^{\mathcal B,\mathrm L}\) are determined by
	the finite expansion
	\begin{equation}
		\label{eq:Laguerre-B-delta-polynomial-expansion}
		\Pi_{\boldsymbol\lambda,\boldsymbol m}^{\mathcal B,\mathrm L}(z)
		=
		\sum_{\ell=1}^{s}
		\sum_{k=0}^{\eta_\ell-1}
		\delta_{\ell,k}^{\mathcal B,\mathrm L}
		E_{\ell,k}^{\mathrm L}(z),
	\end{equation}
	with the convention that the sum is absent when \(|\boldsymbol\eta|=0\).
	For \(H\in\{1,\ldots,r\}\) and \(K\in\{0,\ldots,n_H-1\}\), define the
	effective residues directly by
	\begin{multline}
	\label{eq:Laguerre-B-effective-rho-direct-statement}
			\widehat\rho_{H,K}^{\mathcal B,\mathrm L}
			\coloneq
			(-1)^K
			\frac{
				(\boldsymbol\beta+(b_H+K+1)\one_p)_{\boldsymbol m}
			}{
				K!(n_H-K-1)!
				(\boldsymbol b^{\,*H}-(b_H+K)\one_{r-1})_{\boldsymbol n^{\,*H}}
			}
			\\*
			-
			\sum_{\ell=1}^{s}
			\sum_{k=K+1}^{\eta_\ell-1}
			\delta_{\ell,k}^{\mathcal B,\mathrm L}
			\frac{
				(k-b_H-K)_{\ell-1}
				(\boldsymbol a-(b_H+K)\one_q)_k
			}{
				(-1)^K K!(k-1-K)!
				(\boldsymbol b^{\,*H}-(b_H+K)\one_{r-1})_k
			},
	\end{multline}
	where empty sums are absent. Then, for each
	\(h\in\{1,\ldots,r\}\) with \(n_h\ge1\),
	\begin{equation}
		\label{eq:Laguerre-B-components-effective-hypergeometric}
		B_{\boldsymbol\lambda,\boldsymbol m}^{(h)}(x)
		=
		\sum_{\substack{H=1\\ n_H\ge1}}^{r}
		\sum_{K=0}^{n_H-1}
		\widehat\rho_{H,K}^{\mathcal B,\mathrm L}
		\mathscr R_{h;H,K}^{\mathrm L}(x),
	\end{equation}
	where
	\begin{equation}
		\label{eq:Laguerre-B-block-zero}
		\mathscr R_{h;H,K}^{\mathrm L}(x)=0
		\quad\textnormal{if }h\ne H\textnormal{ and }K=0,
	\end{equation}
	and, in all other cases,
	\begin{equation}
		\label{eq:Laguerre-B-hypergeometric-block-final}
		\mathscr R_{h;H,K}^{\mathrm L}(x)
		\coloneq
		\frac{(\boldsymbol a-b_h\one_q)_1}{(\boldsymbol b^{\,*h}-b_h\one_{r-1})_1}
		\frac{(\boldsymbol b^{\,*h}-(b_H+K)\one_{r-1})_1}
		{(\boldsymbol a-(b_H+K)\one_q)_1}
		\pFq{r+1}{q}
		{
			1,\ \boldsymbol b+(1-b_H-K)\one_r-\boldsymbol e_h
		}
		{
			\boldsymbol a+(1-b_H-K)\one_q
		}
		{x}.
	\end{equation}
	Each hypergeometric block is terminating; its degree is \(K\) if
	\(h=H\), and \(K-1\) if \(h\ne H\).

\end{theorem}

\begin{proof}
	\textit{Step 1: Residue formula and hypergeometric components for the form on the power vector.}
	We first turn to the \(A\)-form. The contour expression
	\eqref{eq:Laguerre-A-contour} has possible poles inside \(\Sigma\) only
	at the zeros of \(D_{\boldsymbol\beta,\boldsymbol m}\). These nodes are assumed to be
	pairwise distinct, and they are \(t=\beta_i+k\), with
	\(i\in\{1,\ldots,p\}\), \(m_i\ge1\), and
	\(k\in\{0,\ldots,m_i-1\}\). Hence all enclosed poles are simple.
	
	Evaluating the contour integral by residues gives
	\[
	\mathcal A_{\boldsymbol\lambda,\boldsymbol m}^{\mathrm L}(x)
	=
	\sum_{\substack{i=1\\m_i\ge1}}^{p}
	\sum_{k=0}^{m_i-1}
	c_{i,k}x^{\beta_i+k},
	\]
	where
	\[
	c_{i,k}
	\coloneq
	\operatorname*{Res}_{t=\beta_i+k}
	\left[
	\frac{\Gamma(t\one_r+\boldsymbol b+\boldsymbol n+\one_r)}
	{\Gamma(t\one_q+\boldsymbol a+\one_q)}
	\frac{1}{D_{\boldsymbol\beta,\boldsymbol m}(t)}
	\right].
	\]
	Equivalently,
	\[
	\mathcal A_{\boldsymbol\lambda,\boldsymbol m}^{\mathrm L}(x)
	=
	\sum_{i=1}^{p}
	A_{\boldsymbol\lambda,\boldsymbol m}^{(i)}(x)x^{\beta_i},
	\]
	where \(A_{\boldsymbol\lambda,\boldsymbol m}^{(i)}\equiv0\) if \(m_i=0\), and, for
	\(m_i\ge1\),
	\[
	A_{\boldsymbol\lambda,\boldsymbol m}^{(i)}(x)
	=
	\sum_{k=0}^{m_i-1}c_{i,k}x^k
	\in\mathbb P_{m_i-1}.
	\]
	This already gives the required expansion in the power vector and the
	degree bounds for its components.
	
	We compute the coefficients explicitly. Fix
	\(i\in\{1,\ldots,p\}\) with \(m_i\ge1\). At the node
	\(t=\beta_i+k\), direct evaluation of the simple residue gives
	\begin{equation}
		\label{eq:Laguerre-A-residue-coefficients}
		c_{i,k}
		=
		\frac{(-1)^{k+1}}{
			k!(m_i-k-1)!
			(\boldsymbol\beta^{\,*i}-\beta_i\one_{p-1}-k\one_{p-1})_{\boldsymbol m^{\,*i}}
		}
		\frac{
			\Gamma((\beta_i+k)\one_r+\boldsymbol b+\boldsymbol n+\one_r)
		}{
			\Gamma((\beta_i+k)\one_q+\boldsymbol a+\one_q)
		},
	\end{equation}
	for \(k\in\{0,\ldots,m_i-1\}\). In particular,
	\[
	c_{i,0}
	=
	-
	\frac{
		\Gamma((\beta_i)\one_r+\boldsymbol b+\boldsymbol n+\one_r)
	}{
		(m_i-1)!\,
		\Gamma((\beta_i)\one_q+\boldsymbol a+\one_q)\,
		(\boldsymbol\beta^{\,*i}-\beta_i\one_{p-1})_{\boldsymbol m^{\,*i}}
	}
	=
	\kappa_i^{\mathrm L,\mathcal A}(\boldsymbol\lambda,\boldsymbol m).
	\]
	For \(k\ge0\), the quotient \(c_{i,k}/c_{i,0}\) is obtained by the usual
	Pochhammer identities. One gets
	\[
	\frac{c_{i,k}}{c_{i,0}}
	=
	\frac{(-m_i+1)_k}{k!}
	\frac{
		((\beta_i)\one_r+\boldsymbol b+\boldsymbol n+\one_r)_k
		(\beta_i\one_{p-1}-\boldsymbol\beta^{\,*i}-\boldsymbol m^{\,*i}+\one_{p-1})_k
	}{
		((\beta_i)\one_q+\boldsymbol a+\one_q)_k
		(\beta_i\one_{p-1}-\boldsymbol\beta^{\,*i}+\one_{p-1})_k
	}.
	\]
	Therefore
	\[
	A_{\boldsymbol\lambda,\boldsymbol m}^{(i)}(x)
	=
	\kappa_i^{\mathrm L,\mathcal A}(\boldsymbol\lambda,\boldsymbol m)
	\sum_{k=0}^{m_i-1}
	\frac{(-m_i+1)_k}{k!}
	\frac{
		((\beta_i)\one_r+\boldsymbol b+\boldsymbol n+\one_r)_k
		(\beta_i\one_{p-1}-\boldsymbol\beta^{\,*i}-\boldsymbol m^{\,*i}+\one_{p-1})_k
	}{
		((\beta_i)\one_q+\boldsymbol a+\one_q)_k
		(\beta_i\one_{p-1}-\boldsymbol\beta^{\,*i}+\one_{p-1})_k
	}x^k.
	\]
	This is exactly the terminating hypergeometric polynomial
	\eqref{eq:Laguerre-A-components-hypergeometric}, with the normalization
	\eqref{eq:Laguerre-A-kappa}.
	
	\textit{Step 2: Orthogonality of the \(A\)-form against the shifted beta block.}
	Let \(h\in\{1,\ldots,r\}\) with \(n_h\ge1\), and let
	\(k\in\{0,\ldots,n_h-1\}\). We pair the contour expression
	\eqref{eq:Laguerre-A-contour} with \(x^k w_h(x;\boldsymbol a,\boldsymbol b)\).
	The exchange of the contour integral with the \(x\)-integral is legitimate.
	Indeed, \(\Sigma\) is a finite union of compact contours around the nodes
	\(\beta_i+j\), and it can be chosen so that, for every \(t\in\Sigma\), the
	Mellin parameter \(t+k+1\) lies in the common fundamental strip of the
	weights involved. Equivalently, the functions
	\(x^{t+k}w_h(x;\boldsymbol a,\boldsymbol b)\) are absolutely integrable uniformly for
	\(t\in\Sigma\). Therefore Fubini's theorem gives
	\begin{equation*}
	\int_0^\infty \mathcal A_{\boldsymbol\lambda,\boldsymbol m}^{\mathrm L}(x)x^k w_h(x;\boldsymbol a,\boldsymbol b)\dx
	=
	\frac{1}{2\pi\mathrm i}
	\int_\Sigma
	\frac{\Gamma(t\one_r+\boldsymbol b+\boldsymbol n+\one_r)}
	{\Gamma(t\one_q+\boldsymbol a+\one_q)D_{\boldsymbol\beta,\boldsymbol m}(t)}
	\left(\int_0^\infty x^{t+k}w_h(x;\boldsymbol a,\boldsymbol b)\dx\right)
	\dt.
	\end{equation*}
	Using \eqref{eq:wh-Mellin},
	\[
	\mathcal M[w_h](t+k+1)
	=
	\frac{\Gamma((t+k)\one_q+\boldsymbol a+\one_q)}
	{\Gamma((t+k)\one_r+\boldsymbol b+\one_r+\boldsymbol e_h)}.
	\]
	Thus the required pairing is
	\begin{equation*}
	\left\langle x^k\mathcal A_{\boldsymbol\lambda,\boldsymbol m}^{\mathrm L},w_h\right\rangle
	=
	\frac{1}{2\pi\mathrm i}
	\int_\Sigma
	\frac{R_{h,k}(t)}
	{D_{\boldsymbol\beta,\boldsymbol m}(t)}
	\dt,
	\end{equation*}
	where
	\[
	R_{h,k}(t)
	=
	\frac{\Gamma(t\one_r+\boldsymbol b+\boldsymbol n+\one_r)}
	{\Gamma(t\one_q+\boldsymbol a+\one_q)}
	\frac{\Gamma((t+k)\one_q+\boldsymbol a+\one_q)}
	{\Gamma((t+k)\one_r+\boldsymbol b+\one_r+\boldsymbol e_h)}.
	\]
	
	We check that \(R_{h,k}\) is a polynomial of sufficiently small
	degree. By \eqref{eq:Laguerre-admissibility-nn}, one has
	\(k\le n_h-1\le n_u\) for every \(u\in\{1,\ldots,r\}\). Hence
	\(\boldsymbol n-k\one_r-\boldsymbol e_h\in\mathbb N_0^r\), and
	\[
	R_{h,k}(t)
	=
	(t\one_q+\boldsymbol a+\one_q)_k
	(t\one_r+\boldsymbol b+(k+1)\one_r+\boldsymbol e_h)_{\boldsymbol n-k\one_r-\boldsymbol e_h}.
	\]
	This is a polynomial of degree
	\(qk+|\boldsymbol n|-rk-1=|\boldsymbol n|+sk-1\). Since
	\eqref{eq:Laguerre-admissibility-n-eta} gives
	\(sk\le s(n_h-1)\le|\boldsymbol\eta|\), we have
	\(R_{h,k}\in\mathbb P_{N-1}\).
	
	Now \(\deg D_{\boldsymbol\beta,\boldsymbol m}=|\boldsymbol m|=N+1\), so
	\(R_{h,k}(t)/D_{\boldsymbol\beta,\boldsymbol m}(t)=\mathrm{O}(t^{-2})\) as
	\(t\to\infty\). Hence its residue at infinity is zero. By the residue
	theorem, the sum of all finite residues is zero. Since \(\Sigma\)
	encloses precisely the poles at the nodes \(\beta_i+j\), with
	\(i\in\{1,\ldots,p\}\), \(m_i\ge1\), and
	\(j\in\{0,\ldots,m_i-1\}\), the contour integral vanishes. This proves
	\eqref{eq:A-orthogonality-wh}.
	
	\textit{Step 3: Orthogonality of the \(A\)-form against the Euler-derivative block.}
	Let \(\ell\in\{1,\ldots,s\}\) with \(\eta_\ell\ge1\), and let
	\(k\in\{0,\ldots,\eta_\ell-1\}\). We pair
	\eqref{eq:Laguerre-A-contour} with \(x^k v_\ell(x;\boldsymbol a,\boldsymbol b)\).
	The same justification of the exchange of integrations applies to the
	weights \(v_\ell\): after the integrations by parts used in
	\eqref{eq:vl-def}, the Mellin transform \eqref{eq:vl-Mellin} is valid in
	the corresponding fundamental strip, and the chosen compact contour
	\(\Sigma\) lies inside it for the Mellin parameters \(t+k+1\). Since
	\(\Sigma\) is a finite union of compact contours and the integrand is
	absolutely integrable uniformly for \(t\in\Sigma\), Fubini's theorem
	allows us to interchange the order of integration. Hence
	\begin{align*}
	\left\langle
	x^k\mathcal A_{\boldsymbol\lambda,\boldsymbol m}^{\mathrm L},v_\ell
	\right\rangle
	&=
	\int_0^\infty
	\mathcal A_{\boldsymbol\lambda,\boldsymbol m}^{\mathrm L}(x)
	x^k v_\ell(x;\boldsymbol a,\boldsymbol b)\dx
	\\
	&=
	\frac{1}{2\pi\mathrm i}
	\int_\Sigma
	\frac{\Gamma(t\one_r+\boldsymbol b+\boldsymbol n+\one_r)}
	{\Gamma(t\one_q+\boldsymbol a+\one_q)D_{\boldsymbol\beta,\boldsymbol m}(t)}
	\left(
	\int_0^\infty x^{t+k}v_\ell(x;\boldsymbol a,\boldsymbol b)\dx
	\right)\dt.
	\end{align*}
	Using \eqref{eq:vl-Mellin},
	\[
	\mathcal M[v_\ell](t+k+1)
	=
	(t+k+1)^{\ell-1}
	\frac{\Gamma((t+k)\one_q+\boldsymbol a+\one_q)}
	{\Gamma((t+k)\one_r+\boldsymbol b+\one_r)}.
	\]
	Therefore the required pairing is
	\begin{equation*}
	\left\langle x^k\mathcal A_{\boldsymbol\lambda,\boldsymbol m}^{\mathrm L},v_\ell\right\rangle = \frac{1}{2\pi\mathrm i}\int_\Sigma \frac{S_{\ell,k}(t)}{D_{\boldsymbol\beta,\boldsymbol m}(t)}\dt,
	\end{equation*}
	with
	\[
	S_{\ell,k}(t)
	=
	\frac{\Gamma(t\one_r+\boldsymbol b+\boldsymbol n+\one_r)}
	{\Gamma(t\one_q+\boldsymbol a+\one_q)}
	\frac{\Gamma((t+k)\one_q+\boldsymbol a+\one_q)}
	{\Gamma((t+k)\one_r+\boldsymbol b+\one_r)}
	(t+k+1)^{\ell-1}.
	\]
	
	By \eqref{eq:Laguerre-admissibility-etan}, one has
	\(k\le\eta_\ell-1\le n_u\) for every \(u\in\{1,\ldots,r\}\). Thus
	\(\boldsymbol n-k\one_r\in\mathbb N_0^r\), and
	\[
	S_{\ell,k}(t)
	=
	(t\one_q+\boldsymbol a+\one_q)_k
	(t\one_r+\boldsymbol b+(k+1)\one_r)_{\boldsymbol n-k\one_r}
	(t+k+1)^{\ell-1}.
	\]
	This is a polynomial of degree
	\(qk+|\boldsymbol n|-rk+\ell-1=|\boldsymbol n|+sk+\ell-1\). By
	\eqref{eq:Laguerre-admissibility-eta-step}, one has
	\(sk+\ell\le s(\eta_\ell-1)+\ell\le|\boldsymbol\eta|\). Hence
	\(S_{\ell,k}\in\mathbb P_{N-1}\).
	
	Again, \(\deg D_{\boldsymbol\beta,\boldsymbol m}=|\boldsymbol m|=N+1\), and therefore
	\(S_{\ell,k}(t)/D_{\boldsymbol\beta,\boldsymbol m}(t)=\mathrm{O}(t^{-2})\) as
	\(t\to\infty\). Its residue at infinity is zero, so the sum of the
	finite residues enclosed by \(\Sigma\) vanishes. The contour integral is
	zero, and \eqref{eq:A-orthogonality-vl} follows.
	
	\textit{Step 4: Mellin transforms of the elementary building blocks for the \(B\)-form.}
	We compute the Mellin transforms of the terms that appear in
	the \(B\)-form. The purpose of this step is to show that, after taking out the
	common factor
	\[
	\frac{\Gamma(z\one_q+\boldsymbol a)}{\Gamma(z\one_r+\boldsymbol b+\boldsymbol n)},
	\]
	the remaining factor should be a polynomial in \(z\), of degree at most
	\(N-1\). This is what allows us to reconstruct the \(B\)-form by solving
	a finite polynomial identity.
	
	For the shifted beta weights and the Euler-derivative weights, we have
	\[
	\mathcal M[w_h](z)
	=
	\frac{\Gamma(z\one_q+\boldsymbol a)}
	{\Gamma(z\one_r+\boldsymbol b+\boldsymbol e_h)},
	\qquad
	h\in\{1,\ldots,r\},
	\]
	and
	\[
	\mathcal M[v_\ell](z)
	=
	z^{\ell-1}
	\frac{\Gamma(z\one_q+\boldsymbol a)}
	{\Gamma(z\one_r+\boldsymbol b)},
	\qquad
	\ell\in\{1,\ldots,s\}.
	\]
	Multiplication by \(x^k\) shifts the Mellin variable. Thus
	\(\mathcal M[x^k w_h](z)=\mathcal M[w_h](z+k)\), and we get
	\begin{equation}
		\label{eq:Laguerre-Mellin-xk-wh-first}
		\mathcal M[x^k w_h](z)
		=
		\frac{\Gamma(z\one_q+\boldsymbol a)}
		{\Gamma(z\one_r+\boldsymbol b+\boldsymbol n)}
		(z\one_q+\boldsymbol a)_k
		\frac{\Gamma(z\one_r+\boldsymbol b+\boldsymbol n)}
		{\Gamma(z\one_r+\boldsymbol b+k\one_r+\boldsymbol e_h)}.
	\end{equation}
	Similarly, since \(\mathcal M[x^k v_\ell](z)=\mathcal M[v_\ell](z+k)\),
	we obtain
	\begin{equation}
		\label{eq:Laguerre-Mellin-xk-vl-first}
		\mathcal M[x^k v_\ell](z)
		=
		\frac{\Gamma(z\one_q+\boldsymbol a)}
		{\Gamma(z\one_r+\boldsymbol b+\boldsymbol n)}
		(z\one_q+\boldsymbol a)_k
		\frac{\Gamma(z\one_r+\boldsymbol b+\boldsymbol n)}
		{\Gamma(z\one_r+\boldsymbol b+k\one_r)}
		(z+k)^{\ell-1}.
	\end{equation}
	
	We use Laguerre admissibility. Let
	\(h\in\{1,\ldots,r\}\) with \(n_h\ge1\), and let
	\(k\in\{0,\ldots,n_h-1\}\). By
	\eqref{eq:Laguerre-admissibility-nn}, we have
	\(k\le n_h-1\le n_u\) for every \(u\in\{1,\ldots,r\}\). Hence
	\(\boldsymbol n-k\one_r-\boldsymbol e_h\in\mathbb N_0^r\). Therefore the quotient of gamma
	functions in \eqref{eq:Laguerre-Mellin-xk-wh-first} is a finite
	Pochhammer product:
	\[
	\frac{\Gamma(z\one_r+\boldsymbol b+\boldsymbol n)}
	{\Gamma(z\one_r+\boldsymbol b+k\one_r+\boldsymbol e_h)}
	=
	(z\one_r+\boldsymbol b+k\one_r+\boldsymbol e_h)_{\boldsymbol n-k\one_r-\boldsymbol e_h}.
	\]
	Thus
	\begin{equation}
		\label{eq:Laguerre-Mellin-xk-wh-proof}
		\mathcal M[x^k w_h](z)
		=
		\frac{\Gamma(z\one_q+\boldsymbol a)}
		{\Gamma(z\one_r+\boldsymbol b+\boldsymbol n)}
		(z\one_q+\boldsymbol a)_k
		(z\one_r+\boldsymbol b+k\one_r+\boldsymbol e_h)_{\boldsymbol n-k\one_r-\boldsymbol e_h}.
	\end{equation}
	
	Now let \(\ell\in\{1,\ldots,s\}\) with \(\eta_\ell\ge1\), and let
	\(k\in\{0,\ldots,\eta_\ell-1\}\). By
	\eqref{eq:Laguerre-admissibility-etan}, we have
	\(k\le\eta_\ell-1\le n_u\) for every \(u\in\{1,\ldots,r\}\). Hence
	\(\boldsymbol n-k\one_r\in\mathbb N_0^r\). Therefore
	\[
	\frac{\Gamma(z\one_r+\boldsymbol b+\boldsymbol n)}
	{\Gamma(z\one_r+\boldsymbol b+k\one_r)}
	=
	(z\one_r+\boldsymbol b+k\one_r)_{\boldsymbol n-k\one_r}.
	\]
	Substituting this into \eqref{eq:Laguerre-Mellin-xk-vl-first} gives
	\begin{equation}
		\label{eq:Laguerre-Mellin-xk-vl-proof}
		\mathcal M[x^k v_\ell](z)
		=
		\frac{\Gamma(z\one_q+\boldsymbol a)}
		{\Gamma(z\one_r+\boldsymbol b+\boldsymbol n)}
		(z\one_q+\boldsymbol a)_k
		(z\one_r+\boldsymbol b+k\one_r)_{\boldsymbol n-k\one_r}
		(z+k)^{\ell-1}.
	\end{equation}
	
	The polynomial factors in
	\eqref{eq:Laguerre-Mellin-xk-wh-proof} and
	\eqref{eq:Laguerre-Mellin-xk-vl-proof} are the two finite polynomial
	families that enter the coefficient system below.
	
	It remains to check the degree bounds. In
	\eqref{eq:Laguerre-Mellin-xk-wh-proof}, the polynomial factor is
	\[
	(z\one_q+\boldsymbol a)_k
	(z\one_r+\boldsymbol b+k\one_r+\boldsymbol e_h)_{\boldsymbol n-k\one_r-\boldsymbol e_h}.
	\]
	The first factor has degree \(qk\), and the second one has degree
	\(|\boldsymbol n|-rk-1\). Hence the total degree is
	\[
	qk+|\boldsymbol n|-rk-1=|\boldsymbol n|+sk-1.
	\]
	Since \(k\le n_h-1\), condition
	\eqref{eq:Laguerre-admissibility-n-eta} gives
	\(sk\le s(n_h-1)\le|\boldsymbol\eta|\). Therefore
	\[
	|\boldsymbol n|+sk-1\le |\boldsymbol n|+|\boldsymbol\eta|-1=N-1.
	\]
	So the polynomial factor in \eqref{eq:Laguerre-Mellin-xk-wh-proof}
	belongs to \(\mathbb P_{N-1}\).
	
	In \eqref{eq:Laguerre-Mellin-xk-vl-proof}, the polynomial factor is
	\[
	(z\one_q+\boldsymbol a)_k
	(z\one_r+\boldsymbol b+k\one_r)_{\boldsymbol n-k\one_r}
	(z+k)^{\ell-1}.
	\]
	Its degree is
	\[
	qk+|\boldsymbol n|-rk+\ell-1=|\boldsymbol n|+sk+\ell-1.
	\]
	Since \(k\le\eta_\ell-1\), condition
	\eqref{eq:Laguerre-admissibility-eta-step} gives
	\[
	sk+\ell\le s(\eta_\ell-1)+\ell\le|\boldsymbol\eta|.
	\]
	Hence
	\[
	|\boldsymbol n|+sk+\ell-1\le|\boldsymbol n|+|\boldsymbol\eta|-1=N-1.
	\]
	Thus the polynomial factor in \eqref{eq:Laguerre-Mellin-xk-vl-proof}
	also belongs to \(\mathbb P_{N-1}\).
	
	We have therefore proved that every polynomial combination in
	\eqref{eq:Laguerre-B-form} has a Mellin transform of the form
	\begin{equation}
		\label{eq:Laguerre-common-Mellin-polynomial-form}
		\mathcal M[\mathcal B](z)
		=
		\frac{\Gamma(z\one_q+\boldsymbol a)}
		{\Gamma(z\one_r+\boldsymbol b+\boldsymbol n)}
		F(z),
		\qquad
		F\in\mathbb P_{N-1}.
	\end{equation}
	
	\textit{Step 5: Determination of the form on the Laguerre I vector.}
	We choose the polynomial \(F\) in
	\eqref{eq:Laguerre-common-Mellin-polynomial-form}. The choice is dictated
	by the orthogonality conditions against the power vector. Thus the
	problem is reduced to expanding a prescribed polynomial in a finite
	family of polynomials.
	
	Let \(\mathcal V_{\boldsymbol\lambda}\) be the space of component vectors
	\[
	\left[
	\{B^{(h)}\}_{h=1}^{r},
	\{D^{(\ell)}\}_{\ell=1}^{s}
	\right]
	\]
	with \(B^{(h)}\in\mathbb P_{n_h-1}\) if \(n_h\ge1\), and
	\(B^{(h)}\equiv0\) if \(n_h=0\), together with
	\(D^{(\ell)}\in\mathbb P_{\eta_\ell-1}\) if \(\eta_\ell\ge1\), and
	\(D^{(\ell)}\equiv0\) if \(\eta_\ell=0\). This space has dimension
	\[
	\dim\mathcal V_{\boldsymbol\lambda}=|\boldsymbol n|+|\boldsymbol\eta|=N.
	\]
	By Step~4, the Mellin transform defines a linear map
	\[
	\Phi_{\boldsymbol\lambda}:\mathcal V_{\boldsymbol\lambda}\longrightarrow
	\mathbb P_{N-1}
	\]
	through the identity
	\begin{equation}
		\label{eq:Laguerre-Phi-definition-proof}
		\mathcal M[\mathcal B](z)
		=
		\frac{\Gamma(z\one_q+\boldsymbol a)}
		{\Gamma(z\one_r+\boldsymbol b+\boldsymbol n)}
		\Phi_{\boldsymbol\lambda}
		\left(
		\{B^{(h)}\}_{h=1}^{r},
		\{D^{(\ell)}\}_{\ell=1}^{s}
		\right)(z).
	\end{equation}
	
	If
	\[
	B^{(h)}(x)=\sum_{k=0}^{n_h-1}b_{h,k}x^k,
	\qquad
	D^{(\ell)}(x)=\sum_{k=0}^{\eta_\ell-1}d_{\ell,k}x^k,
	\]
	then this map is explicitly
	\begin{multline}
		\label{eq:Laguerre-Phi-expanded}
		\Phi_{\boldsymbol\lambda}
		\left(
		\{B^{(h)}\}_{h=1}^{r},
		\{D^{(\ell)}\}_{\ell=1}^{s}
		\right)(z)
		=
		\sum_{\substack{h=1\\n_h\ge1}}^{r}
		\sum_{k=0}^{n_h-1}
		b_{h,k}
		(z\one_q+\boldsymbol a)_k
		\frac{\Gamma(z\one_r+\boldsymbol b+\boldsymbol n)}
		{\Gamma(z\one_r+\boldsymbol b+k\one_r+\boldsymbol e_h)}
		\\*
		+
		\sum_{\substack{\ell=1\\\eta_\ell\ge1}}^{s}
		\sum_{k=0}^{\eta_\ell-1}
		d_{\ell,k}
		(z\one_q+\boldsymbol a)_k
		\frac{\Gamma(z\one_r+\boldsymbol b+\boldsymbol n)}
		{\Gamma(z\one_r+\boldsymbol b+k\one_r)}
		(z+k)^{\ell-1}.
	\end{multline}
	Using the Pochhammer rewritings from Step~4, this becomes
	\begin{multline}
		\label{eq:Laguerre-Phi-Pochhammer-expanded}
		\Phi_{\boldsymbol\lambda}
		\left(
		\{B^{(h)}\}_{h=1}^{r},
		\{D^{(\ell)}\}_{\ell=1}^{s}
		\right)(z)
		=
		\sum_{\substack{h=1\\n_h\ge1}}^{r}
		\sum_{k=0}^{n_h-1}
		b_{h,k}
		(z\one_q+\boldsymbol a)_k
		(z\one_r+\boldsymbol b+k\one_r+\boldsymbol e_h)_{\boldsymbol n-k\one_r-\boldsymbol e_h}
		\\*
		+
		\sum_{\substack{\ell=1\\\eta_\ell\ge1}}^{s}
		\sum_{k=0}^{\eta_\ell-1}
		d_{\ell,k}
		(z\one_q+\boldsymbol a)_k
		(z\one_r+\boldsymbol b+k\one_r)_{\boldsymbol n-k\one_r}
		(z+k)^{\ell-1}.
	\end{multline}
	So the coefficient problem for
	\(\mathcal B_{\boldsymbol\lambda,\boldsymbol m}^{\mathrm L}\) is exactly the problem
	of expanding a prescribed polynomial in \(\mathbb P_{N-1}\) in the
	finite family appearing in
	\eqref{eq:Laguerre-Phi-Pochhammer-expanded}. Equivalently, the
	coefficients \(b_{h,k}\) and \(d_{\ell,k}\) satisfy
		\begin{multline}
			\label{eq:B-component-polynomial-system}
			\sum_{\substack{h=1\\n_h\ge1}}^{r}
			\sum_{k=0}^{n_h-1}
			b_{h,k}
			(z\one_q+\boldsymbol a)_k
			(z\one_r+\boldsymbol b+k\one_r+\boldsymbol e_h)_{\boldsymbol n-k\one_r-\boldsymbol e_h}
			\\*
			+
			\sum_{\substack{\ell=1\\\eta_\ell\ge1}}^{s}
			\sum_{k=0}^{\eta_\ell-1}
		d_{\ell,k}
		(z\one_q+\boldsymbol a)_k
		(z\one_r+\boldsymbol b+k\one_r)_{\boldsymbol n-k\one_r}
			(z+k)^{\ell-1}
			=
			(\boldsymbol\beta+(1-z)\one_p)_{\boldsymbol m}.
		\end{multline}
	
	The matrix of \(\Phi_{\boldsymbol\lambda}\), in the weighted
	coefficient basis and the monomial basis of \(\mathbb P_{N-1}\),
	is the coefficient matrix of
	Proposition~\ref{prop:AT-normality-Laguerre-admissible}, up to
	transposition and the fixed row ordering. Its determinant is
	\(c_{\boldsymbol\lambda}\ne0\). Hence
	\(\Phi_{\boldsymbol\lambda}\) is an isomorphism, proving uniqueness
	of the expansion without any appeal to an AT property.
	
	We choose
	\[
	F(z)=(\boldsymbol\beta+(1-z)\one_p)_{\boldsymbol m}.
	\]
	Because \(|\boldsymbol m|=N-1\), this polynomial belongs to
	\(\mathbb P_{N-1}\). Since \(\Phi_{\boldsymbol\lambda}\) is an isomorphism,
	there is a unique component vector such that
	\[
	\Phi_{\boldsymbol\lambda}
	\left(
	\{B^{(h)}\}_{h=1}^{r},
	\{D^{(\ell)}\}_{\ell=1}^{s}
	\right)(z)
	=
	(\boldsymbol\beta+(1-z)\one_p)_{\boldsymbol m}.
	\]
	Substitution in \eqref{eq:Laguerre-Phi-definition-proof} gives
	\begin{equation}
		\label{eq:Laguerre-B-Mellin-proof}
		\mathcal M[\mathcal B_{\boldsymbol\lambda,\boldsymbol m}^{\mathrm L}](z)
		=
		\frac{\Gamma(z\one_q+\boldsymbol a)}
		{\Gamma(z\one_r+\boldsymbol b+\boldsymbol n)}
		(\boldsymbol\beta+(1-z)\one_p)_{\boldsymbol m},
	\end{equation}
	which is \eqref{eq:Laguerre-B-Mellin}.
	
	The orthogonality now follows directly from the zeros of this Mellin
	transform. Let \(i\in\{1,\ldots,p\}\) with \(m_i\ge1\), and let
	\(k\in\{0,\ldots,m_i-1\}\). At \(z=\beta_i+k+1\), the factor
	\((\beta_i+1-z)_{m_i}=(-k)_{m_i}\) vanishes. Therefore
	\[
	\mathcal M[\mathcal B_{\boldsymbol\lambda,\boldsymbol m}^{\mathrm L}]
	(\beta_i+k+1)=0,
	\]
	which means exactly that
	\[
	\int_0^\infty
	\mathcal B_{\boldsymbol\lambda,\boldsymbol m}^{\mathrm L}(x)
	x^{\beta_i+k}\dx=0.
	\]
	This proves \eqref{eq:Laguerre-B-orthogonality}.

	\textit{Step 6: Effective finite-pole residues and the shifted-beta components.}
	We prove the explicit formula
	\eqref{eq:Laguerre-B-components-effective-hypergeometric}.  Set
	\begin{equation}
		\label{eq:Laguerre-B-rational-function}
		R_{\boldsymbol\lambda,\boldsymbol m}^{\mathcal B,\mathrm L}(z)
		\coloneq
		\frac{(\boldsymbol\beta+(1-z)\one_p)_{\boldsymbol m}}{(z\one_r+\boldsymbol b)_{\boldsymbol n}}.
	\end{equation}
	Dividing the identity \eqref{eq:B-component-polynomial-system} by
	\((z\one_r+\boldsymbol b)_{\boldsymbol n}\) is equivalent to decomposing \(R\) into partial
	fractions at the finite beta poles,
	\begin{equation}
		\label{eq:Laguerre-B-rational-decomposition}
		R_{\boldsymbol\lambda,\boldsymbol m}^{\mathcal B,\mathrm L}(z)
		=
		\Pi_{\boldsymbol\lambda,\boldsymbol m}^{\mathcal B,\mathrm L}(z)
		+
		\sum_{\substack{H=1\\ n_H\ge1}}^{r}
		\sum_{K=0}^{n_H-1}
		\frac{\rho_{H,K}^{\mathcal B,\mathrm L}}{z+b_H+K},
	\end{equation}
	where \(\Pi_{\boldsymbol\lambda,\boldsymbol m}^{\mathcal B,\mathrm L}\) is the
	polynomial part at infinity.  Evaluating the simple residue of \(R\) at
	\(z=-b_H-K\) gives
	\begin{equation}
		\label{eq:Laguerre-B-rho-finite-poles}
		\rho_{H,K}^{\mathcal B,\mathrm L}
		\coloneq
		(-1)^K
		\frac{
			(\boldsymbol\beta+(b_H+K+1)\one_p)_{\boldsymbol m}
		}{
			K!(n_H-K-1)!
			(\boldsymbol b^{\,*H}-(b_H+K)\one_{r-1})_{\boldsymbol n^{\,*H}}
		}.
	\end{equation}

	The polynomial part is expanded in the polynomial parts of the
	confluent extraction symbols.  For \(\ell\in\{1,\ldots,s\}\) and
	\(k\in\{0,\ldots,\eta_\ell-1\}\), define
	\begin{equation}
		\label{eq:Laguerre-B-tilde-Q-symbol}
		\widetilde Q_{\ell,k}^{\mathrm L}(z)
		\coloneq
		\frac{(z\one_q+\boldsymbol a)_k}{(z\one_r+\boldsymbol b)_k}(z+k)_{\ell-1},
		\qquad
		E_{\ell,k}^{\mathrm L}(z)
		\coloneq
		\left[\widetilde Q_{\ell,k}^{\mathrm L}(z)\right]_+,
	\end{equation}
	where \([\cdot]_+\) denotes the polynomial part at \(z=\infty\), and write
	\begin{equation}
		\label{eq:Laguerre-B-Pi-expansion}
		\Pi_{\boldsymbol\lambda,\boldsymbol m}^{\mathcal B,\mathrm L}(z)
		=
		\sum_{\nu=0}^{|\boldsymbol\eta|-1}\pi_\nu z^\nu,
		\qquad
		E_{\ell,k}^{\mathrm L}(z)
		=
		\sum_{\nu=0}^{|\boldsymbol\eta|-1}e_{\nu;\ell,k}^{\mathrm L}z^\nu.
	\end{equation}
	Since \(\boldsymbol\eta\) is on the step-line, the degrees \(sk+\ell-1\),
	with \(\ell\in\{1,\ldots,s\}\) and \(k\in\{0,\ldots,\eta_\ell-1\}\), are
	precisely \(0,\ldots,|\boldsymbol\eta|-1\), and \(E_{\ell,k}^{\mathrm L}\) is
	monic of degree \(sk+\ell-1\).  If \(|\boldsymbol\eta|=0\) the construction is
	empty; otherwise let
	\begin{equation}
		\label{eq:Laguerre-B-E-matrix}
		\mathcal E_{\boldsymbol\eta}^{\mathrm L}
		\coloneq
		\left[e_{\nu;\ell,k}^{\mathrm L}\right]_{
		\substack{\nu\in\{0,\ldots,|\boldsymbol\eta|-1\}\\
		\ell\in\{1,\ldots,s\},\ k\in\{0,\ldots,\eta_\ell-1\}}},
	\end{equation}
	with columns ordered by the degree \(sk+\ell-1\).  By the previous
	remark \(\mathcal E_{\boldsymbol\eta}^{\mathrm L}\) is unitriangular in this
	ordering, so the connection coefficients defined by
	\begin{equation}
		\label{eq:Laguerre-B-delta-connection}
		\left[\delta_{\ell,k}^{\mathcal B,\mathrm L}\right]_{\ell,k}
		=
		\left(\mathcal E_{\boldsymbol\eta}^{\mathrm L}\right)^{-1}
		[\pi_0,\ldots,\pi_{|\boldsymbol\eta|-1}]^{\mathsf T},
	\end{equation}
	equivalently, by Cramer's rule,
	\begin{equation}
		\label{eq:Laguerre-B-delta-Cramer}
		\delta_{\ell,k}^{\mathcal B,\mathrm L}
		=
		\frac{
			\det \mathcal E_{\boldsymbol\eta}^{\mathrm L}(\ell,k)
		}{
			\det \mathcal E_{\boldsymbol\eta}^{\mathrm L}
		},
	\end{equation}
	where \(\mathcal E_{\boldsymbol\eta}^{\mathrm L}(\ell,k)\) is obtained from
	\(\mathcal E_{\boldsymbol\eta}^{\mathrm L}\) by replacing the column indexed by
	\((\ell,k)\) with \([\pi_0,\ldots,\pi_{|\boldsymbol\eta|-1}]^{\mathsf T}\), give the
	unique expansion
	\begin{equation}
		\label{eq:Laguerre-B-polynomial-part-delta-proof}
		\Pi_{\boldsymbol\lambda,\boldsymbol m}^{\mathcal B,\mathrm L}(z)
		=
		\sum_{\ell=1}^{s}
		\sum_{k=0}^{\eta_\ell-1}
		\delta_{\ell,k}^{\mathcal B,\mathrm L}
		E_{\ell,k}^{\mathrm L}(z).
	\end{equation}
	The entries of \(\mathcal E_{\boldsymbol\eta}^{\mathrm L}\) are explicit: with
	\(u=z^{-1}\), put
	\begin{equation}
		\label{eq:Laguerre-B-Ak-Bk}
		A_k(u)
		\coloneq
		\prod_{\rho=1}^{q}\prod_{j=0}^{k-1}\left(1+(a_\rho+j)u\right),
		\qquad
		B_k(u)
		\coloneq
		\prod_{h=1}^{r}\prod_{j=0}^{k-1}\left(1+(b_h+j)u\right);
	\end{equation}
	and, with the empty product understood as one,
	\begin{equation}
		\label{eq:Laguerre-B-C-ell-k}
		C_{\ell,k}(u)
		\coloneq
		\prod_{j=1}^{\ell-1}\left(1+(k+j-1)u\right).
	\end{equation}
	Then, from
	\(
	\widetilde Q_{\ell,k}^{\mathrm L}(z)
	= z^{sk+\ell-1} C_{\ell,k}(u)A_k(u)/B_k(u)
	\),
	\begin{equation}
		\label{eq:Laguerre-B-E-coefficients}
		e_{\nu;\ell,k}^{\mathrm L}
		=
		\left[u^{sk+\ell-1-\nu}\right]
		C_{\ell,k}(u)\frac{A_k(u)}{B_k(u)}.
	\end{equation}

	For the same symbol \(\widetilde Q_{\ell,k}^{\mathrm L}\), the finite
	beta-pole contribution is obtained by ordinary residues.  For
	\(H\in\{1,\ldots,r\}\), \(K\in\{0,\ldots,n_H-1\}\), and \(k\ge K+1\), define
	\begin{equation}
		\label{eq:Laguerre-B-sigma-tail}
		\sigma_{H,K}^{(\ell,k)}
		\coloneq
		\operatorname*{Res}_{z=-b_H-K}
		\widetilde Q_{\ell,k}^{\mathrm L}(z)
		=
		\frac{
			(k-b_H-K)_{\ell-1}(\boldsymbol a-(b_H+K)\one_q)_k
		}{
			(-1)^K K!(k-1-K)!
			(\boldsymbol b^{\,*H}-(b_H+K)\one_{r-1})_k
		}.
	\end{equation}
	The same confluent symbol is therefore used both at infinity and at the
	finite beta nodes.  Since the polynomial part of
	\(\sum_{\ell,k}\delta_{\ell,k}^{\mathcal B,\mathrm L}
	\widetilde Q_{\ell,k}^{\mathrm L}\) is
	\(\Pi_{\boldsymbol\lambda,\boldsymbol m}^{\mathcal B,\mathrm L}\) by
	\eqref{eq:Laguerre-B-polynomial-part-delta-proof}, the rational function
	\begin{equation}
		\label{eq:Laguerre-B-tildeQ-subtracted-identity}
		R_{\boldsymbol\lambda,\boldsymbol m}^{\mathcal B,\mathrm L}(z)
		-
		\sum_{\ell=1}^{s}
		\sum_{k=0}^{\eta_\ell-1}
		\delta_{\ell,k}^{\mathcal B,\mathrm L}
		\widetilde Q_{\ell,k}^{\mathrm L}(z)
	\end{equation}
	has no polynomial part at \(z=\infty\).  Its beta-pole residues are
	therefore exactly
	\begin{equation}
		\label{eq:Laguerre-B-effective-rho}
		\widehat\rho_{H,K}^{\mathcal B,\mathrm L}
		\coloneq
		\rho_{H,K}^{\mathcal B,\mathrm L}
		-
		\sum_{\ell=1}^{s}
		\sum_{k=K+1}^{\eta_\ell-1}
		\delta_{\ell,k}^{\mathcal B,\mathrm L}
		\sigma_{H,K}^{(\ell,k)},
	\end{equation}
	where an inner sum with upper limit smaller than its lower limit is
	absent.  Equivalently,
	\begin{equation}
		\label{eq:Laguerre-B-effective-rational-decomposition}
		R_{\boldsymbol\lambda,\boldsymbol m}^{\mathcal B,\mathrm L}(z)
		-
		\sum_{\ell=1}^{s}
		\sum_{k=0}^{\eta_\ell-1}
		\delta_{\ell,k}^{\mathcal B,\mathrm L}
		\widetilde Q_{\ell,k}^{\mathrm L}(z)
		=
		\sum_{\substack{H=1\\ n_H\ge1}}^{r}
		\sum_{K=0}^{n_H-1}
		\frac{\widehat\rho_{H,K}^{\mathcal B,\mathrm L}}{z+b_H+K}.
	\end{equation}
	The inequalities in
	Proposition~\ref{prop:Laguerre-admissibility-inequalities} imply that every finite
	beta pole arising from these connection coefficients satisfies
	\(K\le n_H-1\), so all poles lie in the displayed range of
	\eqref{eq:Laguerre-B-components-effective-hypergeometric}.

We show that the shifted-beta blocks
	\(\mathscr R_{h;H,K}^{\mathrm L}\) realize the effective finite-pole
	residues \(\widehat\rho_{H,K}^{\mathcal B,\mathrm L}\). The pole-by-pole
	realization available in the Jacobi case \(s=0\) is
	\emph{not} available here: when \(s\ge1\) the reduced rational symbol of a
	single shifted-beta block is not a pure simple fraction. Each finite pole is
	instead reconstructed by its shifted-beta
	block \emph{modulo the Euler-symbol space}.

	Fix a finite beta node and set \(B=b_H+K\). Let
		\begin{equation}
			\label{eq:Laguerre-B-block-reduced-symbol}
			\begin{aligned}
			\mathscr W_{H,K}^{\mathrm L}(z)
			&\coloneq
			\sum_{h=1}^{r}
			\sum_{j=0}^{K-1+\delta_{h,H}}
			\tau_{h,j}^{H,K}
			 \frac{(z\one_q+\boldsymbol a)_j}{(z\one_r+\boldsymbol b)_j\,(z+b_h+j)},
			\\
			\tau_{h,j}^{H,K}
			&=
			\frac{(\boldsymbol a-b_h\one_q)_1}{(\boldsymbol b^{\,*h}-b_h\one_{r-1})_1}
			 \frac{(\boldsymbol b^{\,*h}-B\one_{r-1})_{j+1}(b_h-B)_j}{(\boldsymbol a-B\one_q)_{j+1}},
			\end{aligned}
		\end{equation}
	be the reduced rational symbol of the block vector
	\([\mathscr R_{1;H,K}^{\mathrm L},\ldots,\mathscr R_{r;H,K}^{\mathrm L}]\);
	the coefficients \(\tau_{h,j}^{H,K}\) are the Jacobi reconstruction
	coefficients of
	\eqref{eq:Jacobi-rational-reconstruction-identity}, evaluated at the present
	parameters, and the expansion of the hypergeometric block
	\eqref{eq:Laguerre-B-hypergeometric-block-final} agrees coefficientwise with
	\eqref{eq:Laguerre-B-block-reduced-symbol}. We prove
	\begin{equation}
		\label{eq:Laguerre-B-block-modulo-euler}
		\frac{1}{z+b_H+K}
		-
		\mathscr W_{H,K}^{\mathrm L}(z)
		\in
		\operatorname{span}
		\bigl\{
		\widetilde Q_{\ell,j}^{\mathrm L}:\
		\ell\in\{1,\ldots,s\},\ j\in\{0,\ldots,K-1\}
		\bigr\}.
	\end{equation}

	\noindent\emph{The modified telescoping.} For \(j\ge0\) put
	\begin{equation}
		\label{eq:Laguerre-B-telescoping-Ej}
		\mathcal E_j(z;B)
		\coloneq
		\frac{(\boldsymbol b-B\one_r)_j}{(\boldsymbol a-B\one_q)_j}\,
		\frac{(z\one_q+\boldsymbol a)_j}{(z\one_r+\boldsymbol b)_j},
		\qquad
		\mathcal E_0(z;B)=1.
	\end{equation}
	A direct computation gives, for each \(j\), the transition identity
	\begin{equation}
		\label{eq:Laguerre-B-transition}
		\frac{\mathcal E_j(z;B)}{z+B}
		=
		\sum_{h=1}^{r}
		\tau_{h,j}^{H,K}
		\frac{(z\one_q+\boldsymbol a)_j}{(z\one_r+\boldsymbol b)_j\,(z+b_h+j)}
		+
		\frac{\mathcal E_{j+1}(z;B)}{z+B}
		+
		\frac{(z\one_q+\boldsymbol a)_j}{(z\one_r+\boldsymbol b)_j}\,P_j(z;B),
	\end{equation}
	where \(P_j(\,\cdot\,;B)\) is a polynomial. For \(s=0\) the remainder
	\(P_j\) is absent, whereas for \(s\ge1\)
	it is a nonzero element of \(\mathbb P_{s-1}\).

	To see this, divide \eqref{eq:Laguerre-B-transition} by
	\((z\one_q+\boldsymbol a)_j/(z\one_r+\boldsymbol b)_j\), which gives
	\begin{equation}
		\label{eq:Laguerre-B-Pj-explicit}
		P_j(z;B)
		=
		\frac{(\boldsymbol b-B\one_r)_j}{(\boldsymbol a-B\one_q)_j}\frac{1}{z+B}
		-
		\sum_{h=1}^{r}\frac{\tau_{h,j}^{H,K}}{z+b_h+j}
		-
		\frac{(\boldsymbol b-B\one_r)_{j+1}}{(\boldsymbol a-B\one_q)_{j+1}}
		\frac{(z\one_q+\boldsymbol a+j\one_q)_1}{(z\one_r+\boldsymbol b+j\one_r)_1}\frac{1}{z+B}.
	\end{equation}
	At \(z=-B\) the residue of the first term is
	\((\boldsymbol b-B\one_r)_j/(\boldsymbol a-B\one_q)_j\), while that of the third term is
	\[
	-\frac{(\boldsymbol b-B\one_r)_{j+1}}{(\boldsymbol a-B\one_q)_{j+1}}
	\frac{(\boldsymbol a-B\one_q+j\one_q)_1}{(\boldsymbol b-B\one_r+j\one_r)_1}
	=
	-\frac{(\boldsymbol b-B\one_r)_j}{(\boldsymbol a-B\one_q)_j},
	\]
	so the two cancel. At each node \(z=-b_h-j\) the residue of the third term
	equals \(-\tau_{h,j}^{H,K}\), which cancels the residue of the corresponding
	summand in the middle term. Hence the right-hand side of
	\eqref{eq:Laguerre-B-Pj-explicit} has no poles and is a polynomial. Its
	degree is controlled by the only term that grows at infinity,
	\((z\one_q+\boldsymbol a+j\one_q)_1(z+B)^{-1}/(z\one_r+\boldsymbol b+j\one_r)_1\): the numerator has degree
	\(q=r+s\) and the denominator degree \(r+1\), so the growth is
	\((r+s)-(r+1)=s-1\). Thus \(P_j(\,\cdot\,;B)\in\mathbb P_{s-1}\).

	Since \(\{(z+j)_{\ell-1}:\ell=1,\ldots,s\}\) is a basis of
	\(\mathbb P_{s-1}\), there are scalars \(\gamma_{\ell,j}^{H,K}\) with
	\(P_j(z;B)=\sum_{\ell=1}^{s}\gamma_{\ell,j}^{H,K}(z+j)_{\ell-1}\), whence,
	by \eqref{eq:Laguerre-B-tilde-Q-symbol},
	\begin{equation}
		\label{eq:Laguerre-B-Pj-euler}
		\frac{(z\one_q+\boldsymbol a)_j}{(z\one_r+\boldsymbol b)_j}\,P_j(z;B)
		=
		\sum_{\ell=1}^{s}\gamma_{\ell,j}^{H,K}\,
		\widetilde Q_{\ell,j}^{\mathrm L}(z).
	\end{equation}
	The remainder is therefore an Euler symbol and belongs to the
	Euler-derivative block.

	Summing \eqref{eq:Laguerre-B-transition} over \(j=0,\ldots,K-1\), the terms
	\(\mathcal E_{j+1}(z;B)/(z+B)\) telescope. Using \(\mathcal E_0(z;B)=1\) and
	\(B=b_H+K\), the terminal term is
	\[
	\frac{\mathcal E_K(z;B)}{z+B}
	=
	\tau_{H,K}^{H,K}
	\frac{(z\one_q+\boldsymbol a)_K}{(z\one_r+\boldsymbol b)_K\,(z+b_H+K)},
	\]
	which is precisely the \(h=H,\ j=K\) summand of
	\eqref{eq:Laguerre-B-block-reduced-symbol}. Collecting all shifted-beta
	summands into \(\mathscr W_{H,K}^{\mathrm L}\) yields
	\begin{equation}
		\label{eq:Laguerre-B-pole-modulo-euler}
		\frac{1}{z+b_H+K}
		=
		\mathscr W_{H,K}^{\mathrm L}(z)
		+
		\sum_{j=0}^{K-1}\sum_{\ell=1}^{s}
		\gamma_{\ell,j}^{H,K}\,
		\widetilde Q_{\ell,j}^{\mathrm L}(z),
	\end{equation}
	which is \eqref{eq:Laguerre-B-block-modulo-euler}. The Euler symbols used
	here lie within the finite system:
	\eqref{eq:Laguerre-admissibility-n-eta} gives
	\(sK\le s(n_H-1)\le|\boldsymbol\eta|\), and since \(\boldsymbol\eta\) is on the step-line
	this forces \(K\le\min_{\ell}\eta_\ell\); hence
	\(j\le K-1\le\eta_\ell-1\) for every \(\ell\).

	\noindent\emph{Conclusion.} Insert
	\eqref{eq:Laguerre-B-pole-modulo-euler} into the effective decomposition
	\eqref{eq:Laguerre-B-effective-rational-decomposition}. Multiplying by
	\(\widehat\rho_{H,K}^{\mathcal B,\mathrm L}\) and summing over the finite
	nodes gives
	\[
	R_{\boldsymbol\lambda,\boldsymbol m}^{\mathcal B,\mathrm L}(z)
	=
	\sum_{\substack{H=1\\ n_H\ge1}}^{r}\sum_{K=0}^{n_H-1}
	\widehat\rho_{H,K}^{\mathcal B,\mathrm L}\,
	\mathscr W_{H,K}^{\mathrm L}(z)
	+
	\sum_{\ell=1}^{s}\sum_{k=0}^{\eta_\ell-1}
	\widetilde\delta_{\ell,k}\,
	\widetilde Q_{\ell,k}^{\mathrm L}(z),
	\]
	where \(\widetilde\delta_{\ell,k}\) are the coefficients
	\(\delta_{\ell,k}^{\mathcal B,\mathrm L}\) corrected by the Euler
	contributions \eqref{eq:Laguerre-B-pole-modulo-euler}. By
	\eqref{eq:Laguerre-B-block-reduced-symbol} the first sum is exactly the
	reduced symbol of \(\sum_h B^{(h)}w_h\) with
	\(B^{(h)}=\sum_{H,K}\widehat\rho_{H,K}^{\mathcal B,\mathrm L}
	\mathscr R_{h;H,K}^{\mathrm L}\), and the second is the reduced symbol of an
	Euler-derivative block. This exhibits \(R_{\boldsymbol\lambda,\boldsymbol m}^{\mathcal
	B,\mathrm L}\) as the reduced symbol of a mixed form of the type
	\eqref{eq:Laguerre-B-form}. By the isomorphism of Step~5 its component
	vector is unique; hence the shifted-beta components are
	\eqref{eq:Laguerre-B-components-effective-hypergeometric}, and the
	remaining coordinates are the components multiplying the weights
	\(v_\ell\). This proves
	\eqref{eq:Laguerre-B-components-effective-hypergeometric}.
	
	\textit{Step 7: Meijer and hypergeometric representations of the complete \(B\)-form.}

	We derive the Meijer \(G\)-representation of the complete \(B\)-form.
	Using
	\[
	(\beta_i+1-z)_{m_i}
	=
	\frac{\Gamma(\beta_i+m_i+1-z)}
	{\Gamma(\beta_i+1-z)},
	\]
	the Mellin transform \eqref{eq:Laguerre-B-Mellin-proof} can be written as
	\[
	\mathcal M[\mathcal B_{\boldsymbol\lambda,\boldsymbol m}^{\mathrm L}](z)
	=
	\frac{
		\Gamma(z\one_q+\boldsymbol a)
		\Gamma(\boldsymbol\beta+\boldsymbol m+(1-z)\one_p)
	}{
		\Gamma(z\one_r+\boldsymbol b+\boldsymbol n)
		\Gamma(\boldsymbol\beta+(1-z)\one_p)
	}.
	\]
	With the Mellin--Barnes convention fixed in
	\eqref{eq:Meijer-G-definition}--\eqref{eq:Meijer-G-Mellin-transform}, this
	is the Mellin transform of
	\[
	G_{p+r,p+q}^{q,p}
	\left(
	x\,\middle|\,
	\begin{matrix}
		-\boldsymbol\beta-\boldsymbol m,\ \boldsymbol b+\boldsymbol n\\
		\boldsymbol a,\ -\boldsymbol\beta
	\end{matrix}
	\right).
	\]
	This proves \eqref{eq:Laguerre-B-complete-Meijer-G}.
	
	We next give the complete hypergeometric expression of the linear form and
	account for the resonant parameters excluded in part~(ii).
	The Meijer \(G\)-representation \eqref{eq:Laguerre-B-complete-Meijer-G} is
	defined, and equals the complete \(B\)-form, for \emph{every} admissible
	parameter choice considered here: it has a finite analytic continuation in
	the parameters \(\boldsymbol a\) across the coalescence hyperplanes relevant here,
	and the residue hypothesis \(a_\rho-a_\sigma\notin\mathbb Z\) of part~(ii)
	serves only to keep the poles of the Mellin transform simple, so that the
	inversion integral can be summed termwise.

	Apply Mellin inversion to \eqref{eq:Laguerre-B-Mellin-proof}, closing the
	contour to the left and collecting the residues of
	\[
	\frac{\Gamma(z\one_q+\boldsymbol a)\,\Gamma(\boldsymbol\beta+\boldsymbol m+(1-z)\one_p)}
	{\Gamma(z\one_r+\boldsymbol b+\boldsymbol n)\,\Gamma(\boldsymbol\beta+(1-z)\one_p)}\,x^{-z}
	\]
	at the poles \(z=-a_\rho-k\), \(\rho\in\{1,\ldots,q\}\),
	\(k\in\mathbb N_0\), contributed by \(\Gamma(z\one_q+\boldsymbol a)\). When
	\(a_\rho-a_\sigma\notin\mathbb Z\) every such pole is simple, with
	\(\operatorname*{Res}_{z=-a_\rho-k}\Gamma(z+a_\rho)=(-1)^k/k!\), and summing
	the resulting family over \(k\) yields the \({}_{r+p}F_{q+p-1}\) series
	\eqref{eq:Laguerre-B-complete-hypergeometric}. The argument \((-1)^sx\), and
	its sign, come from the excess \(q-r\), namely \(s\), of numerator gamma factors over the
	beta-shifted denominator factors.

	\emph{Coalescent parameters.} If \(a_\sigma-a_\rho=\nu\in\mathbb Z\) for some
	pair, the towers \(\{-a_\rho-k\}_{k}\) and \(\{-a_\sigma-k\}_{k}\) overlap and
	the corresponding poles merge. The right-hand side of
	\eqref{eq:Laguerre-B-complete-hypergeometric} is then to be read as the limit
	\(a_\sigma\to a_\rho+\nu\) of its values at nearby non-resonant parameters;
	this limit exists and is finite, because by the first paragraph the complete
	form is the Meijer \(G\)-function \eqref{eq:Laguerre-B-complete-Meijer-G},
	with a finite analytic continuation in \(\boldsymbol a\) across the relevant
		coalescence hyperplanes. Equivalently, one inverts the Mellin transform
		directly. Write its Mellin--Barnes integrand as
		\[
		\mathcal I(z;x)
		\coloneq
		\frac{\Gamma(z\one_q+\boldsymbol a)\Gamma(\boldsymbol\beta+\boldsymbol m+(1-z)\one_p)}
		{\Gamma(z\one_r+\boldsymbol b+\boldsymbol n)\Gamma(\boldsymbol\beta+(1-z)\one_p)}x^{-z}.
		\]
		A point at which \(j\) of the poles \(-a_\rho-k\) coincide is a
		pole of order at most \(j\), equal to \(j\) unless a reciprocal gamma factor
		cancels part of the singularity, and its contribution is the exact residue
		\begin{equation}
			\label{eq:Laguerre-B-coalescent-residue}
			\operatorname*{Res}_{z=z_0}\mathcal I(z;x)
			=
			\frac{1}{(j-1)!}
			\frac{\mathrm d^{\,j-1}}{\mathrm dz^{\,j-1}}
			\left[
			(z-z_0)^{j}
			\mathcal I(z;x)
			\right]_{z=z_0}.
		\end{equation}
	Since \(\dfrac{\mathrm d}{\mathrm dz}x^{-z}=-x^{-z}\log x\) and the derivatives of the
	surviving gamma factors produce digamma values, \eqref{eq:Laguerre-B-coalescent-residue}
	contributes a term of the form \(x^{a_\rho+k}P_{j-1}(\log x)\), with
	\(P_{j-1}\) a polynomial of degree at most \(j-1\) in \(\log x\). In
	particular, an uncancelled two-fold coincidence \((j=2)\) replaces the two
	coalescing \({}_{r+p}F_{q+p-1}\) series by a logarithmic
	reduction-of-order solution carrying one power of \(\log x\). This determines
	\eqref{eq:Laguerre-B-complete-hypergeometric} for all admissible parameters.
\end{proof}

\begin{remark}[Kamp\'e de F\'eriet reduction and the confluent correction]
	The formula \eqref{eq:Laguerre-B-components-effective-hypergeometric}
	gives the shifted-beta components as finite sums of terminating
	\({}_{r+1}F_q\) polynomials. There is a structural reason why the
	Jacobi Kamp\'e de F\'eriet reduction does not extend to these formulas
	in a comparably compact uniform form. For each Jacobi pole family, expansion
	of the terminating hypergeometric block followed by the substitution
	\(K=u+\lambda\) produces coefficients made from fixed Pochhammer strings in
	\(u\), \(\lambda\), and \(u+\lambda\); this is exactly a classical
	Kamp\'e de F\'eriet polynomial.

	In the Laguerre I case the coefficient of the block indexed by \((H,K)\)
	is instead
	\[
		\widehat\rho_{H,K}^{\mathcal B,\mathrm L}
		=
		\rho_{H,K}^{\mathcal B,\mathrm L}
		-
		\sum_{\ell=1}^{s}
		\sum_{k=K+1}^{\eta_\ell-1}
		\delta_{\ell,k}^{\mathcal B,\mathrm L}
		\sigma_{H,K}^{(\ell,k)}.
	\]
	Thus the substitution \(K=u+\lambda\) leaves an additional confluent index
	\(k>u+\lambda\), whose coefficients
	\(\delta_{\ell,k}^{\mathcal B,\mathrm L}\) are obtained from the inverse
	unitriangular connection system
	\eqref{eq:Laguerre-B-delta-connection}. The result is a triangular nested
	sum, rather than one double hypergeometric series with fixed parameter
	strings. Reordering the triangular correction does not remove the
	connection coefficients and gives no compression of the formula. The finite
	sum in \eqref{eq:Laguerre-B-components-effective-hypergeometric}, together
	with the connection system, is therefore the natural uniform closed form.

	The correction from
	\(\rho_{H,K}^{\mathcal B,\mathrm L}\) to
	\(\widehat\rho_{H,K}^{\mathcal B,\mathrm L}\) is the contribution of the
	polynomial part at infinity of
	\(R_{\boldsymbol\lambda,\boldsymbol m}^{\mathcal B,\mathrm L}\).  If
	\(|\boldsymbol\eta|\le s\), this correction is absent, and the finite-pole
	residues alone give the shifted-beta components; in this shallow case the
	shifted-beta part may be regrouped by the same double-index argument. This
	does not provide a formula for the full Laguerre I component vector,
	because the Euler-derivative components remain to be determined. For deep indices,
	\(|\boldsymbol\eta|>s\), the correction is necessary; it is encoded explicitly by
	\eqref{eq:Laguerre-B-delta-connection}--
	\eqref{eq:Laguerre-B-effective-rho} through the confluent extraction symbol
	\(\widetilde Q_{\ell,k}^{\mathrm L}\).  The remaining components
	\(D^{(\ell)}_{\boldsymbol\lambda,\boldsymbol m}\), which multiply \(v_\ell\), are then
	determined, together with the displayed \(B\)-components, by the finite
	identity
	\eqref{eq:B-component-polynomial-system}.
\end{remark}

\begin{remark}[The two pure boundaries for \(q>1\)]
	\label{rem:Laguerre-two-pure-boundaries}
	The preceding obstruction is specific to the mixed beta--Euler case
	\(r,s>0\). The two pure boundaries have different structures.

	If the formal boundary \(s=0\), \(r=q\), is allowed, the gamma block is
	absent and the weight vector is
	exactly the Jacobi vector. Under the index identification
	\[
		\boldsymbol n^{\,\mathrm J}=\boldsymbol m^{\,\mathrm L},
		\qquad
		\boldsymbol m^{\,\mathrm J}=\boldsymbol n^{\,\mathrm L},
		\qquad
		\boldsymbol\alpha=\boldsymbol\beta,
	\]
	the component formula is therefore the Jacobi formula, and
	Corollary~\ref{cor:Jacobi-B-KdF} gives its terminating Kamp\'e de F\'eriet
	representation.

	If \(r=0\) and \(s=q\), all components belong to the pure Euler block.
	In this case the polynomial system is a block Newton interpolation problem.
	The next proposition gives the resulting terminating hypergeometric
	expression for every component.
\end{remark}

\begin{proposition}[Terminating hypergeometric components at the pure Euler boundary]
	\label{prop:Laguerre-pure-Euler-components}
	Assume \(r=0\), \(s=q>1\), and put
	\[
		P(z)\coloneq(\boldsymbol\beta+(1-z)\one_p)_{\boldsymbol m},
		\qquad
		A_k(z)\coloneq(z\one_q+\boldsymbol a)_k.
	\]
	For \(0\le k\le\eta_\ell-1\), set
	\[
		\boldsymbol K^{k,\ell}
		\coloneq
		\bigl(K_1^{k,\ell},\ldots,K_q^{k,\ell}\bigr),
		\qquad
		K_\sigma^{k,\ell}
		\coloneq k+\begin{cases}1,&\sigma\le\ell,\\0,&\sigma>\ell,\end{cases}
		\qquad
		n_\rho^{k,\ell}\coloneq K_\rho^{k,\ell}-1,
	\]
	and abbreviate
	\[
		\boldsymbol B_\rho
		\coloneq(B_{1,\rho},\ldots,B_{p,\rho}),
		\qquad
		B_{i,\rho}\coloneq\beta_i+a_\rho+1,
		\qquad
		\boldsymbol C_\rho
		\coloneq(C_{1,\rho},\ldots,C_{q,\rho}),
		\qquad
		C_{\sigma,\rho}\coloneq a_\sigma-a_\rho.
	\]
	Suppose first that \(a_\rho-a_\sigma\notin\mathbb Z\) for
	\(\rho\ne\sigma\). Define
	\begin{multline}
	\label{eq:pure-Euler-Newton-hypergeometric-coefficient}
			c_{k,\ell}
			\coloneq
			\sum_{\substack{\rho=1\\K_\rho^{k,\ell}\ge1}}^q
			\frac{
				(\boldsymbol B_\rho)_{\boldsymbol m}
			}{
				n_\rho^{k,\ell}!
				(\boldsymbol C_\rho^{\,*\rho})_{
					(\boldsymbol K^{k,\ell})^{\,*\rho}}
			}
			\\*
			\times
			\pFq{p+q}{p+q-1}
			{
				-n_\rho^{k,\ell},\
				\boldsymbol B_\rho+\boldsymbol m,\
				\one_{q-1}-\boldsymbol C_\rho^{\,*\rho}
				-(\boldsymbol K^{k,\ell})^{\,*\rho}
			}
			{
				\boldsymbol B_\rho,\
				\one_{q-1}-\boldsymbol C_\rho^{\,*\rho}
			}
			{1}.
	\end{multline}
	The series terminates at \(n_\rho^{k,\ell}\). Extend the definition by
	\(c_{k,\ell}=0\) when \(k\ge\eta_\ell\), and let
	\[
		d_{h,k}
		\coloneq
		\sum_{\ell=h}^{q}
		e_{\ell-h}(a_1,\ldots,a_{\ell-1})c_{k,\ell},
		\qquad
		h\in\{1,\ldots,q\},
	\]
	where \(e_j\) denotes the elementary symmetric polynomial and \(e_0=1\).
	Then the pure Euler components are
	\begin{equation}
		\label{eq:pure-Euler-components-hypergeometric}
		D_{\boldsymbol\eta,\boldsymbol m}^{(h)}(x)
		=
		\sum_{k=0}^{\eta_h-1}d_{h,k}x^k,
		\qquad h\in\{1,\ldots,q\}.
	\end{equation}
	If two lattice nodes coalesce, the same formula is understood by continuity,
	equivalently by replacing ordinary divided differences with confluent ones.
\end{proposition}

\begin{proof}
	Order the lattice nodes as
	\[
		-a_1,\ldots,-a_q,\
		-a_1-1,\ldots,-a_q-1,\ldots .
	\]
	Because \(\boldsymbol\eta\) is on the step-line, the pairs
	\((k,\ell)\) with \(0\le k\le\eta_\ell-1\) form the first
	\(N=|\boldsymbol\eta|\) positions in this order.  Let \(c_{k,\ell}\) be the
	divided difference of \(P\) at all nodes up to and including
	\(-a_\ell-k\). Newton interpolation gives
	\begin{equation}
		\label{eq:pure-Euler-block-Newton-expansion}
		P(z)
		=
		\sum_{\ell=1}^{q}
		\sum_{k=0}^{\eta_\ell-1}
		c_{k,\ell}A_k(z)
		\prod_{\sigma=1}^{\ell-1}(z+a_\sigma+k).
	\end{equation}
	The node polynomial used in this divided difference is
	\[
		\Omega_{k,\ell}(z)
		=
		A_k(z)\prod_{\sigma=1}^{\ell}(z+a_\sigma+k)
		=
		\prod_{\sigma=1}^{q}(z+a_\sigma)_{K_\sigma^{k,\ell}}.
	\]
	Its simple roots are \(-a_\rho-j\), with
	\(0\le j\le K_\rho^{k,\ell}-1\), and the barycentric formula yields
	\begin{equation}
		\label{eq:pure-Euler-Newton-barycentric}
		c_{k,\ell}
		=
		\sum_{\substack{\rho=1\\K_\rho^{k,\ell}\ge1}}^q
		\sum_{j=0}^{K_\rho^{k,\ell}-1}
		\frac{
			(\boldsymbol\beta+(a_\rho+j+1)\one_p)_{\boldsymbol m}
		}{
			(-1)^j j!(K_\rho^{k,\ell}-1-j)!
			\displaystyle\prod_{\substack{\sigma=1\\\sigma\ne\rho}}^q
			(a_\sigma-a_\rho-j)_{K_\sigma^{k,\ell}}
		}.
	\end{equation}
	Use
	\[
		(B+j)_m=(B)_m\frac{(B+m)_j}{(B)_j},
		\qquad
		(C-j)_K=(C)_K\frac{(1-C)_j}{(1-C-K)_j},
	\]
	together with
	\[
		\frac{1}{(-1)^j j!(n-j)!}=\frac{(-n)_j}{n!j!}.
	\]
	Substitution in \eqref{eq:pure-Euler-Newton-barycentric} gives
	exactly \eqref{eq:pure-Euler-Newton-hypergeometric-coefficient}.

	Finally, with \(t=z+k\),
	\[
		\prod_{\sigma=1}^{\ell-1}(t+a_\sigma)
		=
		\sum_{h=1}^{\ell}
		e_{\ell-h}(a_1,\ldots,a_{\ell-1})t^{h-1}.
	\]
	Insert this expansion in \eqref{eq:pure-Euler-block-Newton-expansion} and
	collect the powers \((z+k)^{h-1}\). The result is
	\[
		P(z)
		=
		\sum_{h=1}^{q}\sum_{k=0}^{\eta_h-1}
		d_{h,k}(z\one_q+\boldsymbol a)_k(z+k)^{h-1},
	\]
	which is the specialization of
	\eqref{eq:B-component-polynomial-system} to \(r=0\). Uniqueness of that
	expansion proves \eqref{eq:pure-Euler-components-hypergeometric}.
	Multiplication by the common Mellin factor \(\Gamma(z\one_q+\boldsymbol a)\) gives
	\[
		\mathcal M[\mathcal B_{\boldsymbol\eta,\boldsymbol m}^{\mathrm L}](z)
		=
		\Gamma(z\one_q+\boldsymbol a)(\boldsymbol\beta+(1-z)\one_p)_{\boldsymbol m}.
	\]
	At \(z=\beta_i+j+1\), with \(0\le j\le m_i-1\), the factor
	\((\beta_i+1-z)_{m_i}=(-j)_{m_i}\) vanishes. Hence the components in
	\eqref{eq:pure-Euler-components-hypergeometric} satisfy all the required
	orthogonality conditions.
\end{proof}

\begin{remark}[Operator interpretation of the pure Euler formula]
	At \(r=0\), Mellin transformation gives
	\[
		\mathcal B_{\boldsymbol\eta,\boldsymbol m}^{\mathrm L}=P(\thetaop)w_0,
		\qquad
		xw_0=\prod_{\rho=1}^{q}(\thetaop+a_\rho)w_0.
	\]
	Thus \eqref{eq:pure-Euler-components-hypergeometric} is the normal
	remainder of \(P(\thetaop)\) in the ordered basis
	\(x^k\thetaop^{h-1}w_0\).  It is a finite sum with terminating
	\({}_{p+q}F_{p+q-1}(1)\) coefficients, rather than a single classical
	Kamp\'e de F\'eriet polynomial in \(x\).
\end{remark}

\begin{remark}[Beyond the near-diagonal range]
	\label{rem:Laguerre-beyond-near-diagonal}
	We use the near-diagonal range for the Laguerre I mixed system because
	the Mellin cancellations are then automatic and the
	resulting \(\mathcal A\)- and \(\mathcal B\)-forms admit terminating
	hypergeometric representations. Outside this range, the same contour and
	finite reconstruction formulas may still be applied under additional
	finite cancellation conditions. The corresponding formulas are less
	involved, especially for the \(\mathcal B\)-form, where auxiliary
	polynomial reconstructions may involve extra components.
	
	We therefore restrict the Laguerre I results to the near-diagonal range.
	The finite-cancellation extension beyond this range is not used below.
\end{remark}

\begin{remark}[Components versus the complete form]
	The individual polynomials \(B^{(h)}\) and \(D^{(\ell)}\) are obtained
	from the finite identity \eqref{eq:B-component-polynomial-system}. In
	general they are not single hypergeometric polynomials in one component.
	The hypergeometric formula applies to the complete linear form
	\(\mathcal B\), rather than to each component separately.  The pure Euler
	boundary is the explicit exception described in
	Proposition~\ref{prop:Laguerre-pure-Euler-components}: its components are
	finite sums with terminating hypergeometric coefficients.
\end{remark}

\begin{remark}[Ordinary Laguerre-like situation]
	\label{rem:ordinary-Laguerre-like-edge}
	Take \(p=1\). If, for simplicity, \(\beta_1=0\), then the power vector is
	the scalar vector \([1]\), and the mixed matrix
	\eqref{eq:mixed-Laguerre-matrix} reduces to the Laguerre-like vector
	itself. In part~\textnormal{(ii)} of
	Theorem~\ref{thm:mixed-Laguerre-forms}, the condition
	\(|\boldsymbol m|=N-1\) means \(m_1=N-1\), and the normalization becomes
	\begin{equation}
		\label{eq:example-ordinary-B-Mellin}
		\int_0^\infty x^{z-1}\mathcal B_{\boldsymbol\lambda,N-1}^{\mathrm L}(x)\dx
		=
		\frac{\Gamma(z\one_q+\boldsymbol a)}{\Gamma(z\one_r+\boldsymbol b+\boldsymbol n)}(1-z)_{N-1}.
	\end{equation}
	Consequently,
	\[
	\int_0^\infty x^k\mathcal B_{\boldsymbol\lambda,N-1}^{\mathrm L}(x)\dx=0,
	\qquad k\in\{0,\ldots,N-2\},
	\]
	which is the ordinary Laguerre-like type I normalization. The dual
	formula in part~\textnormal{(i)} has a single power component. With
	\(m_1=N+1\), one obtains
	\begin{align}
		\label{eq:example-ordinary-A}
		A_{\boldsymbol\lambda,N+1}^{(1)}(x)
		={}&
		-
		\frac{\Gamma(\boldsymbol b+\boldsymbol n+\one_r)}{N!\,\Gamma(\boldsymbol a+\one_q)}
		\pFq{r+1}{q}
		{
			-N,\ \boldsymbol b+\boldsymbol n+\one_r
		}
		{
			\boldsymbol a+\one_q
		}
		{x}.
	\end{align}
	Thus the mixed theorem collapses to Wolfs's ordinary Laguerre-like pair of
	type I and type II formulae, up to the harmless global normalizations
	fixed here.
\end{remark}

\begin{example}[The smallest confluent mixed form]
	\label{ex:first-confluent-block}
	Let \(q=2\), \(r=1\) and \(s=1\),
	so that
	\[
	\mathbf U^{\mathrm L}=(w_1,v_1),
	\qquad v_1=w_0.
	\]
	Take \(\boldsymbol n=(1)\), \(\boldsymbol\eta=(1)\), hence \(N=2\), and let the
	power vector consist of the single function \(x^\beta\), with
	\(p=1\) and \(\boldsymbol m=(1)\). Since
	\[
	B^{(1)}\in\mathbb P_{n_1-1}=\mathbb P_0,
	\qquad
	D^{(1)}\in\mathbb P_{\eta_1-1}=\mathbb P_0,
	\]
	the mixed form on the Laguerre I vector has the form
	\[
	\mathcal B_{\beta}^{\mathrm L}(x)=b_0w_1(x)+d_0v_1(x).
	\]
	The finite identity \eqref{eq:B-component-polynomial-system} reduces to
	\begin{equation}
		\label{eq:example-confluent-linear-identity}
		b_0+d_0(z+b)=\beta+1-z.
	\end{equation}
	Therefore
	\[
	d_0=-1,
	\qquad
	b_0=\beta+b+1,
	\]
	and
	\begin{equation}
		\label{eq:example-confluent-B}
		\mathcal B_{\beta}^{\mathrm L}(x)
		=(\beta+b+1)w_1(x)-v_1(x).
	\end{equation}
	Its Mellin transform is
	\[
	\int_0^\infty x^{z-1}\mathcal B_{\beta}^{\mathrm L}(x)\dx
	=
	\frac{\Gamma(z+a_1)\Gamma(z+a_2)}{\Gamma(z+b+1)}(\beta+1-z),
	\]
	and hence
	\[
	\int_0^\infty x^{\beta}\mathcal B_{\beta}^{\mathrm L}(x)\dx=0.
	\]
	Thus, already in the first confluent case, the form expanded in the
	Laguerre I vector is a nontrivial linear combination of the
	beta-shifted and Euler-generated weights.
\end{example}

\begin{example}[The exponential-integral specialization]
	\label{ex:exponential-integral-specialization}
	In Example~\ref{ex:first-confluent-block}, choose
	\[
	a_1=0,
	\qquad
	a_2=b,
	\qquad
	b>0.
	\]
	Then
	\[
	\M[w_0](z)=\frac{\Gamma(z)\Gamma(z+b)}{\Gamma(z+b)}=\Gamma(z),
	\]
	so
	\[
	w_0(x)=\e^{-x}.
	\]
	Moreover,
	\[
	\M[w_1](z)=\frac{\Gamma(z)\Gamma(z+b)}{\Gamma(z+b+1)}=\frac{\Gamma(z)}{z+b}.
	\]
	With the standard notation
	\[
	E_{\nu+1}(x)=\int_1^\infty \e^{-xt}t^{-\nu-1}\dt,
	\]
	one has
	\[
	\int_0^\infty x^{z-1}E_{\nu+1}(x)\dx=\frac{\Gamma(z)}{z+\nu},
	\]
	and therefore
	\[
	w_1(x)=E_{b+1}(x),
	\qquad
	v_1(x)=\e^{-x}.
	\]
	Formula \eqref{eq:example-confluent-B} becomes the explicit mixed
	exponential-integral identity
	\begin{equation}
		\label{eq:example-expint-B}
		\mathcal B_{\beta}^{\mathrm L}(x)
		=(\beta+b+1)E_{b+1}(x)-\e^{-x},
	\end{equation}
	with
	\[
	\int_0^\infty x^{\beta}\bigl((\beta+b+1)E_{b+1}(x)-\e^{-x}\bigr)\dx=0.
	\]
	This gives the first-order member of the exponential-integral
	specialization of the Laguerre I construction.
\end{example}

\begin{example}[A rectangular calculation]
	\label{ex:first-rectangular-calculation}
	Keep \(q=2\), \(r=1\), \(s=1\), but now take two powers
	\[
	\mathbf V=(x^{\beta_1},x^{\beta_2}).
	\]
	Choose
	\[
	\boldsymbol n=(2),
	\qquad
	\boldsymbol\eta=(1),
	\qquad
	N=3,
	\qquad
	\boldsymbol m=(1,1).
	\]
	Then the form on the Laguerre I vector has the shape
	\begin{equation}
		\label{eq:example-rectangular-shape}
		\mathcal B^{\mathrm L}(x)=(b_0+b_1x)w_1(x)+d_0v_1(x).
	\end{equation}
	The identity \eqref{eq:B-component-polynomial-system} becomes
	\begin{equation}
		\label{eq:example-rectangular-identity}
		b_0(z+b+1)+b_1(z+a_1)(z+a_2)+d_0(z+b)(z+b+1)
		=(\beta_1+1-z)(\beta_2+1-z).
	\end{equation}
	Assume
	\[
	\Delta\coloneq(a_1-b-1)(a_2-b-1)\ne0.
	\]
	Comparison of the coefficients of \(z^2\), \(z\), and the constant term gives
	\begin{align}
		\label{eq:example-rectangular-coefficients}
		b_1&=\frac{(b+\beta_1+2)(b+\beta_2+2)}{\Delta},
		&
		d_0&=1-b_1,\nonumber\\
		b_0&=
		\frac{(\beta_1+1)(\beta_2+1)-b_1a_1a_2-d_0b(b+1)}{b+1}.
	\end{align}
	Thus \eqref{eq:example-rectangular-shape} is completely explicit. Its
	Mellin transform is
	\[
	\frac{\Gamma(z+a_1)\Gamma(z+a_2)}{\Gamma(z+b+2)}
	(\beta_1+1-z)(\beta_2+1-z),
	\]
	and therefore
	\[
	\int_0^\infty x^{\beta_j}\mathcal B^{\mathrm L}(x)\dx=0,
	\qquad j\in\{1,2\}.
	\]
	This is the first case where the rectangular nature of the construction is
	visible: one combines two Laguerre I functions against two independent
	powers.
\end{example}

\begin{remark}[A formal degeneration to the Pi\~neiro system]
	The target of this formal degeneration is the mixed Pi\~neiro system
	of~\cite{PineiroMixed2026}.
	The Laguerre I setting considered above assumes \(s\ge1\), so that a
	gamma block remains in the Mellin transform and the corresponding weights
	are supported on \((0,\infty)\). If one formally allows \(s=0\), then
	\(q=r\) and the gamma block disappears. If, in addition,
	\[
	\boldsymbol a=\boldsymbol b=\boldsymbol\gamma,
	\]
	then
	\[
	\frac{\Gamma(z\one_q+\boldsymbol\gamma)}
	{\Gamma(z\one_q+\boldsymbol\gamma+\boldsymbol e_j)}
	=
	\frac{1}{z+\gamma_j}
	=
	\int_0^1 x^{z-1}x^{\gamma_j}\dx,
	\qquad j\in\{1,\ldots,q\}.
	\]
	Thus the resulting weights are
	\(x^{\gamma_j}\boldsymbol 1_{(0,1)}(x)\), and the matrix of measures
	becomes, equivalently,
	\[
	\left[x^{\beta_i+\gamma_j}\right]_{
		\substack{j\in\{1,\ldots,q\}\\ i\in\{1,\ldots,p\}}}
	\dx,
	\qquad x\in(0,1),
	\]
	which is the mixed Pi\~neiro matrix studied in~\cite{PineiroMixed2026}.
	Hence this formal specialization
	does not remain in the Laguerre I setting with \(s\ge1\); it degenerates to
	the Jacobi--Pi\~neiro setting.
\end{remark}

\begin{remark}[Confluent and exponential-integral cases]
	For \((q,r,s)=(2,1,1)\), the Laguerre I vector has one shifted beta
	weight and one Mellin-Euler derivative weight,
	\[
	\mathbf U^{\mathrm L}=(w_1,v_1),
	\qquad
	v_1=w_0.
	\]
	The base weight
	\[
	w_0=\mathcal B_{a_1,b_1}*\mathcal G_{a_2}
	\]
	has Mellin transform
	\[
	\M[w_0](z)
	=
	\frac{\Gamma(z+a_1)\Gamma(z+a_2)}{\Gamma(z+b_1)}.
	\]
	Equivalently, up to normalization,
	\[
	w_0(x)
	=
	x^{a_2}\mathrm e^{-x}
	U(b_1-a_1,a_2-a_1+1,x),
	\]
	so this case contains the confluent hypergeometric weights appearing
	in the work of Lima and Loureiro~\cite{LimaLoureiro2020}, after
	matching parameters and normalizations.
	
	The non-mixed exponential-integral weight system of Van Assche and
	Wolfs~\cite{VanAsscheWolfs2023} is obtained as a boundary degeneration
	of the same \((q,r,s)=(2,1,1)\) block. Indeed, taking
	\[
	a_2=\alpha,\qquad a_1=\alpha+\nu,\qquad b_1=a_1
	\]
	gives
	\[
	\M[w_0](z)=\Gamma(z+\alpha),
	\qquad
	\M[w_1](z)
	=
	\frac{\Gamma(z+\alpha)}{z+\alpha+\nu}.
	\]
	These are the Mellin transforms of
	\(x^\alpha\mathrm e^{-x}\) and \(x^\alpha E_{\nu+1}(x)\),
	respectively. Thus the non-mixed Van Assche--Wolfs pair appears as a
	boundary degeneration of the beta block, not as an interior case under
	the strict assumption \(a_1<b_1\).
	
	The mixed exponential-integral system of Van Assche and Wolfs is also
	recovered from the present mixed construction, again as a boundary
	degeneration. Namely, take
	\[
	(q,r,s)=(2,1,1),
	\qquad
	p=2,
	\qquad
	(a_1,a_2,b_1)=(\nu,0,\nu),
	\qquad
	\boldsymbol\beta=(\alpha,\beta).
	\]
	Then
	\[
	v_1(x)=\mathrm e^{-x},
	\qquad
	w_1(x)=E_{\nu+1}(x),
	\]
	up to the normalization of \(E_{\nu+1}\). Hence the mixed matrix
	\[
	\mathrm d\Lagmat(x)
	=
	\begin{bNiceMatrix}[margin=2pt]
		w_1(x)x^\alpha & w_1(x)x^\beta\\
		v_1(x)x^\alpha & v_1(x)x^\beta
	\end{bNiceMatrix}
	\dx
	\]
	is, after interchanging the two rows,
	\[
	\begin{bNiceMatrix}[margin=2pt]
		\mathrm e^{-x}x^\alpha & \mathrm e^{-x}x^\beta\\
		E_{\nu+1}(x)x^\alpha & E_{\nu+1}(x)x^\beta
	\end{bNiceMatrix}
	\dx.
	\]
	This is the mixed exponential-integral matrix built from the weight
	vector \((\mathrm e^{-x},E_{\nu+1})\) and the power vector
	\([x^\alpha,x^\beta]\). In the present construction the Laguerre I vector
	is ordered as
	\((w_1,v_1)=(E_{\nu+1},\mathrm e^{-x})\), whereas the usual
	exponential-integral order is
	\((\mathrm e^{-x},E_{\nu+1})\). Thus the underlying matrix of weights is
	identified after a row permutation. At the level of the mixed
	orthogonality problems, one must also permute the corresponding
	components of the multi-index.

\end{remark}

\begin{remark}[Zhang's Bessel \(I/K\) system and the \(2\times2\) Laguerre I case]
	\label{rem:Zhang-Bessel-Laguerre-like}
	Zhang's mixed system associated with modified Bessel functions~\cite{Zhang2016BesselMixed} is also
	related to the Laguerre I level, but in a different way from the
	exponential-integral system of Van Assche and Wolfs. The relevant
	Laguerre I specialization has
	\[
		r=0,
		\qquad
		s=2,
		\qquad
		q=2,
		\qquad
		p=2.
	\]
	Choose the Laguerre I seed so that
	\[
		\M[w_0](z)=\Gamma(z)\Gamma(z+\nu).
	\]
	Thus, up to normalization and a dilation of the variable,
	\[
		w_0(x)=x^{\nu/2}K_\nu(2b\sqrt{x})\eqqcolon \rho_{\nu,b}(x).
	\]
	The second Laguerre I component is obtained by applying the Euler
	operator. With \(\thetaop=-x\mathrm d/\mathrm dx\), the elementary
	identity for modified Bessel functions gives
	\[
		\thetaop\rho_{\nu,b}
		=
		b\rho_{\nu+1,b}-\nu\rho_{\nu,b}.
	\]
	Consequently
	\[
		\operatorname{span}\{w_0,\thetaop w_0\}
		=
		\operatorname{span}\{\rho_{\nu,b},\rho_{\nu+1,b}\},
	\]
	again up to the harmless normalization and dilation. Hence the
	\(K\)-Bessel vector used by Zhang is the Laguerre I vector of this
	\(q=2\) specialization, written in an equivalent two-dimensional basis.

	The other side of Zhang's mixed system is not the power vector itself but
	its Bessel \({}_0F_1\) analogue. Namely,
	\[
		\omega_{\mu,a}(x)
		\coloneq
		x^{\mu/2}I_\mu(2a\sqrt{x})
		=
		\frac{a^\mu x^\mu}{\Gamma(\mu+1)}
		{}_0F_1\left(
		\begin{matrix}
		-\\
		\mu+1
		\end{matrix}
		\middle|a^2x\right),
	\]
	and similarly for \(\omega_{\mu+1,a}\). Thus
	\((\omega_{\mu,a},\omega_{\mu+1,a})\) is a Bessel-hypergeometric
	deformation of the power vector \((x^\mu,x^{\mu+1})\). After multiplying
	\(\omega_{\mu+i,a}\) by \(\Gamma(\mu+i+1)a^{-\mu-i}\), the limit
	\(a\to0\) is \(x^{\mu+i}\), for \(i\in\{0,1\}\). Thus the two functions
	reduce to the two powers after an explicit diagonal normalization.

	The precise link is at the level of bimoments. For \(0<a<b\), Zhang uses
	\[
		\int_0^\infty
		\omega_{\mu+i,a}(x)\rho_{\nu+j,b}(x)\dx
		=
		\frac{a^{\mu+i}b^{\nu+j}}
		{2(b^2-a^2)^{\mu+\nu+1+i+j}}
		\Gamma(\mu+\nu+1+i+j),
		\qquad i,j\in\N_0.
	\]
	Set
	\[
		c=b^2-a^2,
		\qquad
		\lambda=\mu+\nu+1,
	\]
	and normalize the shifted Bessel families by
	\[
		\Phi_i(x)=c^ia^{-\mu-i}\omega_{\mu+i,a}(x),
		\qquad
		\Psi_j(x)=c^jb^{-\nu-j}\rho_{\nu+j,b}(x).
	\]
	Then
	\[
		\int_0^\infty\Phi_i(x)\Psi_j(x)\dx
		=
		\frac{1}{2c^\lambda}\Gamma(\lambda+i+j)
		=
		\frac{1}{2c^\lambda}
		\int_0^\infty t^it^j t^{\lambda-1}\mathrm e^{-t}\,\mathrm dt.
	\]
	Therefore Zhang's shifted Bessel bases have, up to a constant factor,
	the same bimoment matrix as the gamma moments associated with the
	Laguerre I specialization above.

	This does not mean that the component polynomials are literally the same
	functions of the same variable. Rather, the Laguerre coefficients are
	written in Bessel bases. If
	\[
		\widehat L_n^{(\lambda-1)}(t)
		=
		\sum_{i=0}^n \ell_{n,i}t^i
	\]
	is the corresponding Laguerre polynomial, then the form before reducing
	Bessel orders is
	\[
		\sum_{i=0}^n \ell_{n,i}\Phi_i(x),
	\]
	which, after undoing the normalization, is Zhang's finite expansion in
	the shifted functions \(\omega_{\mu+i,a}\). The final two-component form
	is obtained by reducing the shifted orders to the basis
	\((\omega_{\mu,a},\omega_{\mu+1,a})\). The same argument applies on the
	dual side, where shifted functions \(\rho_{\nu+j,b}\) are reduced to
	\((\rho_{\nu,b},\rho_{\nu+1,b})\). The coefficients of these order
	reductions are the Lommel polynomials, equivalently terminating
	hypergeometric polynomials. Thus Zhang's \(2\times2\) mixed polynomials
	are the Laguerre I moment problem written in the Bessel bases
	\((\omega_{\mu,a},\omega_{\mu+1,a})\) and
	\((\rho_{\nu,b},\rho_{\nu+1,b})\).
\end{remark}
	
\section{Rodrigues formulas}
\label{sec:Rodrigues-hypergeometric-settings}

The Mellin-transform identities obtained above for the mixed forms
expanded in the weight vector admit a direct differential-operator
interpretation. In the ordinary Jacobi-like and Laguerre-like systems
of Wolfs, the Mellin transform of the complete type~I function contains
a single descending Pochhammer factor. In the rectangular mixed
extension, the power vector has \(p\) components, and the Mellin
transform contains one factor
\[
(\gamma_i+1-s)_{n_i}
\]
for each of them. The Rodrigues formula is therefore obtained by
composing one differential operator for each entry of the power vector.

The identity concerns the complete mixed form expanded in the weight
vector. Only when that vector has a single component does it reduce to a
Rodrigues formula for an individual polynomial. In particular, in the
Jacobi setting the specialization \(q=1\) gives a Rodrigues formula
for the Jacobi--Pi\~neiro type~II polynomial, whereas the specialization
\(p=1\), with the corresponding power exponent equal to zero, recovers
the Rodrigues formula for the ordinary Jacobi-like type~I function of
Wolfs.

For \(\gamma\in\mathbb R\) and \(n\in\mathbb N_0\), define
\begin{equation}
	\label{eq:Rodrigues-block-operator}
	\mathscr R_{\gamma,n}[f](x)
	\coloneq
	x^{-\gamma}
	\frac{\mathrm d^n}{\mathrm dx^n}
	\left(
	x^{\gamma+n}f(x)
	\right).
\end{equation}
Set
\[
\vartheta
\coloneq
x\frac{\mathrm d}{\mathrm dx}.
\]
Then
\begin{equation}
	\label{eq:Rodrigues-block-Euler}
	\mathscr R_{\gamma,n}
	=
	(\vartheta+\gamma+1)_n.
\end{equation}
Indeed, both sides act on a monomial \(x^\lambda\) as multiplication
by
\[
(\lambda+\gamma+1)_n.
\]
Consequently, the operators
\(\mathscr R_{\gamma_i,n_i}\) commute. For
\[
\boldsymbol\gamma=(\gamma_1,\ldots,\gamma_p),
\qquad
\boldsymbol n=(n_1,\ldots,n_p)\in\mathbb N_0^p,
\]
we may therefore define
\begin{equation}
	\label{eq:Rodrigues-composed-operator}
	\mathscr R_{\boldsymbol\gamma,\boldsymbol n}
	\coloneq
	\mathscr R_{\gamma_1,n_1}
	\circ\cdots\circ
	\mathscr R_{\gamma_p,n_p},
\end{equation}
without regard to the order of the factors.

We use the following elementary Mellin-transform identity. Here
and below, \(\M\) denotes the Mellin transform on the corresponding
real-line support \(\Delta\), namely \((0,1)\) in the Jacobi
setting and \((0,\infty)\) in the Laguerre I setting.

\begin{proposition}[Mellin action of the Rodrigues operators]
	\label{prop:Mellin-Rodrigues-action}
	Let \(f\) be a function on \(\Delta\) such that the Mellin transforms
	below exist in a common fundamental strip and the endpoint terms produced
	by the integrations by parts vanish in that strip. Then
	\begin{equation}
		\label{eq:Mellin-Rodrigues-block}
		\M[\mathscr R_{\gamma,n}[f]](s)
		=
		(\gamma+1-s)_n\M[f](s).
	\end{equation}
	Consequently, if the same condition holds for the intermediate functions
	appearing in a partial composition of the operators
	\[
	\mathscr R_{\gamma_1,n_1},
	\ldots,
	\mathscr R_{\gamma_p,n_p},
	\]
	then
	\begin{equation}
		\label{eq:Mellin-Rodrigues-composed}
		\M[\mathscr R_{\boldsymbol\gamma,\boldsymbol n}[f]](s)
		=
		(\boldsymbol\gamma+(1-s)\one_p)_{\boldsymbol n}\M[f](s).
	\end{equation}
\end{proposition}

\begin{proof}
	By \eqref{eq:Rodrigues-block-Euler}, it is enough to determine the
	Mellin action of \(\vartheta\). In the fundamental strip under
	consideration, and with no endpoint contribution in the integration by
	parts,
	\begin{align*}
		\M[\vartheta f](s)
		&=
		\int_{\Delta}
		x^{s-1}
		x\frac{\mathrm d}{\mathrm dx}f(x)\dx
		=
		\int_{\Delta}
		x^{s}
		\frac{\mathrm d}{\mathrm dx}f(x)\dx
		\\
		&=
		-s
		\int_{\Delta}
		x^{s-1}f(x)\dx
		=
		-s\M[f](s).
	\end{align*}
	Hence
	\[
		\M[\mathscr R_{\gamma,n}[f]](s)
		=
		\M[(\vartheta+\gamma+1)_n f](s)
		=
		(\gamma+1-s)_n\M[f](s),
	\]
	which proves \eqref{eq:Mellin-Rodrigues-block}. Since the operators
	in \eqref{eq:Rodrigues-composed-operator} commute, successive
	application of this identity gives
	\eqref{eq:Mellin-Rodrigues-composed}.
\end{proof}

\begin{lemma}[Endpoint behavior for the Rodrigues formulas]
	\label{lem:Rodrigues-endpoint-behavior}
	In the Jacobi setting, assume
	\[
	a_j>-1,\qquad b_j>a_j,
	\qquad j\in\{1,\ldots,q\}.
	\]
	Let \(|\boldsymbol m|=|\boldsymbol n|+1\). Then the shifted seed
	\[
	w_0(x;\boldsymbol a,\boldsymbol b+\boldsymbol m)
	\]
	and all functions obtained from it by partial compositions of the
	Rodrigues operators
	\[
	\mathscr R_{\alpha_1,n_1},\ldots,\mathscr R_{\alpha_p,n_p}
	\]
	have no endpoint contribution in the integrations by parts used in
	Proposition~\ref{prop:Mellin-Rodrigues-action}, in their common
	fundamental Mellin strip.
	
	In the Laguerre I setting, assume
	\[
	a_\rho>-1,\qquad \rho\in\{1,\ldots,q\},
	\]
	and, for the beta-shifted block,
	\[
	b_h>a_h,\qquad h\in\{1,\ldots,r\}.
	\]
	Let
	\[
	|\boldsymbol m|=|\boldsymbol n|+|\boldsymbol\eta|-1.
	\]
	Then the shifted seed
	\[
	w_0(x;\boldsymbol a,\boldsymbol b+\boldsymbol n)
	\]
	and all functions obtained from it by partial compositions of the
	Rodrigues operators
	\[
	\mathscr R_{\beta_1,m_1},\ldots,\mathscr R_{\beta_p,m_p}
	\]
	have no endpoint contribution in the integrations by parts used in
	Proposition~\ref{prop:Mellin-Rodrigues-action}, in their common
	fundamental Mellin strip.
\end{lemma}

\begin{proof}
	In the Jacobi case the support is \((0,1)\). Near \(0\), the seed
	and its partial Rodrigues transforms have power-type behavior, so the
	Mellin transform is defined in a nonempty vertical strip. Near \(1\), the
	shifted seed \(w_0(x;\boldsymbol a,\boldsymbol b+\boldsymbol m)\) has a beta-type endpoint
	behavior with total endpoint exponent
	\[
	\sum_{j=1}^{q}(b_j+m_j-a_j)-1
	=
	\sum_{j=1}^{q}(b_j-a_j)+|\boldsymbol m|-1.
	\]
	Since \(|\boldsymbol m|=|\boldsymbol n|+1\), this exponent is
	\[
	\sum_{j=1}^{q}(b_j-a_j)+|\boldsymbol n|.
	\]
	A partial composition of the Rodrigues operators has order at most
	\(|\boldsymbol n|\). Therefore, in every integration by parts occurring before
	or during the composition, the endpoint exponent at \(1\) remains
	positive, because \(b_j>a_j\) for every \(j\). Hence the boundary term at
	\(1\) vanishes. The boundary term at \(0\) vanishes in the common
	fundamental Mellin strip.
	
	In the Laguerre I case the support is \((0,\infty)\). Near \(0\), the
	seed and its partial Rodrigues transforms again have power-type behavior,
	so there is a nonempty Mellin strip. Near \(\infty\), the gamma block in
	the Laguerre I seed gives rapid, generally stretched-exponential decay.
	Applying finitely many
	Rodrigues operators only produces finite linear combinations of
	derivatives multiplied by powers of \(x\), and therefore preserves
	sufficiently rapid decay at infinity. Hence the boundary term at \(\infty\)
	vanishes. The boundary term at \(0\) vanishes in the common fundamental
	Mellin strip.
\end{proof}

\subsection{Jacobi mixed Rodrigues formula}

The Jacobi Mellin identity for the mixed form expanded in the
Jacobi vector contains one factor
\[
(\alpha_i+1-s)_{n_i}
\]
for each component \(x^{\alpha_i}\) of the power vector. By
Proposition~\ref{prop:Mellin-Rodrigues-action}, these factors are
generated by applying the composed operator
\(\mathscr R_{\boldsymbol\alpha,\boldsymbol n}\) to the seed weight whose
parameter vector \(\boldsymbol b\) has been shifted by \(\boldsymbol m\).

In this subsection, we write
\(\mathcal B_{\boldsymbol n,\boldsymbol m}^{\,\mathrm J}\) for the form
\(\mathcal B_{\boldsymbol n,\boldsymbol m}\) introduced in
Theorem~\ref{thm:mixed-Jacobi-like-forms}, in order to distinguish it
from its Laguerre I analogue below.

\begin{theorem}[Rodrigues formula for the Jacobi mixed form]
	\label{thm:Jacobi-mixed-Rodrigues}
	Assume the hypotheses of
	Theorem~\ref{thm:mixed-Jacobi-like-forms}\textnormal{(ii)}, so that
	\[
	|\boldsymbol m|=|\boldsymbol n|+1.
	\]
	Assume also
	\[
	a_j>-1,\qquad b_j>a_j,
	\qquad j\in\{1,\ldots,q\}.
	\]
	Then the normalized mixed form expanded in the Jacobi vector is
	given by
	\begin{equation}
		\label{eq:Jacobi-mixed-Rodrigues}
		\mathcal B_{\boldsymbol n,\boldsymbol m}^{\,\mathrm J}(x)
		=
		-
		\mathscr R_{\boldsymbol\alpha,\boldsymbol n}
		\left[
		w_0(x;\boldsymbol a,\boldsymbol b+\boldsymbol m)
		\right].
	\end{equation}
	Equivalently,
	\begin{equation*}
		\mathcal B_{\boldsymbol n,\boldsymbol m}^{\,\mathrm J}(x)
		=
		-
		\left(
		\mathscr R_{\alpha_1,n_1}
		\circ\cdots\circ
		\mathscr R_{\alpha_p,n_p}
		\right)
		\left[
		w_0(x;\boldsymbol a,\boldsymbol b+\boldsymbol m)
		\right].
	\end{equation*}
\end{theorem}

\begin{proof}
	By \eqref{eq:w0-Mellin}, with \(\boldsymbol b\) replaced by
	\(\boldsymbol b+\boldsymbol m\), we have
	\[
	\M[w_0(\,\cdot\,;\boldsymbol a,\boldsymbol b+\boldsymbol m)](s)
	=
	\frac{\Gamma(s\one_q+\boldsymbol a)}
	{\Gamma(s\one_q+\boldsymbol b+\boldsymbol m)}.
	\]
	By Lemma~\ref{lem:Rodrigues-endpoint-behavior}, the integrations by
	parts required in Proposition~\ref{prop:Mellin-Rodrigues-action} have
	no endpoint contributions for this seed and its partial Rodrigues
	transforms. Therefore, Proposition~\ref{prop:Mellin-Rodrigues-action}
	yields
	\begin{equation*}
		\M\left[
		-
		\mathscr R_{\boldsymbol\alpha,\boldsymbol n}
		\left[
		w_0(\,\cdot\,;\boldsymbol a,\boldsymbol b+\boldsymbol m)
		\right]
		\right](s)
		 =
		-
		\frac{\Gamma(s\one_q+\boldsymbol a)}
		{\Gamma(s\one_q+\boldsymbol b+\boldsymbol m)}
		(\boldsymbol\alpha+(1-s)\one_p)_{\boldsymbol n}.
	\end{equation*}
	The right-hand side is precisely the Mellin transform prescribed in
	\eqref{eq:B-Mellin} for
	\(\mathcal B_{\boldsymbol n,\boldsymbol m}^{\,\mathrm J}\). Since both sides have
	the same Mellin transform in the present class of functions,
	uniqueness of the Mellin transform gives
	\eqref{eq:Jacobi-mixed-Rodrigues}.
\end{proof}

The preceding identity is a Rodrigues formula for the complete mixed
form. Its two boundary specializations recover familiar polynomial or
type~I formulas. When \(q=1\), the Jacobi vector has one
component, so the complete form is one polynomial times one weight;
the corresponding polynomial is of Jacobi--Pi\~neiro type. When
\(p=1\) and the unique power exponent is zero, the composed operator
reduces to a single ordinary derivative operator and the formula
recovers the ordinary Jacobi-like type~I Rodrigues identity of Wolfs.

\begin{corollary}[Jacobi--Pi\~neiro polynomial Rodrigues formula]
	\label{cor:Jacobi-Pineiro-Rodrigues}
	In the specialization \(q=1\), the mixed form expanded in the
	Jacobi vector has a single component:
	\[
	\mathcal B_{\boldsymbol n,m}^{\,\mathrm J}(x)
	=
	B_{\boldsymbol n,m}^{\,\mathrm J,(1)}(x)
	w_1(x;a,b).
	\]
	Consequently, for \(m=|\boldsymbol n|+1\),
	\begin{equation}
		\label{eq:Jacobi-Pineiro-polynomial-Rodrigues}
		B_{\boldsymbol n,m}^{\,\mathrm J,(1)}(x)
		=
		-
		\frac{
			\mathscr R_{\boldsymbol\alpha,\boldsymbol n}
			\left[
			w_0(x;a,b+m)
			\right]
		}{
			w_1(x;a,b)
		}.
	\end{equation}
	In particular, the quotient on the right-hand side belongs to
	\(\mathbb P_{m-1}\).
\end{corollary}

\begin{proof}
	When \(q=1\), the mixed form
	\(\mathcal B_{\boldsymbol n,m}^{\,\mathrm J}\) has exactly one polynomial
	component in the Jacobi vector. Substituting its expression into
	\eqref{eq:Jacobi-mixed-Rodrigues} and dividing by the nonzero weight
	\(w_1(x;a,b)\) gives
	\eqref{eq:Jacobi-Pineiro-polynomial-Rodrigues}. The asserted
	belonging to \(\mathbb P_{m-1}\) is the polynomial-space condition
	prescribed by the mixed orthogonality problem.
\end{proof}

\begin{remark}[Ordinary Jacobi-like specialization]
	When \(p=1\) and \(\alpha_1=0\), formula
	\eqref{eq:Jacobi-mixed-Rodrigues} becomes, up to the normalization
	used here,
	\[
	\mathcal B_{(n),\boldsymbol m}^{\,\mathrm J}(x)
	=
	-
	\frac{\mathrm d^n}{\mathrm dx^n}
	\left[
	x^n w_0(x;\boldsymbol a,\boldsymbol b+\boldsymbol m)
	\right],
	\qquad
	n=|\boldsymbol m|-1.
	\]
	This is the Rodrigues formula for the ordinary Jacobi-like type~I
	function obtained by Wolfs~\cite{Wolfs2024}.
\end{remark}

\subsection{Laguerre I mixed Rodrigues formula}

The Laguerre I case is obtained in the same way, with the roles of
the two multi-indices adapted to the mixed organization. The power
vector is indexed by \(\boldsymbol m\), while the shift of the seed weight is
determined by the beta-shifted part \(\boldsymbol n\) of
\(\boldsymbol\lambda=(\boldsymbol n,\boldsymbol\eta)\). Consequently, the factors appearing
in the Mellin transform are reproduced by
\(\mathscr R_{\boldsymbol\beta,\boldsymbol m}\).

\begin{theorem}[Rodrigues formula for the Laguerre I mixed form]
	\label{thm:Laguerre-mixed-Rodrigues}
	Assume the hypotheses of
	Theorem~\ref{thm:mixed-Laguerre-forms}\textnormal{(ii)}. Let
	\[
	\boldsymbol\lambda=(\boldsymbol n,\boldsymbol\eta),
	\qquad
	|\boldsymbol m|=|\boldsymbol n|+|\boldsymbol\eta|-1.
	\]
	Assume also
	\[
	a_\rho>-1,\qquad \rho\in\{1,\ldots,q\},
	\]
	and
	\[
	b_h>a_h,\qquad h\in\{1,\ldots,r\}.
	\]
	Then the normalized form on the Laguerre I vector is given by
	\begin{equation}
		\label{eq:Laguerre-mixed-Rodrigues}
		\mathcal B_{\boldsymbol\lambda,\boldsymbol m}^{\mathrm L}(x)
		=
		\mathscr R_{\boldsymbol\beta,\boldsymbol m}
		\left[
		w_0(x;\boldsymbol a,\boldsymbol b+\boldsymbol n)
		\right]
		=
		\left(
		\mathscr R_{\beta_1,m_1}
		\circ\cdots\circ
		\mathscr R_{\beta_p,m_p}
		\right)
		\left[
		w_0(x;\boldsymbol a,\boldsymbol b+\boldsymbol n)
		\right].
	\end{equation}
\end{theorem}

\begin{proof}
	From \eqref{eq:w0-Laguerre-Mellin}, with \(\boldsymbol b\) replaced by
	\(\boldsymbol b+\boldsymbol n\), one has
	\[
	\M[w_0(\,\cdot\,;\boldsymbol a,\boldsymbol b+\boldsymbol n)](s)
	=
	\frac{\Gamma(s\one_q+\boldsymbol a)}
	{\Gamma(s\one_r+\boldsymbol b+\boldsymbol n)}.
	\]
	By Lemma~\ref{lem:Rodrigues-endpoint-behavior}, the integrations by
	parts required in Proposition~\ref{prop:Mellin-Rodrigues-action} have
	no endpoint contributions for this seed and its partial Rodrigues
	transforms. Therefore Proposition~\ref{prop:Mellin-Rodrigues-action}
	gives
	\[
	\M\left[
	\mathscr R_{\boldsymbol\beta,\boldsymbol m}
	\left[
	w_0(\,\cdot\,;\boldsymbol a,\boldsymbol b+\boldsymbol n)
	\right]
	\right](s)
	=
	\frac{\Gamma(s\one_q+\boldsymbol a)}
	{\Gamma(s\one_r+\boldsymbol b+\boldsymbol n)}
	(\boldsymbol\beta+(1-s)\one_p)_{\boldsymbol m}.
	\]
	This is exactly the Mellin transform in
	\eqref{eq:Laguerre-B-Mellin}. Mellin inversion proves
	\eqref{eq:Laguerre-mixed-Rodrigues}.
\end{proof}

The ordinary Laguerre-like family is recovered when the power vector
contains a single component. In this case the composed operator reduces
to one ordinary derivative and the formula agrees with the type~I
Rodrigues identity obtained by Wolfs.

\begin{remark}[Ordinary Laguerre-like specialization]
	For \(p=1\) and \(\beta_1=0\), let
	\[
	N=|\boldsymbol n|+|\boldsymbol\eta|,
	\qquad
	m_1=N-1.
	\]
	Then \eqref{eq:Laguerre-mixed-Rodrigues} reduces to
	\[
	\mathcal B_{\boldsymbol\lambda,N-1}^{\mathrm L}(x)
	=
	\frac{\mathrm d^{N-1}}{\mathrm dx^{N-1}}
	\left[
	x^{N-1}
	w_0(x;\boldsymbol a,\boldsymbol b+\boldsymbol n)
	\right],
	\]
	which is the Rodrigues-type formula for the ordinary Laguerre-like
	type~I function obtained by Wolfs~\cite{Wolfs2024}.
\end{remark}

\section{The hypergeometric settings in the step-line}
\label{sec:step-line-specializations}

The general step-line construction of
Section~\ref{sec:general-mixed-framework} can now be specialized to the
two systems obtained above. In this section, we write the resulting
step-line forms explicitly, including their polynomial components and
normalization constants.  The recurrence coefficients are then evaluated
directly from the step-line pairing
\(\langle x\mathcal B_N,\mathcal A_{N+k}\rangle\).  This avoids introducing
any recurrence between off-step-line neighbouring multi-indices.

For \(d\in\mathbb N\), let
\[
\boldsymbol\sigma_d(N)\in\mathbb N_0^d,
\qquad
N\in\mathbb N_0,
\]
denote the step-line multi-index introduced in
Section~\ref{sec:general-mixed-framework}. We also set
\[
\jmath_N\coloneq \ell_q(N)+1,
\qquad
N\in\mathbb N_0,
\]
so that
\[
\boldsymbol\sigma_q(N+1)-\boldsymbol\sigma_q(N)
=
\boldsymbol e_{\jmath_N}^{(q)}.
\]

\subsection{Jacobi step-line forms}
\label{subsec:Jacobi-step-line}

In the Jacobi setting, the first multi-index refers to the
\(p\)-component power vector, whereas the second refers to the
\(q\)-component Jacobi vector. Define
\begin{equation}
	\label{eq:SL-J-indices}
	\begin{aligned}
		\boldsymbol\nu_N&\coloneq \boldsymbol\sigma_p(N),
		&
		\boldsymbol\mu_N&\coloneq \boldsymbol\sigma_q(N+1),
		\\
		\boldsymbol\nu_N^{\,*}&\coloneq \boldsymbol\sigma_p(N+1),
		&
		\boldsymbol\mu_N^{\,*}&\coloneq \boldsymbol\sigma_q(N),
	\end{aligned}
	\qquad
	N\in\mathbb N_0.
\end{equation}
Then
\[
|\boldsymbol\mu_N|=|\boldsymbol\nu_N|+1,
\qquad
|\boldsymbol\nu_N^{\,*}|=|\boldsymbol\mu_N^{\,*}|+1,
\]
and
\begin{equation}
	\label{eq:SL-incremented-coordinate}
	\boldsymbol\mu_N-\boldsymbol\mu_N^{\,*}
	=
	\boldsymbol e_{\jmath_N}^{(q)}.
\end{equation}
Thus, the unnormalized Jacobi step-line forms are obtained from
the index pairs
\[
\widehat{\mathcal A}_N^{\,\mathrm J}
\coloneq
\mathcal A_{\boldsymbol\nu_N^{\,*},\boldsymbol\mu_N^{\,*}}^{\,\mathrm J},
\qquad
\widehat{\mathcal B}_N^{\,\mathrm J}
\coloneq
\mathcal B_{\boldsymbol\nu_N,\boldsymbol\mu_N}^{\,\mathrm J}.
\]

\subsubsection{The form on the power vector}

The normalized form on the power vector is
\begin{equation}
	\label{eq:Jacobi-SL-A-form}
	\mathcal A_N^{\,\mathrm J}(x)
	=
	\frac{1}{g_N^{\,\mathrm J}}
	\widehat{\mathcal A}_N^{\,\mathrm J}(x)
	=
	\frac{1}{g_N^{\,\mathrm J}}
	\sum_{i=1}^{p}
	\mathscr A_N^{\,\mathrm J,(i)}(x)x^{\alpha_i}.
\end{equation}
For \(i\in\{1,\ldots,p\}\) such that
\((\boldsymbol\nu_N^{\,*})_i\ge1\),
define
\begin{multline}
	\label{eq:Jacobi-SL-A-components}
		\mathscr A_N^{\,\mathrm J,(i)}(x)
		=
		C_{N,i}^{\,\mathrm J,\mathcal A}
		\\*
		\times
		\pFq{p+q}{p+q-1}
		{
			-(\boldsymbol\nu_N^{\,*})_i+1,\
			(\alpha_i+1)\one_q+\boldsymbol b+\boldsymbol\mu_N^{\,*},\
			(\alpha_i+1)\one_{p-1}-\boldsymbol\alpha^{\,*i}
			-(\boldsymbol\nu_N^{\,*})^{\,*i}
		}
		{
			(\alpha_i+1)\one_q+\boldsymbol a,\
			(\alpha_i+1)\one_{p-1}-\boldsymbol\alpha^{\,*i}
		}
		{x},
\end{multline}
where
\begin{equation}
	\label{eq:Jacobi-SL-A-prefactor}
	C_{N,i}^{\,\mathrm J,\mathcal A}
	\coloneq
	(-1)^{|(\boldsymbol\nu_N^{\,*})^{\,*i}|+1}
	\frac{
		\Gamma((\alpha_i+1)\one_q+\boldsymbol b)
	}{
		\Gamma((\alpha_i+1)\one_q+\boldsymbol a)
	}
	\frac{
		((\alpha_i+1)\one_q+\boldsymbol b)_{\boldsymbol\mu_N^{\,*}}
	}{
		((\boldsymbol\nu_N^{\,*})_i-1)!
		((\alpha_i+1)\one_{p-1}-\boldsymbol\alpha^{\,*i}
		-(\boldsymbol\nu_N^{\,*})^{\,*i})_{(\boldsymbol\nu_N^{\,*})^{\,*i}}
	}.
\end{equation}
If
$(\boldsymbol\nu_N^{\,*})_i=0$, we set
$\mathscr A_N^{\,\mathrm J,(i)}\equiv0$.

\subsubsection{The form on the Jacobi vector}

The normalized form on the Jacobi vector is
\begin{equation}
	\label{eq:Jacobi-SL-B-form}
	\mathcal B_N^{\,\mathrm J}(x)
	=
	\frac{1}{h_N^{\,\mathrm J}}
	\widehat{\mathcal B}_N^{\,\mathrm J}(x)
	=
	\frac{1}{h_N^{\,\mathrm J}}
	\sum_{j=1}^{q}
	\mathscr B_N^{\,\mathrm J,(j)}(x)
	w_j(x;\boldsymbol a,\boldsymbol b).
\end{equation}
For
\(J\in\{1,\ldots,q\}\) such that \((\boldsymbol\mu_N)_J\ge1\),
and
\(K\in\{0,\ldots,(\boldsymbol\mu_N)_J-1\}\),
define
\begin{equation}
	\label{eq:Jacobi-SL-B-seed-coefficients}
	\Pi_{N;J,K}^{\,\mathrm J}
	\coloneq
	-
	\frac{
		(\boldsymbol\alpha+(b_J+K+1)\one_p)_{\boldsymbol\nu_N}
	}{
		(-1)^K K!((\boldsymbol\mu_N)_J-K-1)!
		(\boldsymbol b^{\,*J}-(b_J+K)\one_{q-1})_{\boldsymbol\mu_N^{\,*J}}
	}.
\end{equation}
For \(j\in\{1,\ldots,q\}\) and
\(K-1+\delta_{j,J}\ge0\),
set
\begin{equation}
	\label{eq:Jacobi-SL-Psi}
	\Psi_{N;j;J,K}^{\,\mathrm J}(x)
	\coloneq
	\frac{
		(\boldsymbol b^{\,*j}-(b_J+K)\one_{q-1})_1
	}{
		(\boldsymbol a-(b_J+K)\one_q)_1
	}
	\pFq{q+1}{q}
	{
		1,\ \boldsymbol b+(1-b_J-K)\one_q-\boldsymbol e_j
	}
	{
		\boldsymbol a+(1-b_J-K)\one_q
	}
	{x},
\end{equation}
and put
\(\Psi_{N;j;J,0}^{\,\mathrm J}\coloneq0\), \(j\neq J\).
Indeed, the \(J\)-th upper parameter in
\eqref{eq:Jacobi-SL-Psi} is \(-K\) when \(j=J\), and \(1-K\)
when \(j\neq J\). Consequently, every nonzero
\(\Psi_{N;j;J,K}^{\,\mathrm J}\) is a terminating polynomial of degree
\(K-1+\delta_{j,J}\).

For \(j\in\{1,\ldots,q\}\) such that
\((\boldsymbol\mu_N)_j\ge1\),
the step-line component polynomial is
\begin{equation}
	\label{eq:Jacobi-SL-B-components}
	\mathscr B_N^{\,\mathrm J,(j)}(x)
	=
	\frac{
		(\boldsymbol a-b_j\one_q)_1
	}{
		(\boldsymbol b^{\,*j}-b_j\one_{q-1})_1
	}
	\sum_{\substack{J=1\\(\boldsymbol\mu_N)_J\ge1}}^{q}
	\sum_{K=0}^{(\boldsymbol\mu_N)_J-1}
	\Pi_{N;J,K}^{\,\mathrm J}
	\Psi_{N;j;J,K}^{\,\mathrm J}(x).
\end{equation}
If \((\boldsymbol\mu_N)_j=0\), we set
\[
\mathscr B_N^{\,\mathrm J,(j)}\equiv0.
\]

\subsubsection{Normalization and Mellin identity}

The normalization constants fixing the biorthonormal Jacobi
step-line system are
\begin{align}
	\label{eq:SL-J-normalization-constants}
	h_N^{\,\mathrm J}
	&\coloneq
	[x^{(\boldsymbol\mu_N)_{\jmath_N}-1}]
	\mathscr B_N^{\,\mathrm J,(\jmath_N)}(x),
	&
	g_N^{\,\mathrm J}
	&\coloneq
	\int_0^1
	x^{(\boldsymbol\mu_N^{\,*})_{\jmath_N}}
	\widehat{\mathcal A}_N^{\,\mathrm J}(x)
	w_{\jmath_N}(x;\boldsymbol a,\boldsymbol b)\dx.
\end{align}
Here \([x^d]P(x)\) denotes the coefficient of \(x^d\) in \(P\).
The terminating expression \eqref{eq:Jacobi-SL-B-components} evaluates
\(h_N^{\,\mathrm J}\), while expansion of
\eqref{eq:Jacobi-SL-A-components} followed by
\eqref{eq:wj-Mellin} evaluates \(g_N^{\,\mathrm J}\).
Thus both constants are finite Gamma--Pochhammer expressions. The Mellin
identity of the normalized form on the Jacobi vector is
\begin{equation}
	\label{eq:Jacobi-SL-B-Mellin}
	\int_0^1
	x^{s-1}\mathcal B_N^{\,\mathrm J}(x)\dx
	=
	-
	\frac{1}{h_N^{\,\mathrm J}}
	\frac{
		\Gamma(s\one_q+\boldsymbol a)
	}{
		\Gamma(s\one_q+\boldsymbol b+\boldsymbol\mu_N)
	}
	(\boldsymbol\alpha+(1-s)\one_p)_{\boldsymbol\nu_N}.
\end{equation}

Under the boundary degeneration
\[
\boldsymbol a,\boldsymbol b\longrightarrow\boldsymbol\beta,
\]
these formulas specialize to the mixed Pi\~neiro step-line forms.
These are precisely the step-line forms developed in~\cite{PineiroMixed2026}.

\subsection{Laguerre I step-line forms}
\label{subsec:Laguerre-step-line}

In the Laguerre I setting, the notation of the explicit forms places
first the multi-index of the \(q\)-component Laguerre I vector and
second the multi-index of the \(p\)-component power vector. Define
\begin{equation}
	\label{eq:Laguerre-SL-indices}
	\boldsymbol\lambda_N^{\mathrm L}
	\coloneq
	\boldsymbol\sigma_q(N)
	=
	(\boldsymbol n_N^{\mathrm L},\boldsymbol\eta_N^{\mathrm L}),
	\qquad
	\boldsymbol\mu_N^{\mathrm L}
	\coloneq
	\boldsymbol\sigma_p(N),
	\qquad
	N\in\mathbb N_0,
\end{equation}
where
\[
\boldsymbol n_N^{\mathrm L}\in\mathbb N_0^r,
\qquad
\boldsymbol\eta_N^{\mathrm L}\in\mathbb N_0^s.
\]
Then
\[
|\boldsymbol\mu_{N+1}^{\mathrm L}|
=
|\boldsymbol\lambda_N^{\mathrm L}|+1,
\qquad
|\boldsymbol\lambda_{N+1}^{\mathrm L}|
=
|\boldsymbol\mu_N^{\mathrm L}|+1,
\]
and
\begin{equation}
	\label{eq:Laguerre-SL-incremented-coordinate}
	\boldsymbol\lambda_{N+1}^{\mathrm L}
	-
	\boldsymbol\lambda_N^{\mathrm L}
	=
	\boldsymbol e_{\jmath_N}^{(q)}.
\end{equation}
Accordingly, the unnormalized Laguerre I step-line forms are
\[
\widehat{\mathcal A}_N^{\mathrm L}
\coloneq
\mathcal A_{\boldsymbol\lambda_N^{\mathrm L},\,
	\boldsymbol\mu_{N+1}^{\mathrm L}}^{\mathrm L},
\qquad
\widehat{\mathcal B}_N^{\mathrm L}
\coloneq
\mathcal B_{\boldsymbol\lambda_{N+1}^{\mathrm L},\,
	\boldsymbol\mu_N^{\mathrm L}}^{\mathrm L}.
\]

\subsubsection{The form on the power vector}

The normalized form on the power vector is
\begin{equation}
	\label{eq:Laguerre-SL-A-form}
	\mathcal A_N^{\mathrm L}(x)
	=
	\frac{1}{g_N^{\mathrm L}}
	\widehat{\mathcal A}_N^{\mathrm L}(x)
	=
	\frac{1}{g_N^{\mathrm L}}
	\sum_{i=1}^{p}
	\mathscr A_N^{\mathrm L,(i)}(x)x^{\beta_i}.
\end{equation}
For \(i\in\{1,\ldots,p\}\) such that
\((\boldsymbol\mu_{N+1}^{\mathrm L})_i\ge1\),
define
\begin{multline}
	\label{eq:Laguerre-SL-A-components}
		\mathscr A_N^{\mathrm L,(i)}(x)
		=
		C_{N,i}^{\mathrm L,\mathcal A}
		\\*
		\times
		\pFq{r+p}{q+p-1}
		{
			-(\boldsymbol\mu_{N+1}^{\mathrm L})_i+1,\
			(\beta_i+1)\one_r+\boldsymbol b+\boldsymbol n_N^{\mathrm L},\
			(\beta_i+1)\one_{p-1}-\boldsymbol\beta^{\,*i}
			-(\boldsymbol\mu_{N+1}^{\mathrm L})^{\,*i}
		}
		{
			(\beta_i+1)\one_q+\boldsymbol a,\
			(\beta_i+1)\one_{p-1}-\boldsymbol\beta^{\,*i}
		}
		{x},
\end{multline}
where
\begin{equation}
	\label{eq:Laguerre-SL-A-prefactor}
	C_{N,i}^{\mathrm L,\mathcal A}
	\coloneq
	-
	\frac{
		\Gamma((\beta_i+1)\one_r+\boldsymbol b+\boldsymbol n_N^{\mathrm L})
	}{
		\Gamma((\beta_i+1)\one_q+\boldsymbol a)
	}
	\frac{
		1
	}{
		((\boldsymbol\mu_{N+1}^{\mathrm L})_i-1)!
		(\boldsymbol\beta^{\,*i}-\beta_i\one_{p-1})_{(\boldsymbol\mu_{N+1}^{\mathrm L})^{\,*i}}
	}.
\end{equation}
If
\((\boldsymbol\mu_{N+1}^{\mathrm L})_i=0\),
we set
\(\mathscr A_N^{\mathrm L,(i)}\equiv0\).

\subsubsection{The form on the Laguerre I vector}

The normalized form on the Laguerre I vector is
\begin{equation}
	\label{eq:Laguerre-SL-B-form}
	\mathcal B_N^{\mathrm L}(x)
	=
	\frac{1}{h_N^{\mathrm L}}
	\widehat{\mathcal B}_N^{\mathrm L}(x)
	=
	\frac{1}{h_N^{\mathrm L}}
	\left(
	\sum_{h=1}^{r}
	\mathscr B_N^{\mathrm L,(h)}(x)w_h(x;\boldsymbol a,\boldsymbol b)
	+
	\sum_{\ell=1}^{s}
	\mathscr D_N^{\mathrm L,(\ell)}(x)v_\ell(x;\boldsymbol a,\boldsymbol b)
	\right).
\end{equation}
For \(h\in\{1,\ldots,r\}\) such that
\((\boldsymbol n_{N+1}^{\mathrm L})_h\ge1\),
write
\[
\mathscr B_N^{\mathrm L,(h)}(x)
=
\sum_{k=0}^{(\boldsymbol n_{N+1}^{\mathrm L})_h-1}
b_{N;h,k}^{\mathrm L}x^k,
\]
whereas, if \((\boldsymbol n_{N+1}^{\mathrm L})_h=0\), we set
\(\mathscr B_N^{\mathrm L,(h)}\equiv0\).
For \(\ell\in\{1,\ldots,s\}\) such that
\((\boldsymbol\eta_{N+1}^{\mathrm L})_\ell\ge1\),
write
\[
\mathscr D_N^{\mathrm L,(\ell)}(x)
=
\sum_{k=0}^{(\boldsymbol\eta_{N+1}^{\mathrm L})_\ell-1}
d_{N;\ell,k}^{\mathrm L}x^k,
\]
whereas, if \((\boldsymbol\eta_{N+1}^{\mathrm L})_\ell=0\), we set
\(\mathscr D_N^{\mathrm L,(\ell)}\equiv0\).
Introduce, for the step-line data, the compact basis polynomials
	\begin{align}
		\label{eq:Laguerre-SL-PD-basis-polynomials}
		P_{N;h,k}^{\mathrm L}(z)
		&\coloneq
		(z\one_q+\boldsymbol a)_k
		(z\one_r+\boldsymbol b+k\one_r+\boldsymbol e_h)_{\boldsymbol n_{N+1}^{\mathrm L}-k\one_r-\boldsymbol e_h},
		\\
		D_{N;\ell,k}^{\mathrm L}(z)
		&\coloneq
		(z\one_q+\boldsymbol a)_k
		(z\one_r+\boldsymbol b+k\one_r)_{\boldsymbol n_{N+1}^{\mathrm L}-k\one_r}
		(z+k)^{\ell-1}.
	\end{align}
The polynomial components in \eqref{eq:Laguerre-SL-B-form} are uniquely
determined by the finite confluent identity
\begin{multline}
	\label{eq:Laguerre-SL-B-component-system}
	\sum_{\substack{h=1\\(\boldsymbol n_{N+1}^{\mathrm L})_h\ge1}}^{r}
	\sum_{k=0}^{(\boldsymbol n_{N+1}^{\mathrm L})_h-1}
	b_{N;h,k}^{\mathrm L}P_{N;h,k}^{\mathrm L}(z)
	+
	\sum_{\substack{\ell=1\\(\boldsymbol\eta_{N+1}^{\mathrm L})_\ell\ge1}}^{s}
	\sum_{k=0}^{(\boldsymbol\eta_{N+1}^{\mathrm L})_\ell-1}
	d_{N;\ell,k}^{\mathrm L}D_{N;\ell,k}^{\mathrm L}(z)
	=
	(\boldsymbol\beta+(1-z)\one_p)_{\boldsymbol\mu_N^{\mathrm L}}.
\end{multline}

If
\(a_\rho-a_\sigma\notin\mathbb Z\), 
\(\rho\neq\sigma\),
then the complete unnormalized form
\(\widehat{\mathcal B}_N^{\mathrm L}\) admits the residue expansion
\begin{multline}
	\label{eq:Laguerre-SL-B-complete-hypergeometric}
		\widehat{\mathcal B}_N^{\mathrm L}(x)
		=
		\sum_{\rho=1}^{q}
		\frac{
			\Gamma(\boldsymbol a^{\,*\rho}-a_\rho\one_{q-1})
			(\boldsymbol\beta+(a_\rho+1)\one_p)_{\boldsymbol\mu_N^{\mathrm L}}
		}{
			\Gamma(\boldsymbol b+\boldsymbol n_{N+1}^{\mathrm L}
			-a_\rho\one_r)
		}
		x^{a_\rho}
		\\*
		\times
		\pFq{r+p}{q+p-1}
		{
			(a_\rho+1)\one_r-\boldsymbol b-\boldsymbol n_{N+1}^{\mathrm L},\
			\boldsymbol\beta+\boldsymbol\mu_N^{\mathrm L}
			+(a_\rho+1)\one_p
		}
		{
			(a_\rho+1)\one_{q-1}-\boldsymbol a^{\,*\rho},\
			\boldsymbol\beta+(a_\rho+1)\one_p
		}
		{(-1)^s x}.
\end{multline}
Consequently,
\[
h_N^{\mathrm L}\mathcal B_N^{\mathrm L}(x)
=
\widehat{\mathcal B}_N^{\mathrm L}(x)
\]
is given by the right-hand side of
\eqref{eq:Laguerre-SL-B-complete-hypergeometric}.

\subsubsection{Normalization and Mellin identity}

For \(\rho\in\{1,\ldots,q\}\), set
\[
\mathscr U_N^{\mathrm L,(\rho)}
\coloneq
\begin{cases}
	\mathscr B_N^{\mathrm L,(\rho)},
	&
	\rho\in\{1,\ldots,r\},
	\\
	\mathscr D_N^{\mathrm L,(\rho-r)},
	&
	\rho\in\{r+1,\ldots,q\}.
\end{cases}
\]
The normalization constants fixing the biorthonormal Laguerre I
step-line system are
\begin{align}
	\label{eq:SL-L-normalization-constants}
	h_N^{\mathrm L}
	&\coloneq
	[x^{(\boldsymbol\lambda_{N+1}^{\mathrm L})_{\jmath_N}-1}]
	\mathscr U_N^{\mathrm L,(\jmath_N)}(x),
	&
	g_N^{\mathrm L}
	&\coloneq
	\int_0^\infty
	x^{(\boldsymbol\lambda_N^{\mathrm L})_{\jmath_N}}
	\widehat{\mathcal A}_N^{\mathrm L}(x)
	U_{\jmath_N}^{\mathrm L}(x)\dx.
\end{align}
The finite component system
\eqref{eq:Laguerre-SL-B-component-system} evaluates \(h_N^{\mathrm L}\);
expanding \eqref{eq:Laguerre-SL-A-components} and applying the Mellin
transform of \(U_{\jmath_N}^{\mathrm L}\) evaluates \(g_N^{\mathrm L}\).
The Mellin identity for the normalized form on the Laguerre I vector is
\begin{equation}
	\label{eq:Laguerre-SL-B-Mellin}
	\int_0^\infty
	x^{z-1}\mathcal B_N^{\mathrm L}(x)\dx
	=
	\frac{1}{h_N^{\mathrm L}}
	\frac{
		\Gamma(z\one_q+\boldsymbol a)
	}{
		\Gamma(z\one_r+\boldsymbol b+\boldsymbol n_{N+1}^{\mathrm L})
	}
	(\boldsymbol\beta+(1-z)\one_p)_{\boldsymbol\mu_N^{\mathrm L}}.
\end{equation}

\begin{proposition}[Biorthonormal step-line systems]
	\label{prop:hypergeometric-step-line-biorthogonality}
	Assume that the relevant step-line index pairs are normal and that
	the normalization constants
	\[
	h_N^{\,\mathrm J},
	\qquad
	g_N^{\,\mathrm J},
	\qquad
	h_N^{\mathrm L},
	\qquad
	g_N^{\mathrm L}
	\]
	do not vanish. Then
	\begin{equation}
		\label{eq:SL-biorthogonality}
		\left\langle
		\mathcal B_N^{\,\mathrm J},
		\mathcal A_M^{\,\mathrm J}
		\right\rangle
		=
		\left\langle
		\mathcal B_N^{\mathrm L},
		\mathcal A_M^{\mathrm L}
		\right\rangle
		=
		\delta_{N,M},
		\qquad
		N,M\in\mathbb N_0.
	\end{equation}
\end{proposition}

\begin{proof}
	The displayed Jacobi and Laguerre I forms are precisely the
	specializations of the general step-line forms associated with the
	index pairs defined above.  For
	\(\star\in\{\mathrm J,\mathrm L\}\), their orthogonality conditions give
	zero when \(N\ne M\), and the constants \(h_N^\star\) and
	\(g_N^\star\) impose
	\(\langle\mathcal B_N^\star,\mathcal A_N^\star\rangle=1\).
	Hence the general step-line biorthogonality relation gives
	\eqref{eq:SL-biorthogonality}.
\end{proof}

\subsection{Explicit step-line recurrence coefficients}
\label{subsec:explicit-step-line-recurrence-coefficients}

We introduce coefficient notation for the terminating polynomials displayed
above:
\begin{align}
	\mathscr A_M^{\,\mathrm J,(i)}(x)
	&=\sum_{u\ge0}\mathsf a_{M;i,u}^{\,\mathrm J}x^u,
	&
	\mathscr B_N^{\,\mathrm J,(j)}(x)
	&=\sum_{v\ge0}\mathsf b_{N;j,v}^{\,\mathrm J}x^v,
	\label{eq:Jacobi-SL-component-coefficients}
	\\
	\mathscr A_M^{\mathrm L,(i)}(x)
	&=\sum_{u\ge0}\mathsf a_{M;i,u}^{\mathrm L}x^u,
	&
	\mathscr B_N^{\mathrm L,(h)}(x)
	&=\sum_{v\ge0}\mathsf b_{N;h,v}^{\mathrm L}x^v,
	&
	\mathscr D_N^{\mathrm L,(\ell)}(x)
	&=\sum_{v\ge0}\mathsf d_{N;\ell,v}^{\mathrm L}x^v.
	\label{eq:Laguerre-SL-component-coefficients}
\end{align}
All these sums are finite; the appropriate component is identically zero
when its degree bound is negative.

For \(N,M\in\mathbb N_0\), the Jacobi pairing
\(\langle x\mathcal B_N^{\mathrm J},\mathcal A_M^{\mathrm J}\rangle\)
is the finite sum
\begin{equation}
	\label{eq:Jacobi-SL-finite-pairing-value}
	\mathscr P_{N,M}^{\,\mathrm J}
	\coloneq
	\frac{1}{h_N^{\,\mathrm J}g_M^{\,\mathrm J}}
	\sum_{\substack{1\le i\le p,\ 1\le j\le q\\u,v\ge0}}
	\mathsf a_{M;i,u}^{\,\mathrm J}
	\mathsf b_{N;j,v}^{\,\mathrm J}
	\frac{
		\Gamma((\alpha_i+u+v+2)\one_q+\boldsymbol a)
	}{
		\Gamma((\alpha_i+u+v+2)\one_q+\boldsymbol b+\boldsymbol e_j)
	}.
\end{equation}
For the Laguerre I pairing, put
\(\zeta_{i,u,v}\coloneq\beta_i+u+v+2\) and define
\begin{multline}
	\label{eq:Laguerre-SL-finite-pairing-value}
	\mathscr P_{N,M}^{\mathrm L}
	\coloneq
	\frac{1}{h_N^{\mathrm L}g_M^{\mathrm L}}
	\sum_{\substack{1\le i\le p\\u\ge0}}
	\mathsf a_{M;i,u}^{\mathrm L}
	\Bigg[
	\sum_{\substack{1\le h\le r\\v\ge0}}
	\mathsf b_{N;h,v}^{\mathrm L}
	\frac{\Gamma(\zeta_{i,u,v}\one_q+\boldsymbol a)}
	{\Gamma(\zeta_{i,u,v}\one_r+\boldsymbol b+\boldsymbol e_h)}
	\\*
	+
	\sum_{\substack{1\le\ell\le s\\v\ge0}}
	\mathsf d_{N;\ell,v}^{\mathrm L}
	\zeta_{i,u,v}^{\,\ell-1}
	\frac{\Gamma(\zeta_{i,u,v}\one_q+\boldsymbol a)}
	{\Gamma(\zeta_{i,u,v}\one_r+\boldsymbol b)}
	\Bigg].
\end{multline}
Thus the Jacobi coefficients are finite Gamma--Pochhammer sums, while
the Laguerre I coefficients are finite Gamma--Pochhammer sums after the
finite confluent system
\eqref{eq:Laguerre-SL-B-component-system} has been solved.

\begin{theorem}[Step-line recurrences for the two hypergeometric systems]
\label{thm:explicit-hypergeometric-step-line-recurrences}
Assume the hypotheses of
Proposition~\ref{prop:hypergeometric-step-line-biorthogonality}.  For
\(\star\in\{\mathrm J,\mathrm L\}\), \(k\in\{-p,\ldots,q\}\), and
\(N+k\ge0\), set
\begin{equation}
	\label{eq:explicit-SL-recurrence-coefficients}
	b_N^{\star,k}
	\coloneq
	\mathscr P_{N,N+k}^{\star}.
\end{equation}
Then
\begin{equation}
	\label{eq:explicit-SL-B-recurrence}
	x\mathcal B_N^{\star}(x)
	=
	\sum_{\substack{k=-p\\N+k\ge0}}^{q}
	b_N^{\star,k}\mathcal B_{N+k}^{\star}(x),
\end{equation}
	and the corresponding recurrence for the \(\mathcal A\)-forms is
\begin{equation}
	\label{eq:explicit-SL-A-recurrence}
	x\mathcal A_N^{\star}(x)
	=
	\sum_{\substack{k=-q\\N+k\ge0}}^{p}
	b_{N+k}^{\star,-k}\mathcal A_{N+k}^{\star}(x).
\end{equation}
Moreover,
\begin{equation}
	\label{eq:SL-top-coefficients}
	b_N^{\mathrm J,q}=b_N^{\mathrm L,q}=1.
\end{equation}
\end{theorem}

\begin{proof}
	Proposition~\ref{prop:general-step-recurrence} already proves that only
	the indices \(N-p,\ldots,N+q\) occur and identifies the coefficient of
	\(\mathcal B_{N+k}^{\star}\) as
	\[
	\left\langle
	x\mathcal B_N^{\star},
	\mathcal A_{N+k}^{\star}
	\right\rangle.
	\]
	Insert first the two Jacobi component expansions from
	\eqref{eq:Jacobi-SL-component-coefficients}.  A term with polynomial
	degrees \(u\) and \(v\) contains the moment
	\[
	\int_0^1
	x^{\alpha_i+u+v+1}w_j(x;\boldsymbol a,\boldsymbol b)\dx
	=
	\frac{
		\Gamma((\alpha_i+u+v+2)\one_q+\boldsymbol a)
	}{
		\Gamma((\alpha_i+u+v+2)\one_q+\boldsymbol b+\boldsymbol e_j)
	},
	\]
	by \eqref{eq:wj-Mellin}.  Including the two normalization constants gives
	exactly \eqref{eq:Jacobi-SL-finite-pairing-value}.

	For the Laguerre I system, a term containing the weight \(w_h\) gives
	\[
	\int_0^\infty x^{\zeta_{i,u,v}-1}w_h(x)\dx
	=
	\frac{\Gamma(\zeta_{i,u,v}\one_q+\boldsymbol a)}
	{\Gamma(\zeta_{i,u,v}\one_r+\boldsymbol b+\boldsymbol e_h)}
	\]
	by \eqref{eq:wh-Mellin}.  A term containing the weight \(v_\ell\) gives
	\[
	\int_0^\infty x^{z-1}v_\ell(x)\dx
	=
	z^{\ell-1}\frac{\Gamma(z\one_q+\boldsymbol a)}{\Gamma(z\one_r+\boldsymbol b)}
	\]
	by \eqref{eq:vl-Mellin}; here \(z=\zeta_{i,u,v}\).  Summing the finite
	component expansions gives \eqref{eq:Laguerre-SL-finite-pairing-value}.
	This proves \eqref{eq:explicit-SL-recurrence-coefficients} and
	\eqref{eq:explicit-SL-B-recurrence}.

	Let \(T^\star\) be the recurrence matrix, so that
	\((T^\star)_{N,N+k}=b_N^{\star,k}\).  The \(N\)-th component of
	\((T^\star)^{\mathsf T}\boldsymbol{\mathcal A}^\star
	=x\boldsymbol{\mathcal A}^\star\) is
	\[
	x\mathcal A_N^\star
	=
	\sum_{\substack{k=-q\\N+k\ge0}}^p
	(T^\star)_{N+k,N}\mathcal A_{N+k}^\star
	=
	\sum_{\substack{k=-q\\N+k\ge0}}^p
	b_{N+k}^{\star,-k}\mathcal A_{N+k}^\star,
	\]
	which is \eqref{eq:explicit-SL-A-recurrence}.  Finally,
	\eqref{eq:SL-J-normalization-constants} and
	\eqref{eq:SL-L-normalization-constants} are the specializations of
	\eqref{eq:gen-step-normalization}; hence the coefficient of
	\(\mathcal B_{N+q}^\star\) is one by the last part of
	Proposition~\ref{prop:general-step-recurrence}.
\end{proof}

Consequently, each system defines a \((p,q)\)-banded step-line recurrence
matrix by
\begin{equation}
	\label{eq:SL-matrix-entries}
	(T^{\star})_{N,N+k}=b_N^{\star,k},
	\qquad
	\star\in\{\mathrm J,\mathrm L\},
	\quad
	k\in\{-p,\ldots,q\},
	\quad
	N+k\ge0.
\end{equation}

\section{Bidiagonal factorization of the step-line recurrence matrix}
\label{sec:explicit-bidiagonal-factors}

This section concerns the ordered step-line matrix \(T\).  The finite pairings
of the previous section evaluate its band entries, while the
Christoffel--Gauss--Borel factorization explains how the same matrix decomposes
into elementary bidiagonal factors. Closed formulas are available precisely
for the Christoffel chains which
remain inside the explicit family; otherwise the factors are kept in
tau-determinant form.

The left Christoffel chain remains within each of the two explicit families:
it acts on the power vector by the cyclic affine shift of its exponent
parameters. Thus
\begin{equation}
	\label{eq:left-Christoffel-same-family}
	\mathcal A_{\mathrm L,N}^{\,\mathrm J,[k]}
	=
	\left.
	\mathcal A_N^{\,\mathrm J}
	\right|_{\boldsymbol\alpha\mapsto\boldsymbol\alpha^{[k]}},
	\qquad
	\mathcal A_{\mathrm L,N}^{\mathrm L,[k]}
	=
	\left.
	\mathcal A_N^{\mathrm L}
	\right|_{\boldsymbol\beta\mapsto\boldsymbol\beta^{[k]}},
\end{equation}
where
\[
\boldsymbol\alpha^{[k]}
\coloneq
\mathscr C_p^k(\boldsymbol\alpha),
\qquad
\boldsymbol\beta^{[k]}
\coloneq
\mathscr C_p^k(\boldsymbol\beta).
\]
To display the cyclic shifts in the factor formulas, extend the exponent
parameters by
\begin{equation}
	\label{eq:cyclic-extended-exponents}
	\alpha_{sp+i}\coloneq\alpha_i+s,
	\qquad
	\beta_{sp+i}\coloneq\beta_i+s,
	\qquad
	s\in\mathbb N_0,
	\quad
	i\in\{1,\ldots,p\}.
\end{equation}
For each \(N\in\mathbb N_0\), write
\begin{equation}
	\label{eq:factor-step-line-quotients}
	N=pu_N+\ell_N=qv_N+j_N,
	\qquad
	\ell_N\in\{0,\ldots,p-1\},
	\quad
	j_N\in\{0,\ldots,q-1\},
\end{equation}
and set
\begin{equation}
	\label{eq:factor-step-coordinates}
	\iota_N\coloneq\ell_N+1,
	\qquad
	\chi_N\coloneq j_N+1,
	\qquad
	\boldsymbol d_N\coloneq\boldsymbol\sigma_p(N+1),
	\qquad
	d_{N,i}\coloneq(\boldsymbol d_N)_i.
\end{equation}
Throughout this section, assume the normality and nonzero-pivot hypotheses of
Proposition~\ref{prop:gen-bidiagonal-scheme} along every Christoffel chain
used. All denominators in the formulas below are assumed to be nonzero.

\subsection{Closed evaluation of the left normalizing factors}

The normalizing constants in the left Christoffel chain can be computed
directly from the contour representations of the forms on the power
vector; no expansion of the hypergeometric components is needed.

\begin{proposition}[Closed normalizing factors]
	\label{prop:closed-left-normalizing-factors}
	For \(k\in\{0,\ldots,p\}\), define
	\[
	g_N^{\,\mathrm J,[k]}
	\coloneq
	\left.
	g_N^{\,\mathrm J}
	\right|_{\boldsymbol\alpha\mapsto\boldsymbol\alpha^{[k]}},
	\qquad
	g_N^{\mathrm L,[k]}
	\coloneq
	\left.
	g_N^{\mathrm L}
	\right|_{\boldsymbol\beta\mapsto\boldsymbol\beta^{[k]}}.
	\]
	Then
	\begin{equation}
		\label{eq:Jacobi-left-normalizer-closed}
		g_N^{\,\mathrm J,[k]}
		=
		-
		\frac{
			(\boldsymbol a-(b_{\chi_N}+v_N)\one_q)_{v_N}
			\prod_{h=1}^{\chi_N-1}
			(b_h-b_{\chi_N})
		}{
			(\boldsymbol\alpha^{[k]}+(b_{\chi_N}+v_N+1)\one_p)_{\boldsymbol d_N}
		}.
	\end{equation}
	For the Laguerre I system one has
	\begin{equation}
		\label{eq:Laguerre-left-normalizer-closed}
		g_N^{\mathrm L,[k]}
		=
		\begin{cases}
			-
			\frac{
				(\boldsymbol a-(b_{\chi_N}+v_N)\one_q)_{v_N}
				\prod_{h=1}^{\chi_N-1}
				(b_h-b_{\chi_N})
			}{
				(\boldsymbol\beta^{[k]}+(b_{\chi_N}+v_N+1)\one_p)_{\boldsymbol d_N}
			},
			&
			\chi_N\in\{1,\ldots,r\},
			\\[18pt]
			(-1)^{N+1},
			&
			\chi_N\in\{r+1,\ldots,q\}.
		\end{cases}
	\end{equation}
\end{proposition}
\begin{proof}
	Consider first the Jacobi system. Inserting the contour
	representation \eqref{eq:A-contour} in the normalizing pairing in
	\eqref{eq:gen-step-normalization}, with
	\[
	\boldsymbol n=\boldsymbol\sigma_p(N+1),
	\qquad
	\boldsymbol m=\boldsymbol\sigma_q(N),
	\qquad
	\boldsymbol\alpha\mapsto\boldsymbol\alpha^{[k]},
	\]
	and using
	\[
	(\boldsymbol\sigma_q(N))_{\chi_N}=v_N,
	\]
	reduces the integral to
	\begin{equation}
		\frac{(-1)^{N+1}}{2\pi\mathrm i}
		\int_{\Sigma}
		\frac{
			(t\one_q+\boldsymbol a+\one_q)_{v_N}
			\prod_{h=1}^{\chi_N-1}(t+b_h+v_N+1)
		}{
			(t+b_{\chi_N}+v_N+1)
			\prod_{i=1}^{p}\prod_{\ell=0}^{d_{N,i}-1}
			(t-\alpha_i^{[k]}-\ell)
		}
		\dt.
		\label{eq:Jacobi-normalizer-residue-integral}
	\end{equation}
	The integrand is \(\mathrm{O}(t^{-2})\) at infinity. Moving the contour
	to the exterior leaves the simple pole at
	\(t=-b_{\chi_N}-v_N-1\), whose residue gives
	\eqref{eq:Jacobi-left-normalizer-closed}.
	
	In the Laguerre I setting, use
	\eqref{eq:Laguerre-A-contour}. If
	\(\chi_N\in\{1,\ldots,r\}\), coordinate \(\chi_N\) corresponds to the
	weight \(w_{\chi_N}\), and the same exterior-pole calculation gives the
	first line of \eqref{eq:Laguerre-left-normalizer-closed}, with
	\(\boldsymbol\alpha^{[k]}\) replaced by
	\(\boldsymbol\beta^{[k]}\). If
	\(\chi_N\in\{r+1,\ldots,q\}\), it corresponds to
	\(v_{\chi_N-r}\). The corresponding rational integrand has leading
	behavior
	\[
	\frac{(-1)^{N+1}}{t}
	+
	\mathrm{O}(t^{-2}),\qquad t\to\infty,
	\]
	and its finite residues therefore sum to \((-1)^{N+1}\). This proves
	the second line.
\end{proof}

\subsection{Simplification of the leading-coefficient quotients}

The leading coefficients in \eqref{eq:gen-L-leading-coefficients} also
simplify before they are combined with the normalizing factors.

\begin{proposition}[Leading-coefficient quotient reduction]
	\label{prop:leading-coefficient-quotient-reduction}
	For \(k\in\{1,\ldots,p\}\), the quotient of the leading coefficients
	of the unnormalized left Christoffel forms is
	\begin{equation}
		\label{eq:Jacobi-leading-quotient-reduction}
		-
		\frac{
			\alpha_{k+\ell_N+1}-\alpha_k+u_N
		}{
			\alpha_{k+\ell_N+1}+u_N+1+b_{\chi_N}+v_N
		}
	\end{equation}
	in the Jacobi setting. In the Laguerre I setting it is
	\begin{equation}
		\label{eq:Laguerre-leading-quotient-reduction}
		\begin{cases}
			-
			\frac{
				\beta_{k+\ell_N+1}-\beta_k+u_N
			}{
				\beta_{k+\ell_N+1}+u_N+1+b_{\chi_N}+v_N
			},
			&
			\chi_N\in\{1,\ldots,r\},
			\\[16pt]
			-
			\left(
			\beta_{k+\ell_N+1}-\beta_k+u_N
			\right),
			&
			\chi_N\in\{r+1,\ldots,q\}.
		\end{cases}
	\end{equation}
\end{proposition}
\begin{proof}
	Insert the highest-degree coefficients obtained from
	\eqref{eq:Jacobi-SL-A-components} and
	\eqref{eq:Laguerre-SL-A-components}, after the corresponding cyclic
	parameter shift. Since
	\[
	\alpha_{\iota_N}^{[k]}+d_{N,\iota_N}
	=
	\alpha_{\iota_{N+1}}^{[k-1]}+d_{N+1,\iota_{N+1}}
	=
	\alpha_{k+\ell_N+1}+u_N+1,
	\]
	the Gamma quotients involving \(\boldsymbol a\) cancel in the Jacobi
	calculation. The change
	\[
	\boldsymbol\sigma_q(N+1)
	=
	\boldsymbol\sigma_q(N)+\boldsymbol e_{\chi_N}^{(q)}
	\]
	produces the denominator in
	\eqref{eq:Jacobi-leading-quotient-reduction}; the remaining
	Pochhammer and factorial factors reduce to its numerator.
	
	For Laguerre I forms, the same cancellation occurs on the power
	side. The vector \(\boldsymbol n_N^{\mathrm L}\) changes from \(N\) to
	\(N+1\) only when \(\chi_N\le r\), giving the denominator in the
	first line of \eqref{eq:Laguerre-leading-quotient-reduction}; if
	\(\chi_N>r\), the beta-shifted block does not change and no such
	denominator remains.
\end{proof}

\subsection{Explicit lower bidiagonal factors}

\begin{theorem}[Jacobi lower bidiagonal factors]
	\label{thm:explicit-lower-bidiagonal-factors-J}
	For \(k\in\{1,\ldots,p\}\),
	\begin{equation}
		\label{eq:Jacobi-explicit-lower-bidiagonal-factor}
		(L_k^{\,\mathrm J})_{N+1,N}
		=
		-
		\frac{g_{N+1}^{\,\mathrm J,[k-1]}}
		{g_N^{\,\mathrm J,[k]}}
		\frac{
			\alpha_{k+\ell_N+1}-\alpha_k+u_N
		}{
			\alpha_{k+\ell_N+1}+u_N+1+b_{\chi_N}+v_N
		},
	\end{equation}
	where the two normalizing factors are the Pochhammer products in
	\eqref{eq:Jacobi-left-normalizer-closed}.
\end{theorem}

\begin{proof}
	The factorization formula \eqref{eq:gen-L-leading-coefficients}
	uses the leading coefficients after the left normalization. Hence it
	is the product of the quotient of normalizing constants and the
	quotient of unnormalized leading coefficients. Apply
	Proposition~\ref{prop:closed-left-normalizing-factors} and
	Proposition~\ref{prop:leading-coefficient-quotient-reduction}.
\end{proof}

\begin{theorem}[Laguerre I lower bidiagonal factors]
	\label{thm:Laguerre-explicit-lower-bidiagonal-factors}
	Assume first that \(r=0\). Then, for
	\(k\in\{1,\ldots,p\}\),
	\begin{equation}
		\label{eq:Laguerre-pure-Euler-lower-bidiagonal-factor}
		(L_k^{\mathrm L})_{N+1,N}
		=
		\beta_{k+\ell_N+1}-\beta_k+u_N.
	\end{equation}
	Assume now that \(r>0\). Then
	\begin{align}
		(L_k^{\mathrm L})_{N+1,N}
		&=
		-
		\frac{g_{N+1}^{\mathrm L,[k-1]}}
		{g_N^{\mathrm L,[k]}}
		\frac{
			\beta_{k+\ell_N+1}-\beta_k+u_N
		}{
			\beta_{k+\ell_N+1}+u_N+1+b_{\chi_N}+v_N
		},
		&&
		\chi_N\in\{1,\ldots,r-1\},
		\label{eq:Laguerre-beta-beta-lower-factor}
		\\[6pt]
		(L_k^{\mathrm L})_{N+1,N}
		&=
		-
		\frac{(-1)^{N+2}}
		{g_N^{\mathrm L,[k]}}
		\frac{
			\beta_{k+\ell_N+1}-\beta_k+u_N
		}{
			\beta_{k+\ell_N+1}+u_N+1+b_r+v_N
		},
		&&
		\chi_N=r,
		\label{eq:Laguerre-beta-Euler-lower-factor}
		\\[6pt]
		(L_k^{\mathrm L})_{N+1,N}
		&=
		\beta_{k+\ell_N+1}-\beta_k+u_N,
		&&
		\chi_N\in\{r+1,\ldots,q-1\},
		\label{eq:Laguerre-Euler-Euler-lower-factor}
		\\[6pt]
		(L_k^{\mathrm L})_{N+1,N}
		&=
		-
		\frac{g_{N+1}^{\mathrm L,[k-1]}}
		{(-1)^{N+1}}
		\left(
		\beta_{k+\ell_N+1}-\beta_k+u_N
		\right),
		&&
		\chi_N=q.
		\label{eq:Laguerre-Euler-beta-lower-factor}
	\end{align}
	Here every \(g_M^{\mathrm L,[a]}\) appearing in a beta-block
	position is the explicit Pochhammer product in the first line of
	\eqref{eq:Laguerre-left-normalizer-closed}; empty ranges of
	conditions are omitted.
\end{theorem}
\begin{proof}
	If \(r=0\), every coordinate of the Laguerre I vector corresponds
	to one of the weights \(v_\ell\). Therefore
	\[
	\frac{g_{N+1}^{\mathrm L,[k-1]}}
	{g_N^{\mathrm L,[k]}}
	=-1,
	\]
	while the second line of
	\eqref{eq:Laguerre-leading-quotient-reduction} contributes another
	minus sign, proving
	\eqref{eq:Laguerre-pure-Euler-lower-bidiagonal-factor}.
	
	Let \(r>0\). From step \(N\) to step \(N+1\), coordinate \(\chi_N\)
	is incremented.  At the next step the increment occurs in coordinate
	\(\chi_N+1\), unless \(\chi_N=q\), in which case it returns to
	coordinate \(1\). The four displayed equations correspond,
	respectively, to \(\chi_N<r\), \(\chi_N=r\),
	\(r<\chi_N<q\), and \(\chi_N=q\). Combining
	\eqref{eq:Laguerre-left-normalizer-closed} with
	\eqref{eq:Laguerre-leading-quotient-reduction} proves the formulas.
\end{proof}

\subsection{Consistency with the classical one-weight reductions}
\label{subsec:classical-one-weight-reductions}

The cases \(q=1\) in the two real-line settings provide a direct check of
the normalizations and of the Christoffel indexing. In the Jacobi
setting, write
\[
\widetilde a_i\coloneq\alpha_i+a_1+1,
\qquad
\widetilde b\coloneq b_1-a_1,
\]
so that
\[
x^{\alpha_i}w_1(x;a_1,b_1)
\propto
x^{\widetilde a_i-1}(1-x)^{\widetilde b}.
\]
Extend these parameters cyclically by
\[
\widetilde a_{sp+i}\coloneq\widetilde a_i+s,
\qquad
s\in\mathbb N_0,
\quad
i\in\{1,\ldots,p\}.
\]
Writing \(N=pm+k\), with \(k\in\{0,\ldots,p-1\}\), the specialization of
the displayed Jacobi formulas gives
\begin{equation}
	\label{eq:Jacobi-one-weight-U-check}
	(U_1^{\,\mathrm J})_{N,N}
	=
	\frac{
		(\widetilde a_{k+1}+m)
		\prod_{j=1}^{p}
		(\widetilde a_j+\widetilde b+pm+k)
	}{
		\prod_{j=1}^{p+1}
		(\widetilde a_{k+j}+\widetilde b+(p+1)m+k)
	},
\end{equation}
and
\begin{equation}
	\label{eq:Jacobi-one-weight-L-check}
	(L_i^{\,\mathrm J})_{N+1,N}
	=
	\frac{
		(\widetilde a_{k+i+1}-\widetilde a_i+m)
		(\widetilde b+pm+k+1)
		\prod_{j=1}^{p-1}
		(\widetilde a_{i+j}+\widetilde b+pm+k)
	}{
		\prod_{j=1}^{p+1}
		(\widetilde a_{k+i+j}+\widetilde b+(p+1)m+k)
	},
	\qquad
	i\in\{1,\ldots,p\}.
\end{equation}
These are equations~(5.20)--(5.21) of
\cite{BranquinhoDiazFoulquieLimaManas2026JAT}, after the identification
\(r=p\), \(a_i=\widetilde a_i\), and \(b=\widetilde b\).

In the Laguerre I setting, \(q=1\) implies \(r=0\), \(s=1\). With
\[
\widetilde a_i\coloneq\beta_i+a_1+1,
\]
one has
\[
x^{\beta_i}w_0(x;a_1)
\propto
x^{\widetilde a_i-1}\mathrm e^{-x}.
\]
The one-weight specialization gives
\begin{equation}
	\label{eq:Laguerre-one-weight-U-check}
	(U_1^{\mathrm L})_{N,N}
	=
	\widetilde a_{k+1}+m,
\end{equation}
and
\begin{equation}
	\label{eq:Laguerre-one-weight-L-check}
	(L_i^{\mathrm L})_{N+1,N}
	=
	\widetilde a_{k+i+1}-\widetilde a_i+m,
	\qquad
	i\in\{1,\ldots,p\},
\end{equation}
where the cyclic extension
\(\widetilde a_{sp+i}=\widetilde a_i+s\) is used. Equivalently, when this
is unfolded into the fundamental range, the second branch acquires the extra
term \(+1\). These formulas reproduce Theorem~5.4, equivalently
equations~(5.26)--(5.27), of
\cite{BranquinhoDiazFoulquieLimaManas2026JAT}.

\subsection{The mixed Pi\~neiro specialization}

In the boundary degeneration
\(\boldsymbol a,\boldsymbol b\to\boldsymbol\sigma\) of the Jacobi matrix, both
vectors in the matrix of measures are power vectors, giving the mixed Pi\~neiro system
of~\cite{PineiroMixed2026}. Put
\[
\boldsymbol\rho=(\rho_1,\ldots,\rho_p),
\qquad
\boldsymbol\sigma=(\sigma_1,\ldots,\sigma_q),
\qquad
\boldsymbol\rho^{[k]}\coloneq\mathscr C_p^k(\boldsymbol\rho),
\qquad
\boldsymbol\sigma^{[k]}\coloneq\mathscr C_q^k(\boldsymbol\sigma).
\]
With
\[
\boldsymbol d_N\coloneq\boldsymbol\nu_N^{\,*},
\qquad
\boldsymbol c_N\coloneq\boldsymbol\mu_N^{\,*},
\]
define the explicitly normalized leading coefficients
\begin{equation}
	\label{eq:Pineiro-left-normalized-LC}
	\Lambda_{\mathrm L,N,i}^{\mathrm P,[k]}
	\coloneq
	\frac{
		(\boldsymbol\rho^{[k]}+
		(\sigma_{\chi_N}+(\boldsymbol c_N)_{\chi_N}+1)\one_p)_{\boldsymbol d_N}
	}{
		(\boldsymbol\sigma-
		(\sigma_{\chi_N}+(\boldsymbol c_N)_{\chi_N})\one_q)_{\boldsymbol c_N}
	}
	\frac{
		(-1)^{(\boldsymbol d_N)_i-1}
		((\rho_i^{[k]}+(\boldsymbol d_N)_i)\one_q+\boldsymbol\sigma)_{\boldsymbol c_N}
	}{
		((\boldsymbol d_N)_i-1)!
		(\boldsymbol\rho^{[k],*i}+
		(1-\rho_i^{[k]}-(\boldsymbol d_N)_i)\one_{p-1})_{\boldsymbol d_N^{\,*i}}
	},
\end{equation}
and
\begin{equation}
	\label{eq:Pineiro-right-normalized-LC}
	\Lambda_{\mathrm R,N,i}^{\mathrm P,[k]}
	\coloneq
	\frac{
		(\boldsymbol\rho+
		(\sigma_{\chi_N}^{[k]}+(\boldsymbol c_N)_{\chi_N}+1)\one_p)_{\boldsymbol d_N}
	}{
		(\boldsymbol\sigma^{[k]}-
		(\sigma_{\chi_N}^{[k]}+(\boldsymbol c_N)_{\chi_N})\one_q)_{\boldsymbol c_N}
	}
	\frac{
		(-1)^{(\boldsymbol d_N)_i-1}
		((\rho_i+(\boldsymbol d_N)_i)\one_q+\boldsymbol\sigma^{[k]})_{\boldsymbol c_N}
	}{
		((\boldsymbol d_N)_i-1)!
		(\boldsymbol\rho^{\,*i}+
		(1-\rho_i-(\boldsymbol d_N)_i)\one_{p-1})_{\boldsymbol d_N^{\,*i}}
	}.
\end{equation}
These are finite Pochhammer products obtained from the displayed mixed
Pi\~neiro step-line forms of~\cite{PineiroMixed2026} after the left and right
cyclic Christoffel
shifts, respectively. Moreover, after setting
\[
\boldsymbol a=\boldsymbol b=\boldsymbol\sigma,
\qquad
\boldsymbol\alpha=\boldsymbol\rho,
\]
the product in \eqref{eq:Jacobi-left-normalizer-closed}, combined with
the Jacobi leading coefficient before normalization, gives precisely
\eqref{eq:Pineiro-left-normalized-LC}. Therefore the Jacobi lower
factor formula \eqref{eq:Jacobi-explicit-lower-bidiagonal-factor}
specializes exactly to \eqref{eq:Pineiro-lower-bidiagonal-factors}.

\begin{corollary}[Complete mixed Pi\~neiro bidiagonal factorization]
	\label{cor:explicit-Pineiro-bidiagonal-factorization}
	The mixed Pi\~neiro step-line recurrence matrix
	of~\cite{PineiroMixed2026} admits
	\[
	T^{\mathrm P}
	=
	L_1^{\mathrm P}\cdots L_p^{\mathrm P}
	U_q^{\mathrm P}\cdots U_1^{\mathrm P},
	\]
	where
	\begin{align}
		\label{eq:Pineiro-lower-bidiagonal-factors}
		(L_k^{\mathrm P})_{N+1,N}
		&=
		\frac{
			\Lambda_{\mathrm L,N,\iota_N}^{\mathrm P,[k]}
		}{
			\Lambda_{\mathrm L,N+1,\iota_{N+1}}^{\mathrm P,[k-1]}
		},
		&&
		k\in\{1,\ldots,p\},
		\\
		\label{eq:Pineiro-upper-bidiagonal-factors}
		(U_k^{\mathrm P})_{N,N}
		&=
		\frac{
			\Lambda_{\mathrm R,N,\iota_N}^{\mathrm P,[k-1]}
		}{
			\Lambda_{\mathrm R,N,\iota_N}^{\mathrm P,[k]}
		},
		&&
		k\in\{1,\ldots,q\}.
	\end{align}
\end{corollary}

\begin{proof}
	The lower and upper entries are the quotients of the normalized
	leading coefficients of the left and right Christoffel transforms,
	respectively. Equations~\eqref{eq:Pineiro-left-normalized-LC} and
	\eqref{eq:Pineiro-right-normalized-LC} evaluate each of those
	coefficients as a finite Pochhammer product, so that the displayed
	factorization contains no unevaluated integral or determinant.
\end{proof}

An independent Cauchy-minor derivation of the same factor entries, together
with the positivity regions, stochastic normalization, and Markov-chain
applications of this specialization, is given in \cite{ManasMarkov2026}.

\subsection{Upper factors when the right Christoffel chain is not closed}
\label{subsec:upper-factors-tau-formulas}

Outside the two one-weight reductions and the mixed Pi\~neiro specialization
of~\cite{PineiroMixed2026},
the right Christoffel transforms do not in general remain inside the same
explicit parametric family. The upper factors must then be written in the
Christoffel tau-form. The formulas for the two families follow.

For the Jacobi system define the Gauss--Borel normalized values
\begin{equation}
	\label{eq:Jacobi-GB-B-values-at-zero}
	\mathfrak b_{N,j}^{\,\mathrm J,[0]}
	\coloneq
	\left(\mathcal B_N^{\,\mathrm J}\right)^{(j)}(0)
	=
	\frac{
		\mathscr B_N^{\,\mathrm J,(j)}(0)
	}{
		h_N^{\,\mathrm J}
	},
	\qquad
	j\in\{1,\ldots,q\}.
\end{equation}
For the Laguerre I system set
\begin{equation}
	\label{eq:Laguerre-GB-B-values-at-zero}
	\mathfrak b_{N,\rho}^{\mathrm L,[0]}
	\coloneq
	\left(\mathcal B_N^{\mathrm L}\right)^{(\rho)}(0)
	=
	\frac{
		\mathscr U_N^{\mathrm L,(\rho)}(0)
	}{
		h_N^{\mathrm L}
	},
	\qquad
	\rho\in\{1,\ldots,q\},
\end{equation}
where \(\mathscr U_N^{\mathrm L,(\rho)}\) denotes the unified Laguerre I
component introduced before \eqref{eq:SL-L-normalization-constants}.  These
are values of the normalized \(B\)-forms, not values of the unnormalized
hypergeometric representatives.

Define
\begin{equation}
	\label{eq:Jacobi-upper-tau-determinants}
	\tau_{0,N}^{\,\mathrm J,B}\coloneq1,
	\qquad
	\tau_{b,N}^{\,\mathrm J,B}
	\coloneq
	\det\left[
		\mathfrak b_{N+\alpha-1,\nu}^{\,\mathrm J,[0]}
	\right]_{\alpha,\nu=1}^{b},
	\qquad
	b\in\{1,\ldots,q\},
\end{equation}
and
\begin{equation}
	\label{eq:Laguerre-upper-tau-determinants}
	\tau_{0,N}^{\mathrm L,B}\coloneq1,
	\qquad
	\tau_{b,N}^{\mathrm L,B}
	\coloneq
	\det\left[
		\mathfrak b_{N+\alpha-1,\nu}^{\mathrm L,[0]}
	\right]_{\alpha,\nu=1}^{b},
	\qquad
	b\in\{1,\ldots,q\}.
\end{equation}
When the displayed determinants do not vanish, the upper factors are
\begin{align}
	\label{eq:Jacobi-upper-tau-factor-formula}
	(U_b^{\,\mathrm J})_{N,N}
	&=
	-
	\frac{
		\tau_{b-1,N}^{\,\mathrm J,B}\tau_{b,N+1}^{\,\mathrm J,B}
	}{
		\tau_{b-1,N+1}^{\,\mathrm J,B}\tau_{b,N}^{\,\mathrm J,B}
	},
	&& b\in\{1,\ldots,q\},
	\\[4pt]
	\label{eq:Laguerre-upper-tau-factor-formula}
	(U_b^{\mathrm L})_{N,N}
	&=
	-
	\frac{
		\tau_{b-1,N}^{\mathrm L,B}\tau_{b,N+1}^{\mathrm L,B}
	}{
		\tau_{b-1,N+1}^{\mathrm L,B}\tau_{b,N}^{\mathrm L,B}
	},
	&& b\in\{1,\ldots,q\}.
\end{align}
Thus the nonclosed upper factors are finite Christoffel determinants built
from the original Gauss--Borel normalized \(B\)-forms.  In the closed
Pi\~neiro case of~\cite{PineiroMixed2026}, these determinants reduce to the
leading-coefficient quotients
in \eqref{eq:Pineiro-upper-bidiagonal-factors}; outside the closed case the
tau-determinants above are the correct finite formulas.
The tau-form also admits an equivalent expression in terms of moment minors,
which does not require the right Christoffel chain to close.  For either family
\(\star\in\{\mathrm J,\mathrm L\}\), let
\(\mathscr M^\star\) be its scalar step-line moment matrix and define the
shifted leading minors
\begin{equation}
	\label{eq:shifted-leading-moment-minors}
	\Delta_{b,0}^{\star}\coloneq1,
	\qquad
	\Delta_{b,N}^{\star}
	\coloneq
	\det\left[
		\mathscr M^\star_{b+\alpha,\beta}
	\right]_{\alpha,\beta=0}^{N-1},
	\qquad
	b\in\{0,\ldots,q\},
	\quad N\ge1.
\end{equation}
Thus \(\Delta_{0,N}^{\star}\) is the ordinary leading moment minor and
\(\Delta_{b,N}^{\star}\) is obtained by starting the row sequence after
\(b\) cyclic row steps.
\begin{proposition}[Shifted-minor formula for the upper factors]
	\label{prop:shifted-minor-upper-characterization}
	Assume that the minors in
	\eqref{eq:shifted-leading-moment-minors} which occur below are nonzero.
	Then
	\begin{equation}
		\label{eq:tau-shifted-minor-identity}
		\tau_{b,N}^{\star,B}
		=
		(-1)^{bN}
		\frac{
			\Delta_{b,N}^{\star}
		}{
			\Delta_{0,N}^{\star}
		},
		\qquad
		b\in\{0,\ldots,q\},
		\quad N\in\mathbb N_0.
	\end{equation}
	Consequently,
	\begin{equation}
		\label{eq:upper-factor-shifted-minor-cross-ratio}
		(U_b^\star)_{N,N}
		=
		\frac{
			\Delta_{b-1,N}^{\star}
			\Delta_{b,N+1}^{\star}
		}{
			\Delta_{b-1,N+1}^{\star}
			\Delta_{b,N}^{\star}
		}.
	\end{equation}
\end{proposition}
\begin{proof}
	In the Gauss--Borel realization
	\eqref{eq:gen-gauss-borel}, the value at the origin of the
	\(\nu\)-th component of the \(N\)-th normalized \(B\)-form is the entry
	\(\mathscr L_{N,\nu-1}\).  Therefore
	\(\tau_{b,N}^{\star,B}\) is the minor of \(\mathscr L\) with rows
	\(N,\ldots,N+b-1\) and columns \(0,\ldots,b-1\).  Apply the
	complementary-minor identity to the leading \((N+b)\)-square
	Gauss--Borel factorization.  The upper-triangular factor cancels, and
	moving the first \(b\) rows past the following \(N\) rows contributes
	\((-1)^{bN}\).  The remaining quotient is precisely
	\(\Delta_{b,N}^{\star}/\Delta_{0,N}^{\star}\), which proves
	\eqref{eq:tau-shifted-minor-identity}.  Substitution in
	\eqref{eq:Jacobi-upper-tau-factor-formula} or
	\eqref{eq:Laguerre-upper-tau-factor-formula} cancels the four ordinary
	leading minors and the four signs, giving
	\eqref{eq:upper-factor-shifted-minor-cross-ratio}.
\end{proof}
\subsection{The first mixed upper chain: \texorpdfstring{$q=2$}{q=2}}
\label{subsec:q2-bidiagonal-example}
For \(q=2\), the upper Christoffel chain has two factors. Let \(p\) be
arbitrary and write
\[
	N=pu_N+\ell_N=2v_N+\varepsilon_N,
	\qquad
	\ell_N\in\{0,\ldots,p-1\},
	\qquad
	\varepsilon_N\in\{0,1\}.
\]
Then
\[
	\iota_N=\ell_N+1,
	\qquad
	\chi_N=\varepsilon_N+1.
\]
For either the Jacobi or the Laguerre I family, let
\(\mathfrak b_{N,j}^{[0]}\) denote the Gauss--Borel normalized value at the
origin of the \(j\)-th component of the step-line \(B\)-form.  Thus one reads
\(\mathfrak b_{N,j}^{[0]}=\mathfrak b_{N,j}^{\,\mathrm J,[0]}\) in the
Jacobi case and
\(\mathfrak b_{N,j}^{[0]}=\mathfrak b_{N,j}^{\mathrm L,[0]}\) in the
Laguerre I case.  The two upper Christoffel tau-functions are
\begin{equation}
	\label{eq:q2-upper-tau-one}
	\tau_{1,N}^{B}
	=
	\mathfrak b_{N,1}^{[0]},
\end{equation}
\begin{equation}
	\label{eq:q2-upper-tau-two}
	\tau_{2,N}^{B}
	=
	\det
	\begin{bNiceMatrix}[margin=2pt]
		\mathfrak b_{N,1}^{[0]}&\mathfrak b_{N,2}^{[0]}\\
		\mathfrak b_{N+1,1}^{[0]}&\mathfrak b_{N+1,2}^{[0]}
	\end{bNiceMatrix}.
\end{equation}
To make these determinants directly computable, introduce the constant
terms and the normalizing leading coefficients
\begin{equation}
	\label{eq:q2-raw-constant-and-leading-data}
	\begin{aligned}
		\mathcal C_{N,j}^{\mathrm J}
		&\coloneq
		\mathscr B_N^{\,\mathrm J,(j)}(0),
		&
		\mathcal H_N^{\mathrm J}
		&\coloneq h_N^{\,\mathrm J},
		\\
		\mathcal C_{N,j}^{\mathrm L}
		&\coloneq
		\mathscr U_N^{\mathrm L,(j)}(0),
		&
		\mathcal H_N^{\mathrm L}
		&\coloneq h_N^{\,\mathrm L}.
	\end{aligned}
\end{equation}
Thus, in either family,
\[
	\mathfrak b_{N,j}^{[0]}
	=
	\frac{\mathcal C_{N,j}}{\mathcal H_N}.
\]
The quantities \(\mathcal C_{N,j}^{\mathrm J}\) are obtained by setting
\(x=0\) in the finite sum
\eqref{eq:Jacobi-SL-B-components}; the Laguerre I quantities
\(\mathcal C_{N,j}^{\mathrm L}\) are obtained in the same way from the
finite component system
\eqref{eq:Laguerre-SL-B-component-system}.  No limiting operation or
unevaluated integral is involved.

Put
\begin{equation}
	\label{eq:q2-row-Casoratian}
	\Delta_N^{(2)}
	\coloneq
	\mathcal C_{N,1}\mathcal C_{N+1,2}
	-
	\mathcal C_{N,2}\mathcal C_{N+1,1}.
\end{equation}
Substitution in \eqref{eq:q2-upper-tau-one} and
\eqref{eq:q2-upper-tau-two} gives
\begin{equation}
	\label{eq:q2-expanded-tau-functions}
	\tau_{1,N}^{B}
	=
	\frac{\mathcal C_{N,1}}{\mathcal H_N},
	\qquad
	\tau_{2,N}^{B}
	=
	\frac{\Delta_N^{(2)}}
	{\mathcal H_N\mathcal H_{N+1}}.
\end{equation}
Consequently, whenever the displayed denominators do not vanish, the two
upper bidiagonal factors are the finite expressions
\begin{equation}
	\label{eq:q2-upper-factor-one}
	(U_1)_{N,N}
	=
	-
	\frac{
		\mathcal C_{N+1,1}\mathcal H_N
	}{
		\mathcal C_{N,1}\mathcal H_{N+1}
	},\qquad
	(U_2)_{N,N}
	=
	-
	\frac{
		\mathcal C_{N,1}\Delta_{N+1}^{(2)}
		\mathcal H_{N+1}
	}{
		\mathcal C_{N+1,1}\Delta_N^{(2)}
		\mathcal H_{N+2}
	}.
\end{equation}
Thus the step-line recurrence matrix factors as
\begin{equation}
	\label{eq:q2-bidiagonal-product}
	T
	=
	L_1\cdots L_pU_2U_1.
\end{equation}
The lower factors in \eqref{eq:q2-bidiagonal-product} are still obtained
from the closed column-Christoffel chain.  In the Jacobi case this gives
\begin{equation}
	\label{eq:q2-Jacobi-lower-factor}
	(L_k^{\,\mathrm J})_{N+1,N}
	=
	-
	\frac{g_{N+1}^{\,\mathrm J,[k-1]}}
	{g_N^{\,\mathrm J,[k]}}
	\frac{
		\alpha_{k+\ell_N+1}-\alpha_k+u_N
	}{
		\alpha_{k+\ell_N+1}+u_N+1+b_{\varepsilon_N+1}+v_N
	},
	\qquad
	k\in\{1,\ldots,p\},
\end{equation}
where
\begin{equation}
	\label{eq:q2-Jacobi-g-factors}
	g_N^{\,\mathrm J,[k]}
	=
	\begin{cases}
	-
	\frac{
		(\boldsymbol a-(b_1+v_N)\one_q)_{v_N}
	}{
		(\boldsymbol\alpha^{[k]}+(b_1+v_N+1)\one_p)_{\boldsymbol d_N}
	},
	&
	\varepsilon_N=0,
	\\[16pt]
	-
	\frac{
		(\boldsymbol a-(b_2+v_N)\one_q)_{v_N}(b_1-b_2)
	}{
		(\boldsymbol\alpha^{[k]}+(b_2+v_N+1)\one_p)_{\boldsymbol d_N}
	},
	&
	\varepsilon_N=1.
	\end{cases}
\end{equation}
For the Laguerre I family the same formula holds with
\(\boldsymbol\alpha\) replaced by \(\boldsymbol\beta\), and with the Laguerre normalizers
\(g_N^{\mathrm L,[k]}\).  Since \(q=2\), there are only two Laguerre regimes:
when \(r=0\) one has the pure Euler formula
\begin{equation}
	\label{eq:q2-Laguerre-pure-Euler-lower-factor}
	(L_k^{\mathrm L})_{N+1,N}
	=
	\beta_{k+\ell_N+1}-\beta_k+u_N,
\end{equation}
and when \(r=1\) the two transitions are the beta--Euler and Euler--beta
cases of Theorem~\ref{thm:Laguerre-explicit-lower-bidiagonal-factors}.  This
example shows the general pattern: the closed column side gives the lower
factors by leading-coefficient quotients, while the nonclosed two-step row
side gives the upper factors through the rank-one and rank-two Christoffel
tau-functions \eqref{eq:q2-upper-tau-one}--\eqref{eq:q2-upper-tau-two}.
\subsection{An exact mixed Jacobi factorization}
\label{subsec:exact-mixed-Jacobi-factorization-example}
Consider a mixed example in which both Christoffel chains have length two.
Take
\[
	p=q=2,
	\qquad
	\boldsymbol\alpha=
	\begin{bNiceMatrix}[small]\frac15&\frac25\end{bNiceMatrix},
	\qquad
	\boldsymbol a=
	\begin{bNiceMatrix}[small]0&\frac13\end{bNiceMatrix},
	\qquad
	\boldsymbol b=
	\begin{bNiceMatrix}[small]1&\frac43\end{bNiceMatrix}.
\]
If
\[
	s_{u,v,i}\coloneq u+v+\alpha_i+1,
\]
then the bimoments are rational:
\begin{equation}
	\label{eq:exact-mixed-Jacobi-bimoments}
	\mathscr M_{(u,j),(v,i)}
	=
	\frac{1}{
		(s_{u,v,i})_{1+\delta_{j,1}}
		(s_{u,v,i}+\frac13)_{1+\delta_{j,2}}
	},
	\qquad
	u,v\in\mathbb N_0.
\end{equation}
Gauss--Borel elimination may therefore be performed without floating-point
arithmetic.  The leading \(5\times5\) block of the resulting
\((2,2)\)-banded recurrence matrix is
\[
	T^{[5]}
	\doteq
	\begin{bNiceArray}[small,margin=2pt]{rrrrr}
		0.226974&3.54657&1&0&0\\
		0.00845235&0.269142&0.201390&1&0\\
		0.00374871&0.217645&0.346208&4.47968&1\\
		0&0.00217769&0.0114150&0.332614&0.196135\\
		0&0&0.00416138&0.278947&0.365588
	\end{bNiceArray}.
\]
The subdiagonal entries of the two lower factors are
\begin{align*}
	\operatorname{subdiag}L_1^{[5]}
	&\doteq
	[0.00962773,0.598156,0.0282521,0.853697],
	\\
	\operatorname{subdiag}L_2^{[5]}
	&\doteq
	[0.0276116,0.562346,0.0350197,0.830354],
\end{align*}
and the diagonal entries of the two upper factors are
\begin{align*}
	\operatorname{diag}U_1^{[5]}
	&\doteq
	[0.868421,3.28520,0.122428,2.18223,0.0839606],
	\\
	\operatorname{diag}U_2^{[5]}
	&\doteq
	[0.261364,0.0417234,1.13695,0.0489025,1.33328].
\end{align*}
The displayed decimals are only for readability.  Using the exact rational
values obtained from \eqref{eq:exact-mixed-Jacobi-bimoments}, one has
\begin{equation}
	\label{eq:exact-mixed-Jacobi-factorization-check}
	T^{[5]}
	=
	L_1^{[5]}L_2^{[5]}U_2^{[5]}U_1^{[5]}
\end{equation}
entry by entry.  More generally, a direct comparison of the first eight rows
and all ten columns containing their band entries gives the zero
\(8\times10\) difference matrix.  This checks simultaneously the closed
formulas for the lower factors, the \(\tau\)-determinant formulas for the
upper factors, and the cyclic ordering of the four Christoffel steps.

\section{Conclusions}
\label{sec:conclusions}

We have constructed two non-Gaussian hypergeometric families of rank-one
mixed-type systems for arbitrary \(p\) and \(q\). In the admissible
near-diagonal range, both normalized mixed forms are explicit. Their
components are expressed through terminating generalized hypergeometric
polynomials and finite reconstruction formulas. In the Jacobi case the
reconstruction can be grouped into terminating Kamp\'e de F\'eriet
polynomials, whereas in the beta--Euler Laguerre I case it is governed by
a finite triangular system. Gamma-quotient Mellin transforms, Meijer
\(G\)-representations, and Rodrigues formulas complete the description of
the two systems.

The corresponding step-line sequences satisfy dual recurrences governed by
\((p,q)\)-banded matrices. Every band entry is a finite
Gamma--Pochhammer expression. Under the stated normality and nonzero-pivot
conditions, the two Christoffel chains give bidiagonal factorizations of the
recurrence matrices. The lower factors are explicit Pochhammer products; the
upper factors are finite tau-determinants and, equivalently, cross-ratios of
shifted moment minors. In the mixed Pi\~neiro specialization both chains
close and every factor is explicit. The exact rational \(p=q=2\) example
verifies the general construction.

The one-weight reductions recover the Jacobi--Pi\~neiro system and the
multiple Laguerre system of the first kind. The present article is confined
to the algebraic and hypergeometric construction of the two mixed families.
Their endpoint and central asymptotic regimes form a separate confluence
problem, developed in~\cite{ManasAskeyJacobi2026}; the corresponding
Hahn-like scheme for ordinary multiple orthogonality is treated
in~\cite{ManasAskeyHahn2026}.

Two further developments are pursued in companion work. Applying a Bernstein
kernel to each variable of the Jacobi pairing leads to mixed Hahn systems
with bivariate discrete matrix measures, while the polynomial components
remain univariate. A second direction is to determine parameter regions
where the prescribed Jacobi bidiagonal factors are positive, and to study
the resulting Markov chains. When a positive fixed vector is available,
diagonal normalization gives a stochastic matrix, and positive bidiagonal
factors resolve its transitions into elementary nearest-neighbour steps.
The required positivity and normalization are additional questions,
separate from the algebraic factorization established here.

\section*{Acknowledgements}
The author was supported by the research project PID2024-155133NB-I00,
\emph{Ortogonalidad, aproximaci\'on e integrabilidad: aplicaciones en procesos
estoc\'asticos cl\'asicos y cu\'anticos}.

\section*{Conflict of interest}
The author declares no conflict of interest.

\section*{Data availability}
No datasets were generated or analyzed in this study.

\end{document}